\documentclass[11pt,a4paper]{amsart}

\usepackage[english]{babel}
\usepackage[T1]{fontenc}
\usepackage[utf8]{inputenc}

\usepackage{geometry}
\usepackage{relsize}
\usepackage[nodisplayskipstretch]{setspace}

\usepackage[alphabetic]{amsrefs} % replaces natbib/biblatex
\DefineSimpleKey{bib}{primaryclass}{}
\DefineSimpleKey{bib}{archiveprefix}{}
\newcommand*{\bracketize}[1]{[#1]}
\BibSpec{arxiv}{%
  +{}{\PrintAuthors}{author}
  +{,}{ \textit}{title}
  +{}{ \parenthesize}{date}
  +{,}{ arXiv:}{eprint}
  +{}{ \bracketize}{primaryclass}
}
\usepackage[dvipsnames]{xcolor}
\usepackage{graphicx}
\usepackage{tikz}
\usetikzlibrary{positioning,arrows.meta,cd,babel}
\usepackage[percent]{overpic}
\usepackage[shortlabels]{enumitem}
\setlist[itemize]{topsep=3pt,itemsep=3pt}

\graphicspath{ {figures/} }
\usepackage{amssymb}      % amsmath/thm are already in amsart
\usepackage{mathtools}    % small, safe enhancement to amsmath
\usepackage{mathrsfs}     % \mathscr
\usepackage{dsfont}       % \mathds for indicator 1, etc.
\usepackage{xfrac}        % \sfrac
\usepackage{nccmath}
\definecolor{myRED}{HTML}{E6332A}
\definecolor{myGRAY}{HTML}{8E8E8E}
\definecolor{myMIDDGRAY}{HTML}{E5E5E5}
\definecolor{myMIDGRAY}{HTML}{E2E2E2}
\definecolor{myDARKGRAY}{HTML}{404647}
\definecolor{myLIGHTGRAY}{HTML}{F9F9F9}
\definecolor{myBLUE}{HTML}{03468F}
\definecolor{myPURPLE}{HTML}{8C54D0}
\definecolor{myGREEN}{HTML}{007355}
\definecolor{myYELLOW}{HTML}{FFD300}
\definecolor{myDARKYELLOW}{HTML}{FF9000}

\usepackage{hyperref}
\hypersetup{
  colorlinks=true,
  linkcolor=blue,
  urlcolor=blue,
  citecolor=blue
}

\numberwithin{equation}{section}

\newtheorem{theorem}{Theorem}[section]
\newtheorem{lemma}[theorem]{Lemma}
\newtheorem{proposition}[theorem]{Proposition}
\newtheorem{corollary}[theorem]{Corollary}
\newtheorem{claim}[theorem]{Claim}

\newtheorem*{theorem*}{Theorem}
\newtheorem*{pushinglemma*}{Pushing Lemma}
\newtheorem*{lemmasep*}{Proposition 7.6}
\newtheorem*{definition*}{Definition}

\usepackage{alphalph}
\makeatletter
\newtheorem{thmx}{Theorem}

\makeatother

\makeatletter
\newtheorem{thmxc}{Corollary}

\makeatother

\theoremstyle{definition}

\newtheorem{remark}[theorem]{Remark}

\newcommand{\mycomment}[1]{}
\newcommand\sbullet[1][.75]{\mathbin{\vcenter{\hbox{\scalebox{#1}{$\bullet$}}}}}

\newcommand{\sspc}{\hspace*{0.04cm}}
\newcommand{\spc}{\hspace*{0.08cm}}
\newcommand{\pp}{{\sspc\prime}}

\newcommand{\homeo}{\mathrm{Homeo}}

\newcommand{\mcg}{\mathrm{MCG}}

\renewcommand{\O}{O}
\renewcommand{\P}{\mathcal P}

\newcommand{\OO}{\mathcal O}

\newcommand{\geo}{\mathsmaller{\#}}

\newcommand{\R}{\mathbb R}
\newcommand{\N}{\mathbb N}
\newcommand{\Z}{\mathbb Z}

\newcommand{\G}{\mathcal G}

\newcommand{\F}{\mathcal F}
\newcommand{\B}{\mathcal B}

\newcommand{\class}[1]{\big[\hspace*{0.08cm} #1 \hspace*{0.08cm} \big]}
\newcommand{\bclass}[1]{\big[\hspace*{0.08cm} #1 \hspace*{0.08cm} \big]_{\B}}

\newcommand{\orb}{\mathrm{Orb}(f)}

\title[A foliated viewpoint on homotopy Brouwer theory]{A foliated viewpoint on homotopy Brouwer theory}
\author[N. Schuback]{Nelson Schuback}
\address{Sorbonne Université and Université Paris Cité, CNRS, IMJ-PRG, F-75005 Paris, France}
\email{\href{mailto:nelson.schuback@imj-prg.fr}{nelson.schuback@imj-prg.fr}}
\thanks{This project has received funding from the European Union’s Horizon 2020 research and innovation programme under the Marie Skłodowska-Curie grant agreement No 945332. \includegraphics*[scale = 0.03]{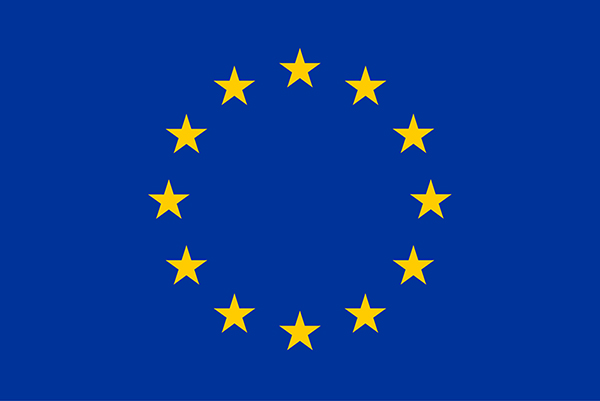}}

\begin{document}

\begin{abstract}
   Brouwer homeomorphisms are fixed-point–free, orientation-preserving homeomorphisms of the plane. In recent years, their dynamics have been mostly studied through two complementary approaches, one introduced by Handel \cite{handel99} and the other by Le Calvez \cite{lec1},\break each offering a distinct perspective on the behavior of Brouwer homeomorphisms. In this work, we present a unified framework that allows Le Calvez's foliated methods to recover, and in some cases improve, the classical results of Handel’s theory.

\end{abstract}

\maketitle


\vspace*{-0cm}

\tableofcontents

\vspace*{-0.4cm}

\setlength{\baselineskip}{1.25\baselineskip}

\section{Introduction}

In his unpublished preprint \cite{Han86}, Handel proposed a description of the dynamics of zero-entropy surface homeomorphisms, claiming that their orbits could be “tracked’’ by a geodesic lamination of the surface \cite[Theorem 2.5]{Han86}. However, questions regarding the completeness of his arguments persist, and to this day there is still no published or widely accepted proof of Handel's claim in its full extent. Despite this, the preprint has been highly influential: many of its ideas and methods have since been developed into independent works. A notable example can be found in \cite{handel99}, where Handel laid the foundations of the so-called Homotopy Brouwer Theory.\break  Since its inception, this theory has shown to be an important tool to study surface homeomorphisms, enabling, for example, the proof of the $C^0$ Arnold conjecture on surfaces \cite{MATSUMOTO2000191}.

As is typical in topological surface dynamics, many of Handel's ideas rely on a profound understanding of the dynamics of Brouwer homeomorphisms, that is, orientation-preserving, fixed-point free, homeomorphisms of the plane. These maps arise naturally in topological surface dynamics, since any (orientation-preserving) surface homeomorphism, when restricted to the complement of its fixed-point set, is lifted to a Brouwer homeomorphism on the universal cover. These maps were named after L. E. J. Brouwer, who first studied them in \cite{Brouwer} as part of his efforts to address Hilbert’s fifth problem. Mistakenly, Brouwer initially believed that any Brouwer homeomorphism $f$ is conjugate to a translation, but he later realized that this was not the case. 
Eventually, his investigations culminated in the \emph{Brouwer Translation Theorem} \cite{Brouwer}, which asserts that through every point of the plane passes a \textit{Brouwer line} $\lambda$ of $f$, which is a properly embedded line $\lambda: \R \longrightarrow \R^2$ dynamically transverse to $f$, in the sense that,
$$f(L(\lambda)) \subset L(\lambda)\quad \text{ and } \quad f^{-1}(R(\lambda)) \subset R(\lambda),$$
where $L(\lambda)$ and $R(\lambda)$ denote the components of \(\R^2\setminus \lambda\) respectively on the left and right of \(\lambda\). 
In particular, this theorem implies that the non-wandering set of $f$ is empty, and that every single orbit $\O(f,x) = \{f^n(x)\}_{n \in \Z}$ is a closed infinite discrete subset of the plane.

 The main purpose of Homotopy Brouwer Theory is to describe the homotopic behavior of a Brouwer homeomorphism $f$ relative to a finite collection of its orbits \(\OO=\{\O_1,\dots,\O_r\}\). It does  so by studying and classifying the \emph{Brouwer mapping class} $\class{f,\sspc \OO}$, which is defined as the isotopy class of $f$ relative to the set $\OO:=\bigcup_{i=1}^r \O_i \subset \R^2$. 
 In \cite{handel99}, Handel introduces a framework, inspired by the ideas of the Nielsen–Thurston machinery, that that enables the classification of all Brouwer mapping classes in the cases $r = 1$ and $r = 2$. This approach was subsequently refined by Le Roux \cite{LEROUX_2017} and later by Bavard \cite{BAVARD_2017}, culminating in classifications of Brouwer mapping classes relative to $r = 3$ and $r = 4$ orbits, respectively. 

 The slow progress in this area can be largely attributed to the intricate and technical nature of Handel's machinery, which heavily relies on the notion of \textit{fitted families} \cite[Section 5]{handel99} and \textit{homotopy translation arcs}---a homotopic analogue of Brouwer's translation arcs. However, 
 recent advances toward a Nielsen–Thurston like classification for surfaces of infinite type \cite{bestvina2023nielsenthurstonclassificationsurfacesinfinite} have reignited interest in the study of big mapping class groups, and in particular, in the subfamily of Brouwer mapping classes within the big mapping class group $\mcg(\R^2 \setminus \Z)$. Therefore, advancements in Homotopy Brouwer theory would likely play an important role in achieving a complete version of a Nielsen–Thurston classification for surfaces of infinite type.

 Since the establishment of Handel's machinery, the framework of Brouwer Theory has been fundamentally reshaped. In his seminal work \cite{lec1}, Le Calvez presented a foliated version of the Brouwer translation theorem, which asserts that there always exists an oriented topological  foliation of the plane \(\F\) whose leaves are Brouwer lines of $f$, called a \textit{transverse foliation} of $f$.\break  This result establishes an entirely new framework to describe the dynamics of Brouwer homeomorphisms by means of trajectories positively transverse $\F$. In its equivariant form \cite{lec2}, this framework serves as the basis for the modern and powerful techniques of Forcing Theory, developed by Le Calvez and Tal \cite{LCTal2018,LeCalvezTal2022}. Among the numerous applications of Forcing Theory, one also finds results in the direction of Handel’s preprint. For example, \cite[Theorem A]{LeCalvezTal2022}, which provides a definitive proof of a key result in the preprint \cite[Theorem 9.1]{Han86}. And more recently, \cite[Corollary G]{garcíasassi2024}, which presents and ergodic theoretical version of the main theorem \cite[Theorem 2.5]{Han86}, proving the existence of a geodesic lamination tracking typical orbits in the zero-entropy case.

 The main objective of this work is to develop a framework that allows the foliated methods developed by Le Calvez to recover, and in some cases extend, the classical results of Homotopy Brouwer Theory, independetly of the Nielsen--Thurston-inspired machinery developed by Handel.

\subsection*{Description of the results}

\quad \ Let $f$ be a Brouwer homeomorphism, and let $\OO = \{\O_{\sspc 1},\ldots,\O_{\sspc r}\}$ be a collection of orbits of $f$. Since every orbit of $f$ is an infinite discrete closed subset of $\R^2$, the surface $\R^2\setminus \OO:=\R^2\setminus \bigcup_{i=1}^r\O_{\sspc i}$ is homeomorphic to the infinite-type surface $\R^2\setminus (\sspc\Z \times \{0\})$, known as the \textit{flute surface} (see \cite{Basmajian1993HyperbolicSF}). Consequently, the surface \(\R^2\setminus \OO\) admits a complete hyperbolic metric of first kind, and the map $f\vert_{\R^2\setminus \OO}$ induces an action $f_\geo$ on the space of geodesics of \(\R^2\setminus \OO\).

\vspace*{0.2cm}
A \textit{geodesic planar line} is a complete geodesic of $\R^2\setminus \OO$ that is a properly embedded line on the whole plane $\R^2$. Given a geodesic planar line $g$, we introduce the following convention:

\begin{itemize}
    \item We say that $g$ \textit{separates} $\OO'\subset \OO$ if each orbit in $\OO'$ lies entirely in either $R(g)$ or $L(g)$, with both $R(g)$ and $L(g)$  containing at least one orbit from $\OO'$ each.
    
    \vspace{0.1cm}
    \item We say that $g$ is \textit{crossed} by $\OO'\subset \OO$ if each orbit in $\OO'$ intersects both $R(g)$ and $L(g)$.
\end{itemize}

\begin{thmx}\label{thmx:structural_thm} There exists a family $\sspc\G$ of pairwise disjoint geodesic planar lines satisfying

    \begin{itemize}
        \item[\textup{\textbf{(i)}} ] Every geodesic in $\G$ is isotopic on  $\R^2\setminus \OO$ to a Brouwer line of $f$.
        
        \item[\textup{\textbf{(ii)}}$\spc$] Every connected component of $\R^2\setminus(\G\cup \OO)\sspc$ is homeomorphic to $\R^2$ or $\R^2\setminus \{0\}.$
        
        \item[\textup{\textbf{(iii)}}] The family $\G$ decomposable into two subfamilies $\G = \mathcal S\sspc \sqcup \sspc\mathcal P$ satisfying
        
\vspace*{0.1cm}

\noindent --- Each geodesic $\sigma \in \mathcal S$ separates $\OO$ and satisfies $f_\geo(\sigma)=\sigma$.

    \vspace*{0.1cm}

       \noindent--- Each geodesic $\lambda \in \mathcal P$ is crossed by some non-empty subset of $\OO$ and satisfies $$\overline{\rule{0pt}{3.5mm}L(f^{n+1}_\geo(\lambda))}\subset L(f^{n}_\geo(\lambda)).$$
       
        \item[\textup{\textbf{(iv)}}] For any two distinct orbits $\O_i,\O_j \in \OO$, the following dichotomy holds:
        
        \vspace*{0.1cm}

         \noindent--- Either there exists a geodesic $\lambda \in \mathcal P$ crossed by $\{\O_i,\O_j\}$.%, meaning that each $\O$ and $\O^\pp$ intersect both sides of $\lambda$.

         \vspace*{0.1cm}
        
        \noindent--- Or there exists a geodesic $\sigma \in \mathcal S$ that separates $\{\O_i,\O_j\}$.%., meaning that $\O$ and $\O^\pp$ are contained in different sides of $\sigma$.
    \end{itemize}
\end{thmx}

Essentially, the family of geodesic leaves $\G$ provides a geometric structure that enables us to track the homotopic dynamics of $\class{f,\OO}$. The subfamily $\mathcal S$ of \emph{separating geodesics} plays a role analogous to that of reducing lines in Handel's theory, as they partition the set of orbits $\OO$ into dynamically independent subsets. On the other hand, the subfamily $\mathcal P$ of \emph{pushing geodesics} describes the homotopic dynamics of $f$, as they are crossed by orbits in $\OO$ and are pushed forward by $f_\geo$ in a manner reminiscent of Brouwer lines.

It is worth noting that the family $\G$ is not necessarily uniquely defined for $\class{f,\OO}$, since it is constructed as a family of geodesic representatives of leaves of a transverse foliation $\F$ of $f$, which always exists but are not uniquely defined \cite{lec1}.

\begin{figure}[h!]
    \center 
    \hspace*{-0.2cm}\begin{overpic}[width=6.6cm, height=3.7cm, tics=10]{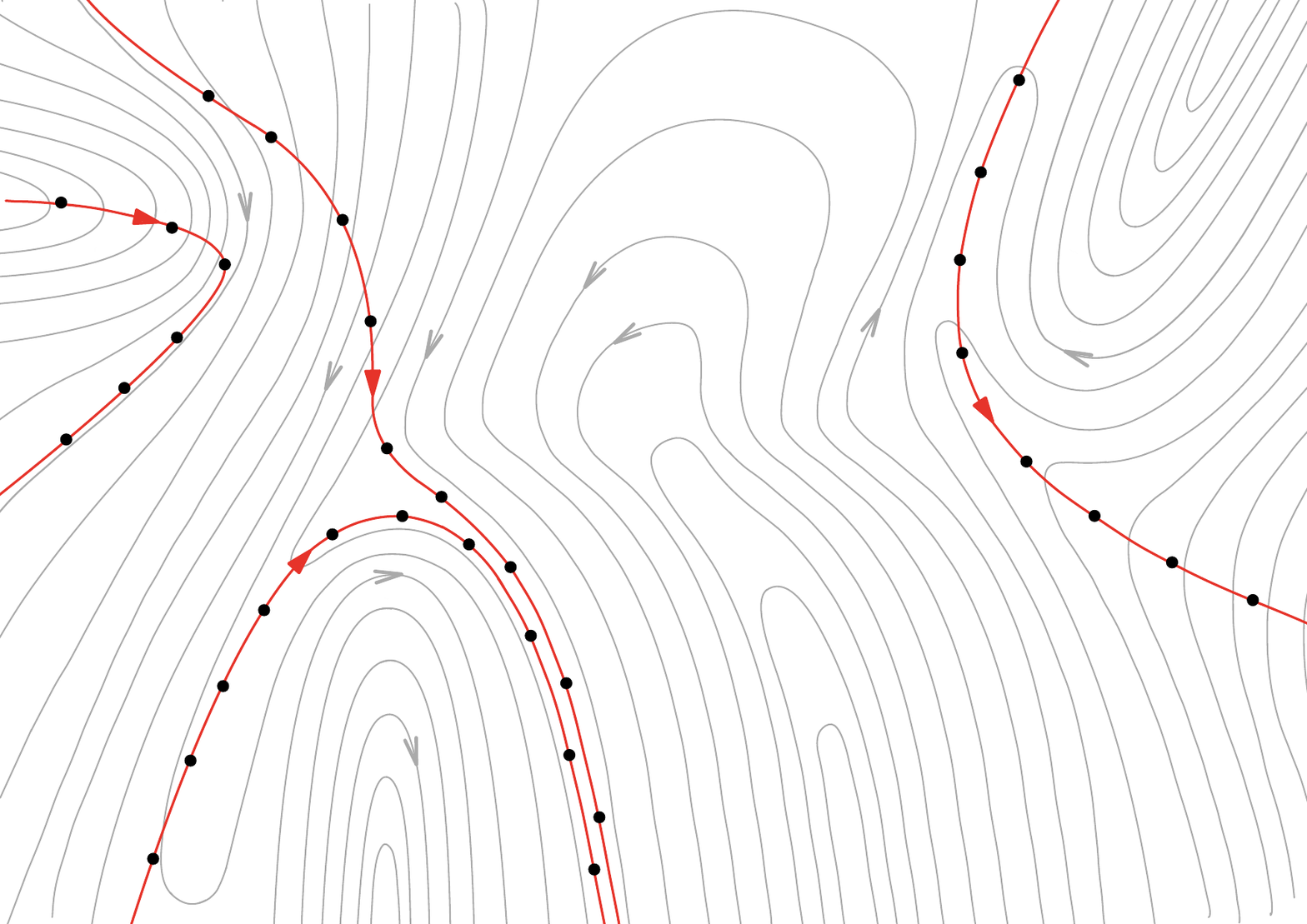}\put (91,1) {\colorbox{white}{\color{gray}\LARGE$\displaystyle \rule{0cm}{0.46cm}\F$}}
\end{overpic}
\hspace*{1.2cm}\begin{overpic}[width=6.6cm, height=3.7cm, tics=10]{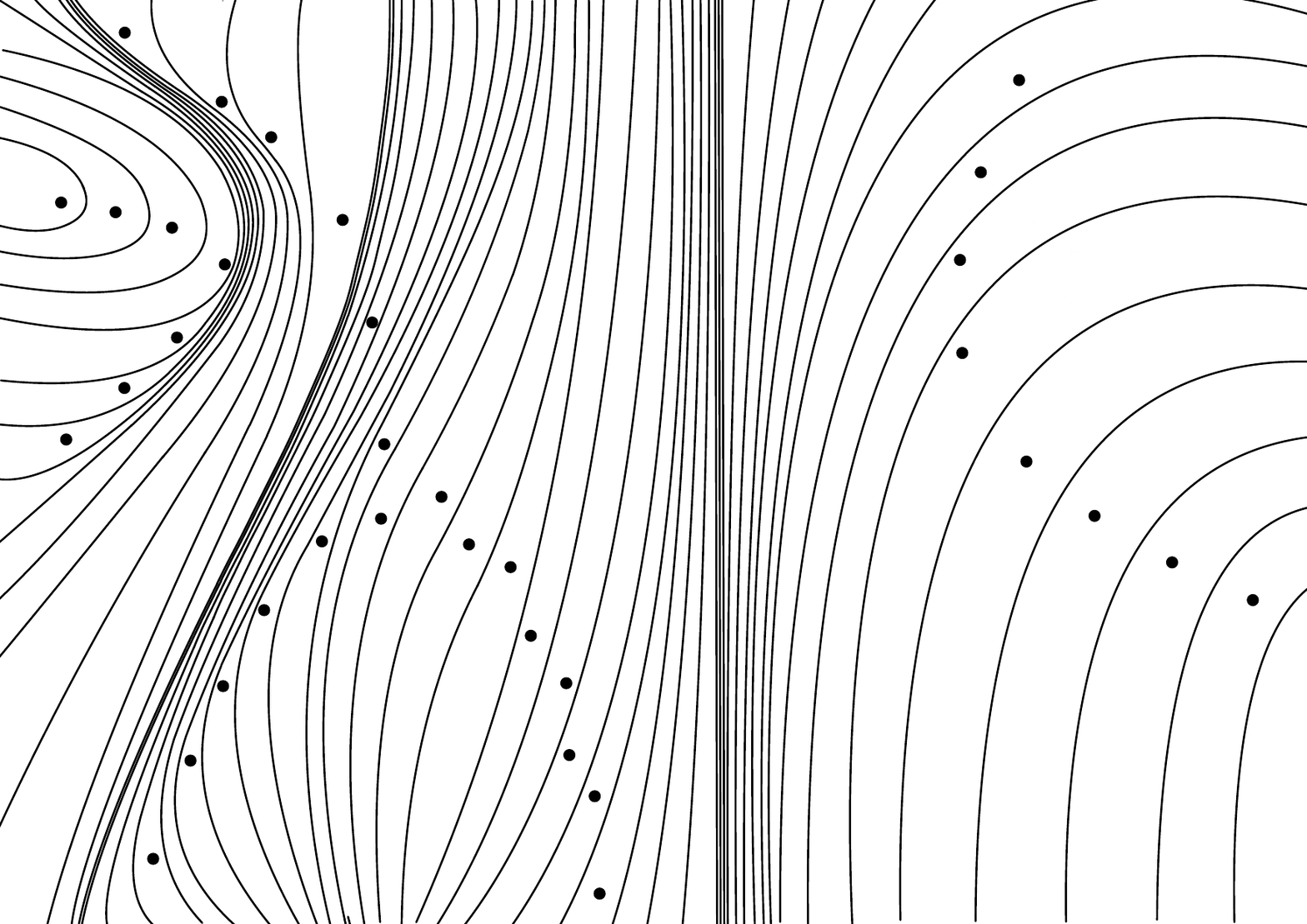}
    \put (90,2) {\colorbox{white}{\color{black}\LARGE$\displaystyle \rule{0cm}{0.48cm}\sspc \G$}}
\end{overpic}
\end{figure}

We now turn to the principal dynamical feature of our foliated framework, the Pushing Lemma. This result desribes the dynamical behavior of pushing-equivalent classes in $\P$ and serves as a key ingredient in establishing our approach independently of Handel’s original machinery.

 Two geodesic leaves $\lambda, \lambda^\pp \in \mathcal P$ are said to be \textit{pushing-equivalent} if they are crossed the same orbits in $\OO$, and no orbit in $\OO$ lies entirely on the connected component of $\R^2\setminus (\lambda \cup \lambda^\pp)$  bounded by $\lambda$ and $\lambda^\pp$. Roughtly speaking, two pushing geodesic leaves are pushing equivalent whenever the region between them has finitely many punctures and is trivially laminated by $\G$.  Note that this relation defines an equivalence relation $\sim$ on the subfamily $\mathcal P\subset \G$.

 \vspace*{0.3cm}

\begin{figure}[h!]
    \center 
    \hspace*{-0.1cm}\begin{overpic}[width=3.3cm, height=3.4cm, tics=10]{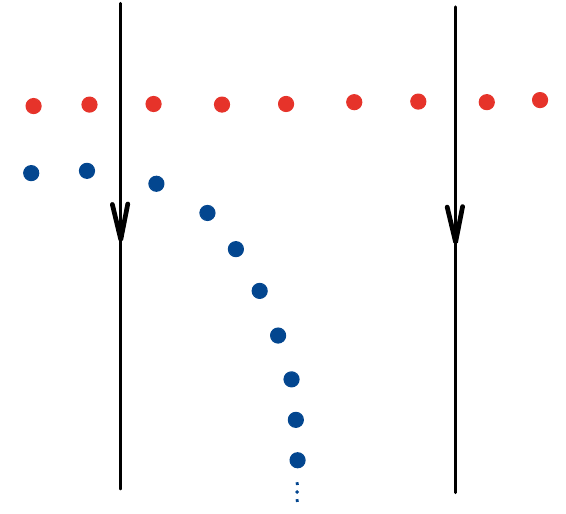}
        \put (34,-15) {{\color{black}\large$\displaystyle \lambda\not\sim\lambda'$}}
        \put (9,95) {{\color{black}\large$\displaystyle \lambda$}}
        \put (66,95) {{\color{black}\large$\displaystyle \lambda'$}}
        \put (-14.5,78.5) {{\color{myRED}\large$\displaystyle \O_{\sspc i}$}}
        \put (-14.5,61) {{\color{myBLUE}\large$\displaystyle \O_{\sspc j}$}}
\end{overpic}
\hspace*{1.5cm}\begin{overpic}[width=3.3cm, height=3.4cm, tics=10]{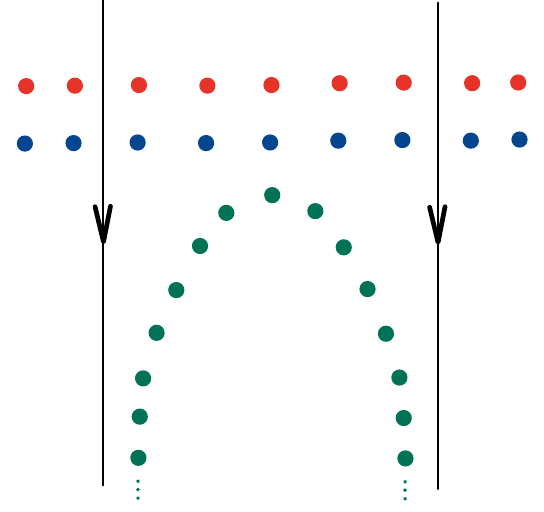}
    \put (9,95) {{\color{black}\large$\displaystyle \lambda$}}
    \put (34,-15) {{\color{black}\large$\displaystyle \lambda\not\sim\lambda'$}}
        \put (66,95) {{\color{black}\large$\displaystyle \lambda'$}}
        \put (-14.6,81.5) {{\color{myRED}\large$\displaystyle \O_{\sspc i}$}}
        \put (-14.6,67.2) {{\color{myBLUE}\large$\displaystyle \O_{\sspc j}$}}
        \put (31,25) {{\color{myGREEN}\large$\displaystyle \O_{\sspc k}$}}
\end{overpic}
\hspace*{1.55cm}\begin{overpic}[width=3.2cm, height=3.4cm, tics=10]{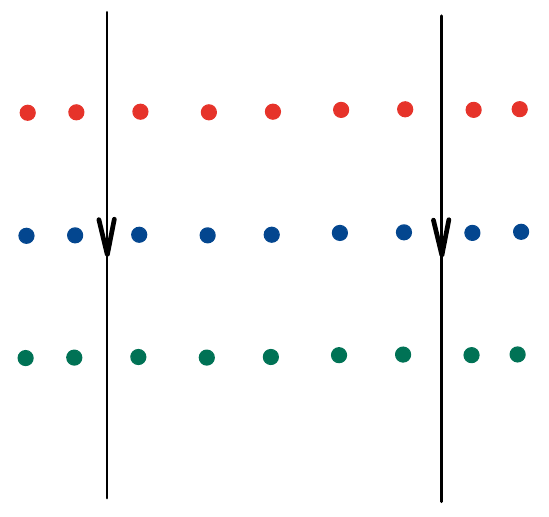}
    \put (9,95) {{\color{black}\large$\displaystyle \lambda$}}
    \put (34,-15) {{\color{black}\large$\displaystyle \lambda\sim\lambda'$}}
        \put (66,95) {{\color{black}\large$\displaystyle \lambda'$}}
        \put (-14,78) {{\color{myRED}\large$\displaystyle \O_{\sspc i}$}}
        \put (-14,53) {{\color{myBLUE}\large$\displaystyle \O_{\sspc j}$}}
        \put (-14,28) {{\color{myGREEN}\large$\displaystyle \O_{\sspc k}$}}
\end{overpic}
\end{figure}

 \vspace*{0.3cm}

\begin{pushinglemma*}
    Two geodesics \(\lambda,\lambda^{\prime} \in \P\) satisfying \(L(\lambda') \subset L(\lambda)\) are pushing-equivalent if, and only if, there exists an integer $N>0$ that satisfies
    \mycomment{-0.07cm}
    $$ L(f^N_\geo(\lambda)) \subset L(\lambda').
    $$
\end{pushinglemma*}

Before we head to our main theorem, we need to establish a result that finds the simplest possible family $\G$ satisfying Theorem \ref{thmx:structural_thm}. For that, we say that a pushing geodesic leaf $\lambda \in \mathcal P$ is 
\vspace*{0.1cm}
\begin{itemize}
    \item[] \textit{Forward-wandering$\spc$:} if the family of iterates $\{f^n_\geo(\lambda)\}_{n>0}$ is locally-finite.
    \vspace*{0.1cm}
    \item[] \textit{Backward-wandering$\spc$:} if the family of iterates $\{f^n_\geo(\lambda)\}_{n<0}$ is locally-finite.
\end{itemize}
\vspace*{0.05cm}
Observe that geodesic leaves in $\mathcal S$ are always fixed by $f_\geo$, and thus they cannot be wandering.

\vspace*{0.1cm}
\begin{thmx}\label{thmx:II-B}
    There exists a family $\sspc\G$ of pairwise disjoint geodesic planar lines such that, in addition to the properties of Theorem \ref{thmx:structural_thm}, the following holds:
\begin{itemize}[leftmargin=1.5cm]
    \item[\textup{\textbf{(W1)}}] For any forward-wandering geodesic $\lambda \in \P$, the subfamily 
    $$\G_{L,\lambda}:=\{\lambda^\pp \in \G \mid \lambda^\pp\subset \overline{\rule{0cm}{0.22cm}L(\lambda)}\}$$
    is forward-invariant under $f_\geo$ and is formed by geodesics pushing-equivalent to $\lambda$.

    \vspace*{0.1cm}
    \item[\textup{\textbf{(W2)}}] Similarly, for any backward-wandering geodesic $\lambda \in \P$, the subfamily
    $$\G_{R,\lambda}:=\{\lambda^\pp \in \G \mid \lambda^\pp\subset \overline{\rule{0cm}{0.22cm}R(\lambda)}\}$$
    is backward-invariant under $f_\geo$ and is formed by geodesics pushing-equivalent to $\lambda$.
\end{itemize}
\end{thmx}

Relying on the Pushing Lemma and the technical methods introduced in \cite{schuback1}, we obtain a complete characterization of wandering geodesic leaves,  which until now could only be derived using Handel’s machinery (see \cite[Case (1) in the proof of Proposition 6.6]{handel99}).

\mycomment{0.04cm}

\begin{thmx}\label{thmx:II-C}
     Every geodesic leaf $\lambda\in \G$ satisfies the following equivalences
     \mycomment{-0.05cm}
     $$
\ \ \lambda  \text{ is forward-wandering } \ \spc \iff \sspc \text{ no orbit in } \OO \text{ lies entirely on } L(\lambda),
$$
\mycomment{-0.6cm}
$$
\ \ \lambda \text{ is backward-wandering }  \iff \sspc \text{ no orbit in } \OO \text{ lies entirely on } R(\lambda).
$$
\end{thmx}

Property (iv) from Theorem \ref{thmx:structural_thm} implies that in the cases $r\leq 2$, there are two possible settings: 

\vspace*{0.1cm}
\begin{itemize}[leftmargin = 1.4cm]
    \item[(1)] Either there exists a pushing geodesic leaf $\lambda \in \P$ crossed by every orbit in $\OO$.
    
    \vspace*{0.1cm}
    \item[(2)] Or there are two orbits in $\OO$, which are separated by a separating geodesic leaf.
\end{itemize}

\vspace*{0.1cm}
\noindent By treating case (1) with Theorem \ref{thmx:II-C}, and case (2) with Proposition \ref{prop:separated_geodesic_trajectories}, we recover Handel's classification of Brouwer mapping classes relative to $r\leq2$ orbits, originally established in \cite{handel99}.
Furthermore, our approach extends Handel’s results to the setting of relative fixed-point–free isotopies, yielding a new classification theorem.

The \textit{fine Brouwer mapping class} $\bclass{f,\sspc\OO}$ is the set of all Brouwer homeomorphisms $f^\pp$ that are isotopic to $f$ relative to $\OO$ via an isotopy $(f_t)_{t\in [0,1]}$ that satisfies $\text{Fix}(f_t) = \varnothing$ for all $t \in [0,1]$.\break
By denoting the space of Brouwer homeomorphisms by $\mathcal B$, we can finally state our final result.
    
\begin{thmxc}\label{cor:absolute_brouwer_classification-intro}
 If there exists a pushing geodesic leaf $\lambda \in \P$ crossed by every orbit in $\OO$, in particular if $r=1$, then the fine Brouwer mapping class $\bclass{f,\sspc\OO}$ coincides with $\class{f,\OO} \cap \mathcal B$.
\end{thmxc}

\section{Preliminaries}

\subsection{Topology on the flute surface}\label{sec:topology_R2Z}

TThe surface $\R^2\setminus \Z$ given by puncturing the plane at the integer points $\Z\equiv\Z\times\{0\} \sspc\subset \R^2$ is often called the \textit{flute surface}. Note that, $\R^2\setminus \Z$ is homeomorphic to $\R^2\setminus X$ if, and only if, $X\subset \R^2$ is the image of a proper embedding of $\Z$ into $\R^2$. The fundamental group $\pi_1(\R^2\setminus \Z)$ is isomorphic to the free group on countably many generators, with the generators represented by loops around each point of $\Z\subset \R^2$. Since $\pi_1(\R^2\setminus \Z)$ is not finitely generated, $\R^2\setminus \Z$ is called a surface of infinite type.

Let $\homeo\sspc(\R^2\setminus \Z)$ be the group of self homeomorphisms of $\R^2\setminus \Z\sspc$, endowed with the compact-open topology, or equivalently, the topology of uniform convergence on compact sets.  Let $\homeo^+\sspc(\R^2\setminus \Z)$ be the subgroup of orientation-preserving homeomorphisms of $\R^2\setminus \Z$.
We similarly define the homeomorphism groups $\homeo\sspc(\R^2)$ and $\homeo^+\sspc(\R^2)$ for the plane.

Each homeomorphism $h \in \homeo\sspc(\R^2\setminus \Z)$ admits a unique extension to the plane, often denoted by $\overline{\rule{0cm}{0.28cm}\sspc h\sspc }\in \homeo\sspc(\R^2)$. This means that $\overline{\rule{0cm}{0.28cm}\sspc h\sspc }$ satisfies $\overline{\rule{0cm}{0.28cm}\sspc h\sspc }\vert_{\R^2\setminus \Z} = h$, and thus $\overline{\rule{0cm}{0.28cm}\sspc h\sspc }(\Z) = \Z$.

An isotopy $(h_t)_{t\in [0,1]}$ on $\R^2\setminus \Z$ is a continuous path $t\longmapsto h_t$ within $\homeo\sspc(\R^2\setminus \Z)$.  
Two homeomorphisms $h,h^\pp \in \homeo\sspc(\R^2\setminus \Z)$ are isotopic if there exists an isotopy $(h_t)_{t\in [0,1]}$ on $\R^2\setminus \Z$ such that $h_0 = h$ and $h_1 = h^\pp\sspc$. This defines an equivalence relation, and the isotopy class of $h\in \homeo\sspc(\R^2\setminus \Z)$ is denoted by $\class h$. We similarly define isotopies on the plane.

Given a closed subset $X\subset \R^2$, we say that an isotopy $(h_t)_{t\in [0,1]}$ on $\R^2$ is relative to $X$ if it satisfies $h_t(x) = h_0(x)$, for all $t\in [0,1]$ and $x \in X$. This also defines an equivalence relation in $\homeo\sspc(\R^2)$, and the isotopy class of $h \in \homeo\sspc(\R^2)$ relative to $X$ is denoted by $\class{h,\sspc X}$.

As explained in \cite{BeguinCrovisierLeRoux2020}, the notion of isotopy on $\R^2$ relative to a closed subset $X\subset\R^2$ is generally stronger than the notion of isotopy on the surface $\R^2\setminus X$. However, in the particular
\noindent case where $X$ is totally disconnected, both notions are equivalent \cite[Proposition 1.5]{BeguinCrovisierLeRoux2020}. Since $\Z\subset \R^2$ is closed and totally disconnected, every isotopy $(h_t)_{t\in [0,1]}$ on $\R^2\setminus \Z$ is extended to an isotopy $(\overline{\rule{0cm}{0.28cm}\sspc h\sspc }_t)_{t\in [0,1]}$ on $\R^2$ relative to $\Z$, and vice versa through the restriction map.

 Let $\homeo_{\sspc 0}\sspc(\R^2\setminus \Z)$ be the set of homeomorphisms of $\R^2\setminus \Z$ isotopic to the identity, meaning, the connected component of $\homeo\sspc(\R^2\setminus \Z)$ that contains the identity map $\textup{id}_{\R^2\setminus \Z}$.

The \textit{mapping class group} of $\R^2\setminus \Z\sspc$ is then defined as the quotient group
 \mycomment{-0.12cm}
$$\textup{MCG}(\R^2\setminus \Z) := \homeo^+\sspc(\R^2\setminus \Z) /\sspc \homeo_{\sspc 0}\sspc(\R^2\setminus \Z).$$

 \mycomment{-0.24cm}
\noindent
In other words, $\mcg(\R^2\setminus \Z)$ is the group of isotopy classes of elements of $\homeo^+\sspc(\R^2\setminus \Z)$.

A \textit{line} on $\R^2\setminus \Z$ is a proper embedding of $\R$ into $\R^2\setminus \Z$. A \textit{planar line} on $\R^2\setminus \Z$ is a line that is also a topological line on the plane.
We define the \textit{$\alpha$-limit point} and \textit{$\omega$-limit point} of a line $\ell$ as the points $\alpha(\ell),\sspc \omega(\ell) \in \R^2$ that realize the following limits (if they exist):
 \mycomment{-0.12cm}
$$\alpha(\ell):=\lim_{t\rightarrow-\infty} \ell(t) \quad \text{ and } \quad \omega(\ell):=\lim_{t\rightarrow\infty} \ell(t).$$
Note that, planar lines correspond to lines that have neither an $\alpha$-limit nor an $\omega$-limit point.

For any line $\ell$, we denote by $\overline{\rule{0cm}{0.3cm}\sspc \ell\sspc }$ the closure of $\ell$ in $\R^2$, or equivalently, $\overline{\rule{0cm}{0.3cm}\sspc \ell\sspc } = \ell \cup \{\alpha(\ell),\sspc \omega(\ell)\}$.

Recall that, a family of lines $\{\ell_i\}_{i \in I}$ is called locally-finite if, for every compact set $K\subset \R^2$, the set of indices $\{i\in I : \ell_i \cap K \neq \varnothing\}$ is finite.

A \textit{proper multiline} on $\R^2\setminus \Z$ is a (possibly finite) family of lines $\{\ell_i\}_{i \in I}$ such that 
\begin{itemize}
    \item[$\sbullet$] For any $i\in I$, if the limit point $\alpha(\ell_i)$ exists, then  $\alpha(\ell_i) = \omega(\ell_{i-1})$.
    \item[$\sbullet$] The union $\bigcup_{i\in I} \overline{\rule{0cm}{0.3cm}\sspc \ell_i\sspc }$ is a topological line on the plane.
\end{itemize}
 \mycomment{0.2cm}
Consequently, if $\{\ell_i\}_{i \in I}$ is a proper multiline, then it is locally-finite and pairwise disjoint.

We reserve the notation $\Lambda= \prod_{i\in I}\ell_i$ to denote a proper multiline $\{\ell_i\}_{i \in I}$, and similarly to the definition of left and right sides of a topological line on the plane, we denote
 \vspace{0.1cm}
$$R(\Lambda):= R
\Bigl( \sspc \bigcup_{i\in I} \overline{\rule{0cm}{0.3cm}\sspc \ell_i\sspc } \sspc\Bigr) \quad \text{ and } \quad L(\Lambda):= L\Bigl( \sspc \bigcup_{i\in I} \overline{\rule{0cm}{0.3cm}\sspc \ell_i\sspc } \sspc\Bigr),$$
where $\bigcup_{i\in I} \overline{\rule{0cm}{0.3cm}\sspc \ell_i\sspc }$ is oriented naturally according to the orientation of the lines $\ell_i$.

\vspace*{0cm}

\begin{figure}[h!]
    \center
    \hspace*{-6cm}\begin{overpic}[width=8.3cm, height=3.7cm, tics=10]{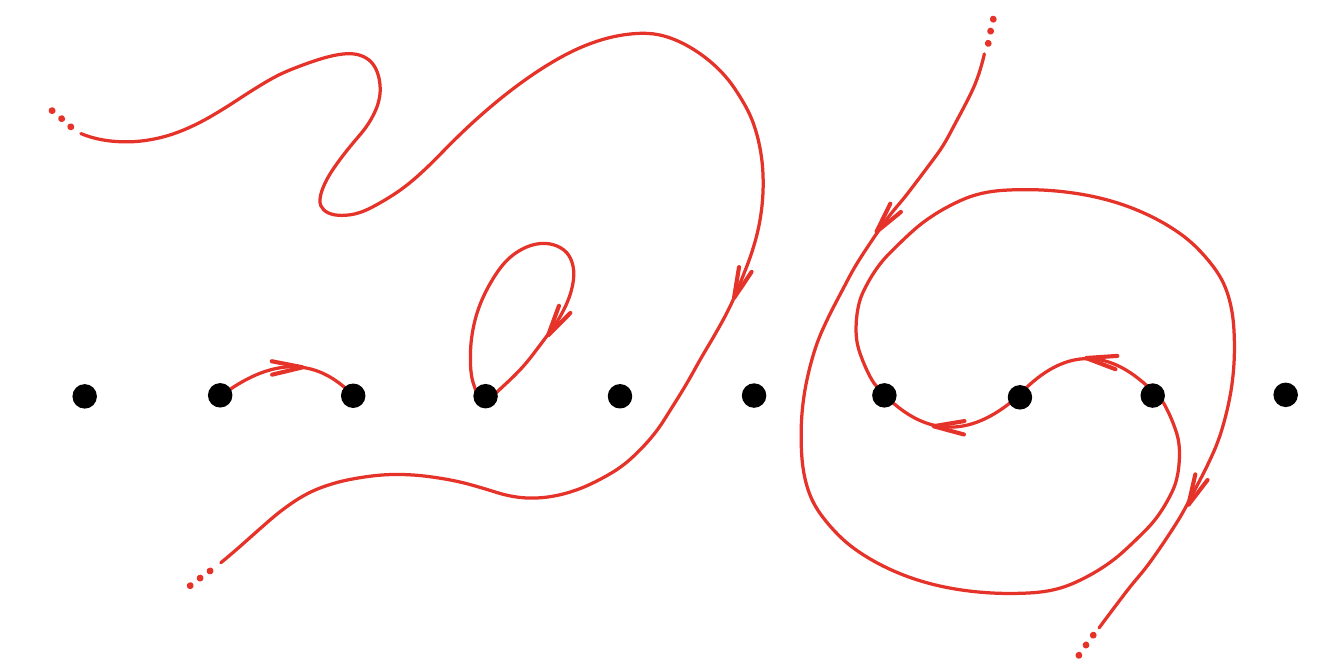}
        \put (33,42) {{\color{myRED}\large$\displaystyle \ell_3$}}
        \put (108,35.2) {{\color{black}\large$\displaystyle \ell_1,\spc...\sspc,\spc\ell_7 \ \text{ are lines}$}}
        \put (108,23.2) {{\color{black}\large$\displaystyle \ell_3\  \text{ is a planar line}$}}
        \put (108,10.2) {{\color{black}\large$\displaystyle \Lambda=\prod_{i=4}^7 \ell_i\  \text{ is a proper multiline}$}}
        \put (-3,16.5) {{\color{black}\large$\displaystyle \Z$}}
        \put (65.5,35) {{\color{myRED}\large$\displaystyle \ell_4$}}
        \put (20,23.3) {{\color{myRED}\large$\displaystyle \ell_1$}}
        \put (80,23.3) {{\color{myRED}\large$\displaystyle \ell_5$}}
        \put (70,19.5) {{\color{myRED}\large$\displaystyle \ell_6$}}
        \put (40,30.5) {{\color{myRED}\large$\displaystyle \ell_2$}}
        \put (87,32) {{\color{myRED}\large$\displaystyle \ell_7$}}
        % \put (58,15) {\colorbox{white}{$\rule{0cm}{0.3cm}\ \  $}}
        % \put (58.3,15.3) {{\color{myBLUE}\large$\displaystyle \Gamma_{\O^\pp} $}}
        % \put (57.6,31.1) {\colorbox{white}{\color{myDARKGRAY}\large$\rule{0cm}{0.3cm}\   \  $}}
        % \put (58.3,31.2) {{\color{myRED}\large$\displaystyle \Gamma_\O$}}
        % \put (129.6,42.5) {{\color{black}\normalsize$\displaystyle C_\O$}{\color{black}\normalsize$\displaystyle \ =\ $}{\color{black}\normalsize$\displaystyle C_{\O^\pp}$}}
        % \put (123,33) {{\color{black}\normalsize$\displaystyle \Rtop \O$}{\color{black}\normalsize$\displaystyle \ =\ $}{\color{black}\normalsize$\displaystyle \Rtop {\O^\pp}$}}
        % \put (123,23) {{\color{black}\normalsize$\displaystyle \Rbot \O$}{\color{black}\normalsize$\displaystyle \ =\ $}{\color{black}\normalsize$\displaystyle \Rbot {\O^\pp}$}}
        % \put (123,13) {{\color{black}\normalsize$\displaystyle \Ltop \O$}{\color{black}\normalsize$\displaystyle \ \thinsubsetneq\ $}{\color{black}\normalsize$\displaystyle \Ltop {\O^\pp}$}}
        % \put (123,3) {{\color{black}\normalsize$\displaystyle \Lbot \O$}{\color{black}\normalsize$\displaystyle \ \thinsupsetneq \ $}{\color{black}\normalsize$\displaystyle \Lbot {\O^\pp}$}}
\end{overpic}
\end{figure}

Two lines $\ell$ and $\ell^\pp$ are said to be isotopic if there exists an isotopy $(h_t)_{t\in [0,1]}$ on $\R^2\setminus \Z$ that satisfies $h_0 = \textup{id}_{\R^2\setminus \Z}$ and $h_1(\ell) = \ell^\pp\sspc$. Roughtly speaking, two lines $\ell$ and $\ell^\pp$ are isotopic if one can be deformed into the other through an isotopy on $\R^2\setminus \Z$. This defines an equivalence relation in the set of lines on $\R^2\setminus \Z$, and the isotopy class of a line $\ell$ is denoted by $\class{\ell}$.

The following result is well-known and can be traced back to Epstein \cite{Epstein1966}. %We content ourselves with stating it in the context of the surface $\R^2\setminus \Z$.

\begin{theorem}[Epstein, \cite{Epstein1966}]\label{thm:epstein}
    Let $\ell,\ell^\pp:\R\longrightarrow \R^2\setminus \Z$ be two lines. If there exists a proper continuous map
    $ H:[0,1]\times \R \longrightarrow \R^2\setminus \Z^2 $ that satisfies $H(0,t) = \ell(t)$ and $H(1,t) = \ell^\pp(t)\sspc$, for all $t\in \R$, then the lines $\ell$ and $\ell^\pp$ are isotopic.
\end{theorem}

Note that, if $\ell$ and $\ell^\pp$ are isotopic lines, then they share the same $\alpha$- and $\omega$-limit points. Moreover, if $\ell$ and $\ell^\pp$ are isotopic planar lines, then the punctures of $\R^2\setminus \Z$ lying on the right side $R(\ell)$ are the same as those lying on $R(\ell^\pp)$, and similarly for the left sides $L(\ell)$ and $L(\ell^\pp)$.

Two proper multilines $\Lambda = \medmath{\prod_{i\in I}}\spc \ell_i \sspc$ and $\Lambda^\pp = \medmath{\prod_{j\in J}}\spc \ell_j^\pp \sspc$ are said to be \textit{globally isotopic} if there exists an isotopy $(h_t)_{t\in [0,1]}$ on $\R^2\setminus \Z$ that satisfies  $h_0 = \textup{id}_{\R^2\setminus \Z}$ and $h_1\bigl(\bigcup_{i\in I} \ell_i\bigr) = \bigcup_{j\in J} \ell_j^\pp\sspc$.

A line $\ell$ is said to be \textit{essential} on $\R^2\setminus \Z$ if there exists a compact subset $K\subset \R^2\setminus\Z$ such that, for any isotopy $(h_t)_{t\in [0,1]}$ on $\R^2\setminus \Z$ from the identity, $h_t(\ell) \cap K \neq \varnothing$ for all $t\in [0,1]$. Equivalently, $\ell$ is essential if no connected component of $\R^2\setminus( \Z\cup \ell)$ is homeomorphic to $\R^2$.
Note that, the property of being essential is preserved under isotopies.

A proper multiline $\Lambda = \medmath{\prod_{i\in I}}\spc \ell_i \sspc$ is said to be \textit{globally essential} if the family $\{\ell_i\}_{i\in I}$ is formed by pairwise non-isotopic (up to reversing orientation) essential lines.

 \mycomment{0.2cm}

\begin{figure}[h!]
    \center
    \hspace*{0cm}\begin{overpic}[width=14cm, height=3.9cm, tics=10]{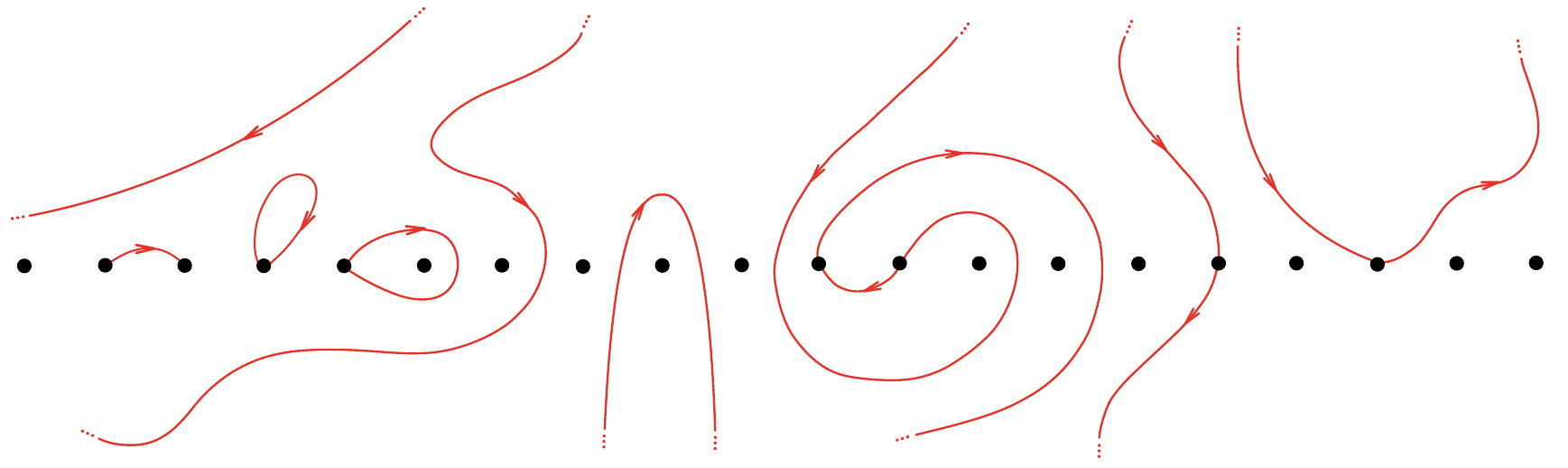}
        \put (101.5,11.5) {{\color{black}\large$\displaystyle \Z$}}
        \put (13.5,21.5) {{\color{myRED}\large$\displaystyle \ell_1$}}
        \put (19.5,18.3) {{\color{myRED}\large$\displaystyle \ell_3$}}
        \put (25,15.6) {{\color{myRED}\large$\displaystyle \ell_4$}}
        \put (30.5,18.8) {{\color{myRED}\large$\displaystyle \ell_5$}}
        \put (41,17.7) {{\color{myRED}\large$\displaystyle \ell_6$}}
         \put (54,22.2) {{\color{myRED}\large$\displaystyle \ell_7$}}
         \put (53.5,7.6) {{\color{myRED}\large$\displaystyle \ell_8$}}
        \put (8,9.5) {{\color{myRED}\large$\displaystyle \ell_2$}}
        \put (82,17.5) {{\color{myRED}\large$\displaystyle \ell_{12}$}}
        \put (92.9,18.6) {{\color{myRED}\large$\displaystyle \ell_{13}$}}
        \put (60,20) {{\color{myRED}\large$\displaystyle \ell_9$}}
        \put (74.5,20.4) {{\color{myRED}\large$\displaystyle \ell_{10}$}}
        \put (75.7,5.8) {{\color{myRED}\large$\displaystyle \ell_{11}$}}
\end{overpic}
\end{figure}

In the figure above, we illustrate the lines $\ell_1, \dots, \ell_{13}$, which satisfy the following:
\begin{itemize}[leftmargin=1.4cm]
\item[$\sbullet$] $\ell_1$ and $\ell_3$ are non-essential lines, while all the others are essential;
 \mycomment{0.2cm}
\item[$\sbullet$] $\displaystyle \Lambda = \prod_{i=7}^9 \ell_i$ and $\displaystyle \Lambda' = \prod_{i=10}^{11} \ell_i$ are globally essential proper multilines;
\item[$\sbullet$] $\displaystyle \Lambda^{\prime \prime} = \prod_{i=12}^{13} \ell_i$ is a proper multiline that is not globally essential, since $\ell_{12}$ becomes isotopic to $\ell_{13}$ after reversing its orientation.
\end{itemize}

\subsection{Hyperbolic geometry on the flute surface}\label{sec:hyperbolic_geometry_R2Z}

The flute surface $\R^2\setminus \Z$ can be endowed with a complete hyperbolic metric of first kind with infinite volume (see \cite{Basmajian1993HyperbolicSF}). This essentially means that $\R^2\setminus \Z$ can be realized as a quotient of the form
$$\R^2\setminus \Z = \mathbb H^2/\Gamma,$$
where $\mathbb H^2$ is the hyperbolic plane and $\Gamma < \text{Isom}^+(\mathbb H^2)$ is a torsion-free Fuchsian subgroup admitting the whole unit circle as its limit set, that is, $\Lambda_\Gamma =\partial_\infty \mathbb H^2$ (see \cite{katok1992fuchsian} for details).

 \mycomment{0.1cm}

\begin{remark}
    Easy method to construct a complete hyperbolic metric of first kind on $\R^2\setminus \Z$:
    
    \vspace*{0.35cm}

    Let $T$ be the horizontal unitary translation on $\R^2$. Note that the quotient
    $ \Sigma:=(\R^2\setminus \Z )/\spc T$
    defines a topological surface homeomorphic to the $3$-punctured sphere, which is known to admit a complete hyperbolic metric of finite volume. Thus, there exists a finitely generated torsion-free Fuchsian subgroup $\Gamma < \text{Isom}^+(\mathbb H^2)$ such that $\Sigma=\mathbb H^2/\Gamma$ and $\text{vol}(\mathbb H^2/\Gamma)<\infty$. From the finite volume condition, it follows automatically that $\Lambda_\Gamma = \partial_\infty \mathbb H^2$ (see \cite{katok1992fuchsian}).

    Next, observe that $\R^2\setminus \Z \longrightarrow (\sspc\R^2\setminus \Z \sspc)/\spc T$ is a normal covering map of infinite degree.  This implies, by the Galois correspondence for covering spaces, that there exists an infinite index normal subgroup $\Gamma^\pp \lhd \Gamma$ that realizes the quotient $\R^2\setminus \Z = \mathbb H^2/\Gamma^\pp$. Observe that this induces a complete hyperbolic metric of infinite volume on $\R^2\setminus \Z$, since 
     \mycomment{-0.1cm}
    $$\text{vol}(\mathbb H^2/\Gamma^\pp) = [\Gamma:\Gamma^\pp]\cdot \text{vol}(\mathbb H^2/\Gamma) = \infty.$$

     \mycomment{-0.25cm}
    \noindent However, even if the volume of $\R^2\setminus \Z = \mathbb H^2/\Gamma^\pp$ is infinite, we still have that $\Lambda_{\Gamma^\pp} = \partial_\infty \mathbb H^2$. Indeed, using the fact that $\Gamma^\pp$ is a normal subgroup of $\Gamma$, one can show that the limit set $\Lambda_{\Gamma^\pp}$ is invariant under the action of $\Gamma$. Using the minimality of the action of $\Gamma$ on $\Lambda_\Gamma = \partial_\infty \mathbb H^2$, one can conclude that $\Lambda_{\Gamma^\pp} = \Lambda_\Gamma = \partial_\infty \mathbb H^2$.
    Thus, the resulting hyperbolic structure on $\R^2\setminus \Z$ induced by the quotient $\R^2\setminus \Z = \mathbb H^2/\Gamma^\pp$ is indeed of first kind.
\end{remark}

The above remark is added as an illustration, and it does not play any role in the arguments of this work. Now, we present the features of such a hyperbolic structure on $\R^2\setminus \Z$.

 \mycomment{0.2cm}

The (complete) geodesics on the surface $\R^2\setminus \Z$ are the projections of the (complete) geodesics of $\mathbb H^2$ by the universal covering map $\widetilde \pi:\mathbb H^2\longrightarrow \mathbb H^2/\Gamma = \R^2\setminus \Z$. 
A \textit{geodesic line} on $\R^2\setminus \Z$ is a geodesic that is also a line.
A \textit{geodesic lamination} on $\R^2\setminus \Z$ is a closed subset of $\R^2\setminus \Z$ that is foliated by pairwise disjoint geodesic lines. It is a classic result that, if $\{g_i\}_{i\in I}$ is a family of pairwise disjoint  geodesic lines on $\R^2\setminus \Z$, then the closure  $\overline{\rule{0cm}{0.31cm}\bigcup_{i\in I} g_i}$ is a geodesic lamination.

An important consequence of the hyperbolic structure on $\R^2\setminus \Z$ being of first kind is that every essential line $\ell$  on $\R^2\setminus \Z$ is isotopic to a unique geodesic line $\ell^\geo$, which we referred to as the \textit{geodesic representative} of $\ell$ (see \cite[Proposition 2.1]{MATSUMOTO2000191} for a detailed proof). 

We remark that, if $\{\ell_i\}_{i\in I}$ is a locally-finite family of essential lines, then the family of geodesic representatives $\{\ell_i^\geo\}_{i\in I}$ is also locally-finite (see \cite[Lemma 1.4, Lemma 3.2]{LEROUX_2017}).

It is a well-known fact that geodesic representatives minimize the number of intersections within their isotopy classes, and this property is called the \textit{minimal intersection property}. More precisely, for any two essential lines $\ell$ and $\ell^\pp$ on $\R^2\setminus \Z$, the following inequality holds
$$\#(\ell^\geo \cap (\ell^\pp)^\geo) \leq \#(\ell \cap \ell^\pp).$$
Even further, for any bigon $B\subset\R^2$ formed by the geodesics $\ell^\geo$ and $(\ell^\pp)^\geo$, meaning a topological disk $B\subset \R^2$ whose boundary is formed by two arcs $\alpha\subset \ell^\geo$ and $\alpha^\pp\subset (\ell^\pp)^\geo$ sharing the same endpoints, there exists at least one puncture of $\R^2\setminus\Z$ lying in the interior of $B$ (For more details, see \cite[Proposition 1.2]{LEROUX_2017}.)

Any homeomorphism $h \in \homeo\sspc(\R^2\setminus \Z)$ induces an action on geodesic lines of $\R^2\setminus \Z$, denoted by $h_\geo$, that maps any geodesic line $g$ to the geodesic representative of $h(g)$, that is, $h_\geo(g) := h(g)^\geo.$
In particular, this implies that $h_\geo(\ell^\geo) = h(\ell)^\geo$, for any essential line $\ell$.

The next result is central to our arguments, and it is refered to as the \textit{Straightening Principle}. It corresponds to \cite[Lemma 3.5]{handel99}, but we also refer to \cite[Lemma 1.4, Lemma 3.2]{LEROUX_2017}. In his master's thesis \cite{Conejeros2012}, Conejeros provides a detailed proof of this result.

\begin{lemma}[Straightening Principle, \cite{handel99}]\label{lemma:straightening_principle}
    Let $\{\ell_i\}_{i\in I}$ be a locally-finite family of essential lines on $\R^2\setminus \Z$ that is pairwise disjoint and pairwise non-isotopic. Then, there exists an isotopy $(h_t)_{t\in [0,1]}$ on $\R^2\setminus \Z$ from the identity that satisfies $h_1(\ell_i) = \ell_i^\geo$ for every $i\in I$.
\end{lemma}

Due to Lemma \ref{lemma:straightening_principle}, every globally essential proper multiline $\Lambda = \prod_{i\in I}\ell_i$ is globally isotopic to $\Lambda^\geo = \prod_{i\in I}\ell_i^\geo$, which is also a proper multiline, called the \textit{geodesic representative} of $\Lambda$.

\begin{remark}\label{remark:straightening_principle}
    In Lemma \ref{lemma:straightening_principle}, we state a version of the Straightening Principle sufficient for our purposes. 
    However, in particular for the proof of Proposition \ref{lemma:simplification_wandering_leaves}, we need to remark that the isotopy $(h_t)_{t\in [0,1]}$ given by the Straightening Principle can be taken to not move lines which already coincide with their geodesic representatives. 
\end{remark}

\subsection{Homotopy Brouwer Theory}\label{sec:homotopy_brouwer_theory}
Let \( f:\mathbb{R}^2 \longrightarrow \mathbb{R}^2 \) be a \textit{Brouwer homeomorphism}, that is, an orientation-preserving and fixed-point free homeomorphism of the plane.
The orbit of a point \( x \in \mathbb{R}^2 \) under \( f \) is defined as the set
$ \mathcal{O}(f, x) = \{f^n(x) \mid n \in \mathbb{Z}\} \subset \mathbb{R}^2.$

Let $\OO = \{\O_{\sspc 1},\ldots,\O_{\sspc r}\}$ be a finite collection of distinguished orbits of $f$.  By abuse of notation, we may denote $\OO := \bigcup_{i=1}^r\O_{\sspc i}\subset \R^2$.
The isotopy class of $f$ relative to the set $\OO$ is denoted by $\class{f,\sspc \OO}$, and it is referred to as a \emph{Brouwer mapping class relative to $r$-orbits}. Alternatively, $\class{f,\sspc \OO}$ can be defined as the isotopy class of the map $f\vert_{\R^2\setminus \OO}$ on the surface $\R^2\setminus \OO$, which is homeomorphic to $\R^2\setminus \Z$, because $\OO$ is an infinite closed discrete subset of $\R^2$ (see \cite[Proposition 1.5]{BeguinCrovisierLeRoux2020}).
This alternative viewpoint allows us to see Brouwer mapping classes as elements of $\mcg(\R^2\setminus \Z)$.

The main objective of Homotopy Brouwer Theory is to study the dynamics of $\class{f,\sspc \OO}$, particularly, to classify up to conjugacy all Brouwer mapping classes relative to $r$-orbits for a certain integer $r>0$. To be more precise, we say that $\class{f,\sspc \OO}$ and $\class{g,\OO^\pp}$ are conjugate if there exists an orientation-preserving homeomorphism $h:\R^2\longrightarrow \R^2$ that satisfies
$$ h(\OO^\pp)=\OO \quad \text{ and } \quad \bigl[\spc f,\sspc \OO\spc\bigr] = \bigl[\spc h\circ g \circ h^{-1},\sspc \OO\spc\bigr].$$
We note that, our notion of conjugacy follows the one considered by Bavard in \cite{BAVARD_2017}, which is the most natural one when seeing Brouwer mapping classes as elements of $\mcg(\R^2\setminus \Z)$. It is worth noting that, in \cite{handel99}, Handel also consider conjugacies by orientation-reversing homeomorphisms, and in \cite{LEROUX_2017}, Le Roux consider orientation-preserving conjugacies that preserve the indexation of the orbits. These differences are not a problem, as the statements of the result can be adapted depending on the context.

Although occasionally mentioned to highlight parallels with the classical framework of Homotopy Brouwer Theory, this work does not rely on homotopy translation arcs, homotopy Brouwer lines, or even the machinery of Fitted Families introduced in \cite{handel99}. Consequently, we do not provide their formal definitions in this preliminary section, as they are unnecessary for the development of our arguments and appear here only as points of reference.

In what follows we present the main classification results of Brouwer mapping classes.

The first classification theorem was presented by Handel. The proof of this result is based on the machinery of homotopy translation arcs and fitted families (see \cite{handel99} for details).
Before we enunciate the result, denote $\Z_i:=\Z \times \{i\}$ and $\Z_{i,j}:=\Z \times \{i,j\}$, for any $i,j\in \Z$. 
Let $T$ be the horizontal translation on $\R^2$, that moves each point in $\Z_i$ one unit to the right, and let $R$ be the time-one map of a Reeb flow (as illustrated in the next figure), that moves each point in $\Z_1$ one unit to the right and each point in $\Z_2$ one unit to the left.

\vspace*{0.2cm}
\begin{figure}[h!]
    \center 
    \hspace*{-0.4cm}\begin{overpic}[width=5cm, height=3cm, tics=10]{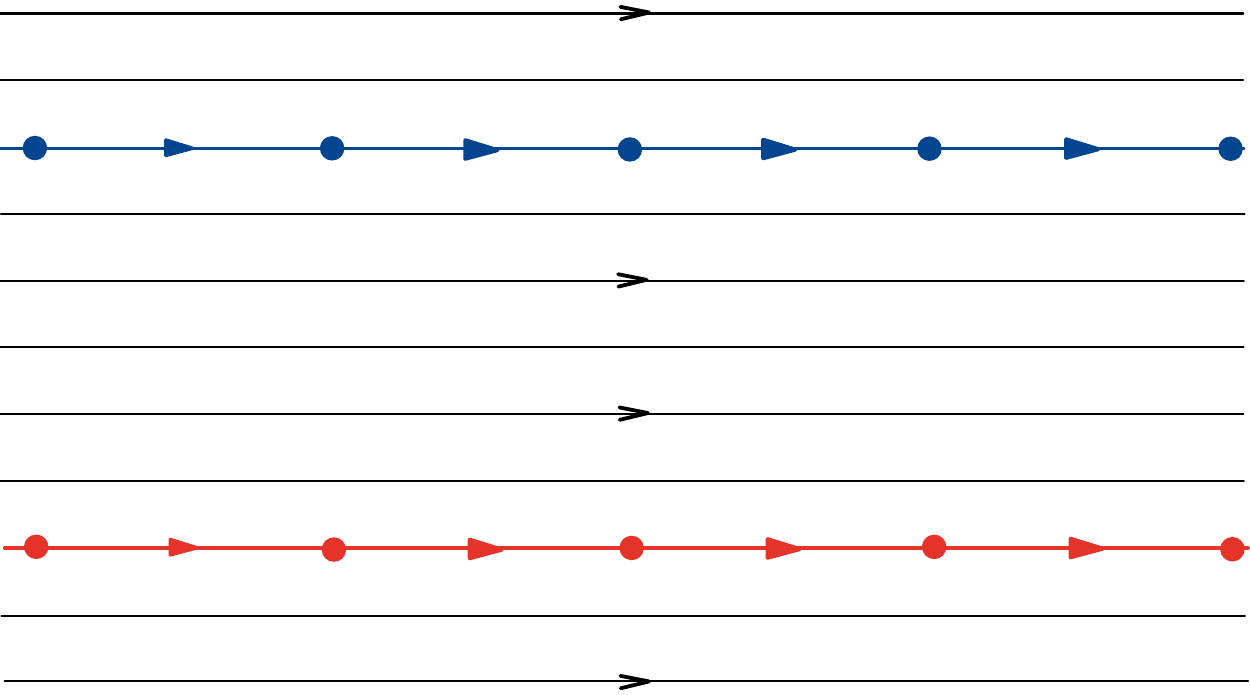}
        \put (-15,56) {{\color{black}\Large$\displaystyle T$}}
        \put (105,45) {{\color{myBLUE}\large$\displaystyle \Z_2$}}
        \put (105,9.5) {{\color{myRED}\large$\displaystyle \Z_1$}}
\end{overpic}
\hspace*{2.9cm}\begin{overpic}[width=5cm, height=3cm, tics=10]{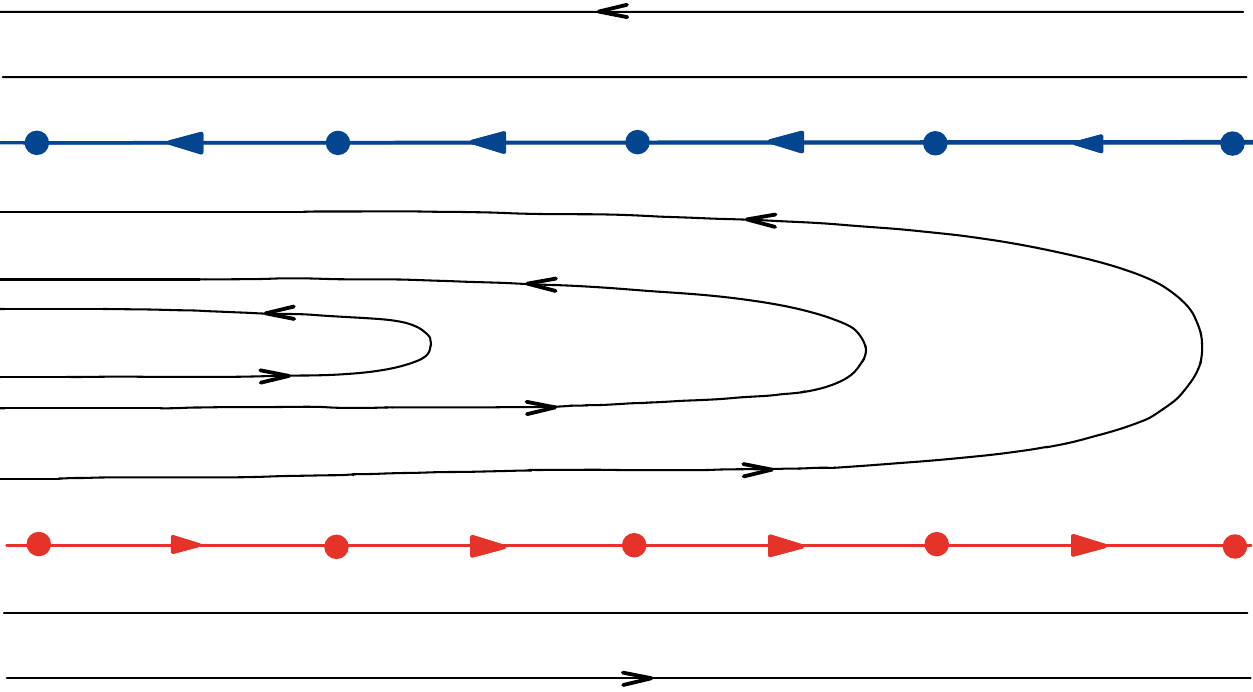}
    \put (-15,56) {{\color{black}\Large$\displaystyle R$}}
    \put (105,45) {{\color{myBLUE}\large$\displaystyle \Z_2$}}
        \put (105,9.5) {{\color{myRED}\large$\displaystyle \Z_1$}}
\end{overpic}
\end{figure}

\vspace*{-0.2cm}

\begin{theorem*}[Handel, {\cite{handel99}}]
\label{thm:handel_classification}
Every Brouwer mapping class $\class{f,\OO}$ relative to $r \leq 2$ orbits  admits a complete classification:
\begin{itemize}[leftmargin=1.2cm]
    \item[\textup{\textbf{(i)}}] If $r=1$, the class $\class{f,\sspc \OO}$ is conjugate to the translation class $\class{T,\spc\Z_1}$.
    \item[\textup{\textbf{(ii)}}] If $r=2$, the class $\class{f,\sspc \OO}$ is conjugate to one of the following classes
    $$ \class{T\sspc,\spc \Z_{1,2}}\sspc, \quad \class{T^{-1},\spc \Z_{1,2}}\sspc, \quad \class{R\sspc ,\spc \Z_{1,2}}\sspc,  \quad \class{R^{-1},\spc \Z_{1,2}}\sspc.$$
\end{itemize}
\end{theorem*}

 \mycomment{-0.1cm}
The Brouwer mapping class $\class{f,\sspc \OO}$ is said to be a \textit{flow-class} if it admits the time one map of a flow as a representative. Observe that, according to Handel's classification, we know that every Brouwer mapping class relative to at most two orbits is a flow class.

The next classification result came with \cite{LEROUX_2017}, where Le Roux classified all Brouwer mapping classes relative to three orbits, up to conjugacy. In particular, he shows that in the case relative to three orbits, every Brouwer mapping class is also a flow-class.

Soon after, Bavard introduced powerful new techniques in \cite{BAVARD_2017}, which enabled the classification of Brouwer mapping classes relative to four orbits and, possibly, paving the way  towards more general classifications. A striking consequence of her result is that, unlike the cases relative to one, two, and three orbits, there exists infinitely many conjugacy classes of Brouwer mapping classes relative to four orbits. Not only that, but also, some of these Brouwer mapping classes are not flow classes. Suggesting that, from four orbits onward, the behavior of Brouwer mapping classes may become significantly more intricate.

We do not include here the precise statements of Le Roux's and Bavard's results, as they are not yet directly recovered by the methods developed in this work.

\subsection{Foliated Brouwer Theory}\label{sec:foliated_brouwer}
Let \( f:\mathbb{R}^2 \longrightarrow \mathbb{R}^2 \) be a \textit{Brouwer homeomorphism}, that is, an orientation-preserving and fixed-point free homeomorphism of the plane.

The orbit of a point \( x \in \mathbb{R}^2 \) under \( f \) is defined as the set
$$ \mathcal{O}(f, x) = \{f^n(x) \mid n \in \mathbb{Z}\} \subset \mathbb{R}^2.$$
Sometimes, it may be useful to denote the set of orbits of \( f \) by $\orb$. As proved in \cite{Brouwer}, every orbit of $f$ is wandering and, thus, it has an empty limit set. This means that every orbit of $f$ is the image of a proper embedding of $\mathbb{Z}$ into the plane.

A \textit{Brouwer line} for \( f \) is an oriented topological line \( \lambda\) on \( \mathbb{R}^2 \) that satisfies
$$ f(L(\lambda)) \subset L(\lambda) \quad \text{ and } \quad f^{-1}(R(\lambda)) \subset R(\lambda),$$
where \( L(\lambda) \) and \( R(\lambda) \) denote the left and right sides of \( \lambda \). The classical \textit{Brouwer translation theorem} \cite{Brouwer} asserts that through every point of the plane passes a Brouwer line for \( f \).

\begin{theorem*}[Le Calvez, \cite{lec1}] There exists an oriented topological foliation \(\F\) of the plane such that every leaf $\phi \in \F$ is a Brouwer line of $f$.
\end{theorem*}

Any such foliation of the plane by Brouwer lines of \( f \) is called a \textit{transverse foliation} of \( f \). It is worth noting that a Brouwer homeomorphism admits several transverse foliations, which  may exhibit significantly different structures and non-homeomorphic leaf spaces, for instance.

Let us fix a transverse foliation $\F$ of a Brouwer homeomorphism \( f \). In this contex, one can show that every point \( x \in \mathbb{R}^2 \) may be connected to its image \( f(x) \) by a path $\gamma_x:[0,1]\longrightarrow \R^2$ that is positively transverse to the foliation $\F$, in the sense that $\gamma_x$ intersects each leaf of $\phi \in \F$ at most once, always from right $R(\phi)$ to left $L(\phi)$. By concatenating the paths $\gamma_{f^n(x)}$ for all integers $n\in \Z$, we obtain a positively transverse path 
$$ \Gamma := \prod_{n\in \Z} \gamma_{f^n(x)},$$
that connects the entire orbit $\mathcal{O}(f,x)$, called a \textit{transverse trajectory} of $\O(f,x)$ (with respect to the foliation $\F$). Since the transverse foliation $\F$ is usually fixed, we often omit its mention when referring to transverse trajectories.
It is worth noting that an orbit $\O(f,x)$ admits uncountably many transverse trajectories, which depend on the choice of paths $\{\gamma_{f^n(x)}\}_{n \in \Z}$.

Indenpendently of the choice of transverse paths $\{\gamma_{f^n(x)}\}_{n \in \Z}$, the leaves that are intersected by the transverse trajectory $\Gamma_\O$ of $\O=\O(f,x)$ are determined by the orbit itself.  In fact, the set $C_\O \subset \F$ of leaves \textit{crossed} by the orbit $\O$ can be equivalently defined as
$$ C_\O = \{\phi \in \F \mid \phi \cap \Gamma_\O \neq \varnothing\} = \{\phi \in \F \mid \O \text{ intersects } R(\phi) \text{ and } L(\phi)\}.$$

We remark that, if the choice of paths $\{\gamma_{f^n(x)}\}_{n \in \Z}$ is locally-finite on the plane, then the transverse trajectory $\Gamma_\O$ is a topological line, and it is called a \textit{proper transverse trajectory} of the orbit $\O$. The converse is also true.

\textbf{ Generic transverse foliations:}\
A transverse foliation $\F$ of \( f \) is said to be \textit{generic} for $\OO$ if each leaf $\phi \in \F$ intersects at most one orbit in $\OO$. This condition is generic, as we explain below.

Assume that a leaf $\phi \in \F$ intersects two orbits in $\OO$, at points $x,x^\pp\in \phi$. Let $B$ be a closed ball centered at $x$ that is sufficiently small so that $B \cap \OO = \{x\}$ and $f(B) \cap B = \varnothing.$ Observe that, we can perturbe $\F$ inside the ball $B$ in such a way that $x$ no longer lies on the leaf passing through $x^\pp$. Since $B$ satisfies $f(B) \cap B = \varnothing$, the leaves of $\F$ continue to be Brouwer lines of \( f \) after the perturbation. Now, since $\OO$ is a finite collection of orbits of $f$, and each orbit of $f$ is a closed discrete subset of $\R^2$, we have that for any compact set $K\subset \R^2$, the set $K\cap \OO$ is finite. Therefore, by repeating the perturbation process described above on a locally-finite family of disjoint balls covering $\OO$, we can obtain a transverse foliation $\F'$ of \( f \) that is generic for $\OO$.

\vspace*{0.3cm}
\textbf{ Proper transverse trajectories:}\ In a previous work \cite{schuback1}, we established that every orbit of $f$ admits a proper transverse trajectory with respect to the transverse foliation $\F$. Moreover, in the case where $\F$ is generic for $\OO$, we demonstrated the following structural result.

\begin{theorem}\label{theorem:1B-intro}
    Let \(\F\) be a transverse foliation of $f$, which we assume to be generic for \(\OO\). Then, there exists a family $\bigl\{\Gamma_{\O}\bigr\}_{\O \in \OO}$ of proper transverse trajectories of the orbits in $\OO$ that satisfy
    \mycomment{-0.12cm}
    $$\Gamma_{\O} \cap \Gamma_{\O^\pp} \cap \overline{\rule{0pt}{3.6mm}L(\phi)} = \varnothing, \quad \forall \sspc \O,\O^\pp \in \OO,\  \O \neq \O^\pp.$$
    % \mycomment{-0.22cm}
    % \noindent Moreover, the order at which the intersection points $p_\O \in \Gamma_{\O} \cap \phi$ appear along the leaf $\phi$ is compatible with the relation $\lesssim_L$, in the sense that $p_\O < p_{\O^\pp}$ along the leaf $\phi$ only if $\O \lesssim_L \O^\pp$. 
\end{theorem}

\vspace*{-0.3cm}

\begin{figure}[h!]
    \center\begin{overpic}[width=7.2cm, height=3.6cm, tics=10]{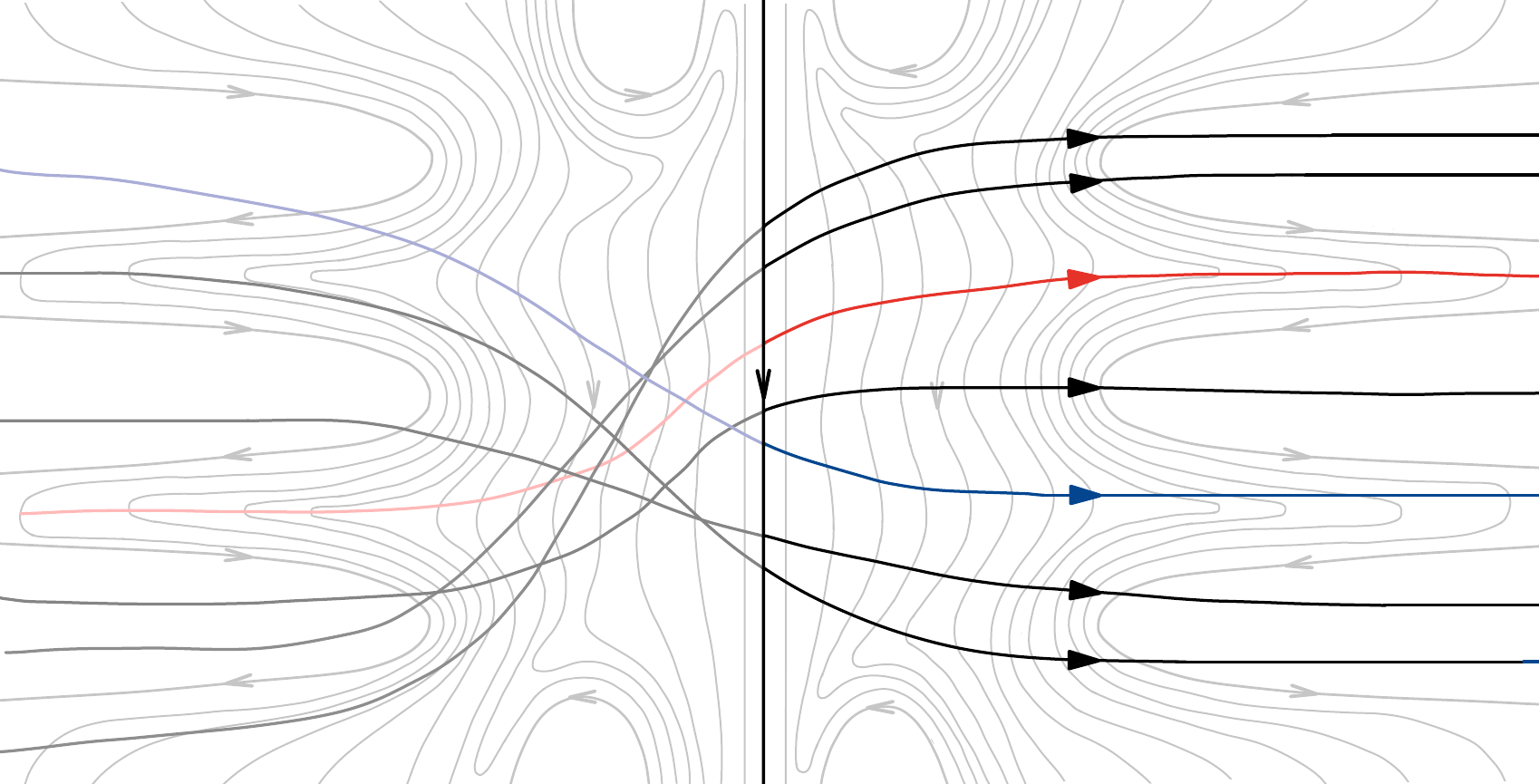}
        \put (44.5,-6) {\colorbox{white}{\color{black}\large$\displaystyle\  \phi \spc$}}
        % \put (58,15) {\colorbox{white}{$\rule{0cm}{0.3cm}\ \  $}}
         \put (103.5,15.5) {{\color{myBLUE}\large$\displaystyle \Gamma_{\O^\pp} $}}
         %\put (57.6,31.1) {\colorbox{white}{\color{myDARKGRAY}\large$\rule{0cm}{0.3cm}\   \  $}}
         \put (103.5,30) {{\color{myRED}\large$\displaystyle \Gamma_\O$}}
         \put (-25,20) {{\color{myRED}\large$\displaystyle \O \ $}{\color{black}\large$\displaystyle \lesssim_L$}{\color{myBLUE}\large $\spc\displaystyle \O^\pp$}}
\end{overpic}
\end{figure}

\vspace*{0.1cm}
A similar result holds if we replace $\overline{\rule{0pt}{3.6mm}L(\phi)}$ and $\lesssim_L$ for their right analogues $\overline{\rule{0pt}{3.6mm}R(\phi)}$ and $\lesssim_R$. 

\addtocontents{toc}{\protect\setcounter{tocdepth}{2}}

\section{A foliated-homotopy framework}\label{ch:geodesic_framework}

\subsection{Transverse geodesic laminations} Let $\F$ be a transverse foliation of $f$, and define %of leaves in $\F$ that are essential lines on $\R^2\setminus \OO$, given by
 \mycomment{-0.16cm}
\begin{equation*}
    \G := \{\spc\phi^\geo \mid \phi\in \F \text{ is an essential line on } \R^2\setminus \OO\sspc\}.
\end{equation*}
By the minimal intersection property of geodesics, since the leaves of $\F$ are pairwise disjoint, the geodesic representatives in $\G$ must either coincide or be disjoint. This implies that $\G$ is a family of pairwise disjoint geodesic lines on $\R^2\setminus \OO$. Thus, viewing $\G$ as a subset of $\R^2$ via the identification $\G = \bigcup_{\lambda \in \G} \lambda$, the closure $\overline{\rule{0cm}{0.31cm}\sspc\G}$ is a geodesic lamination of $\R^2 \setminus \OO$ (Section \ref{sec:hyperbolic_geometry_R2Z}). Although the difference $\overline{\rule{0cm}{0.31cm}\sspc\G}\setminus \G$ may be non-empty, by abuse of notation, we refer to  $\G$ as a \emph{transverse geodesic lamination}, and its elements as \emph{geodesic leaves}, associated to $\class{f,\OO}$.

\begin{figure}[h!]
    \center 
    \hspace*{-0.2cm}\begin{overpic}[width=6.6cm, height=3.7cm, tics=10]{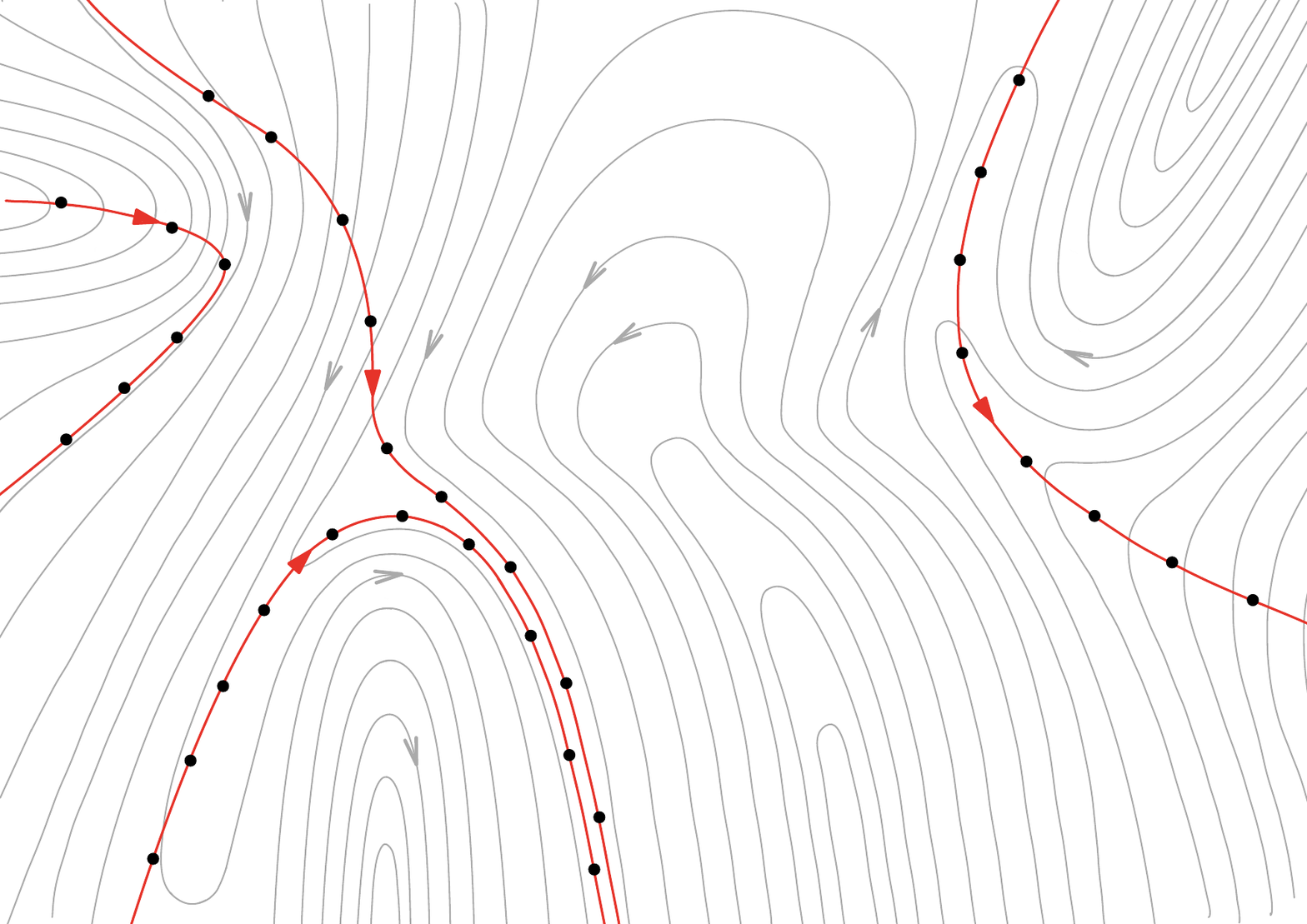}\put (91,1) {\colorbox{white}{\color{gray}\LARGE$\displaystyle \rule{0cm}{0.46cm}\F$}}
\end{overpic}
\hspace*{1.2cm}\begin{overpic}[width=6.6cm, height=3.7cm, tics=10]{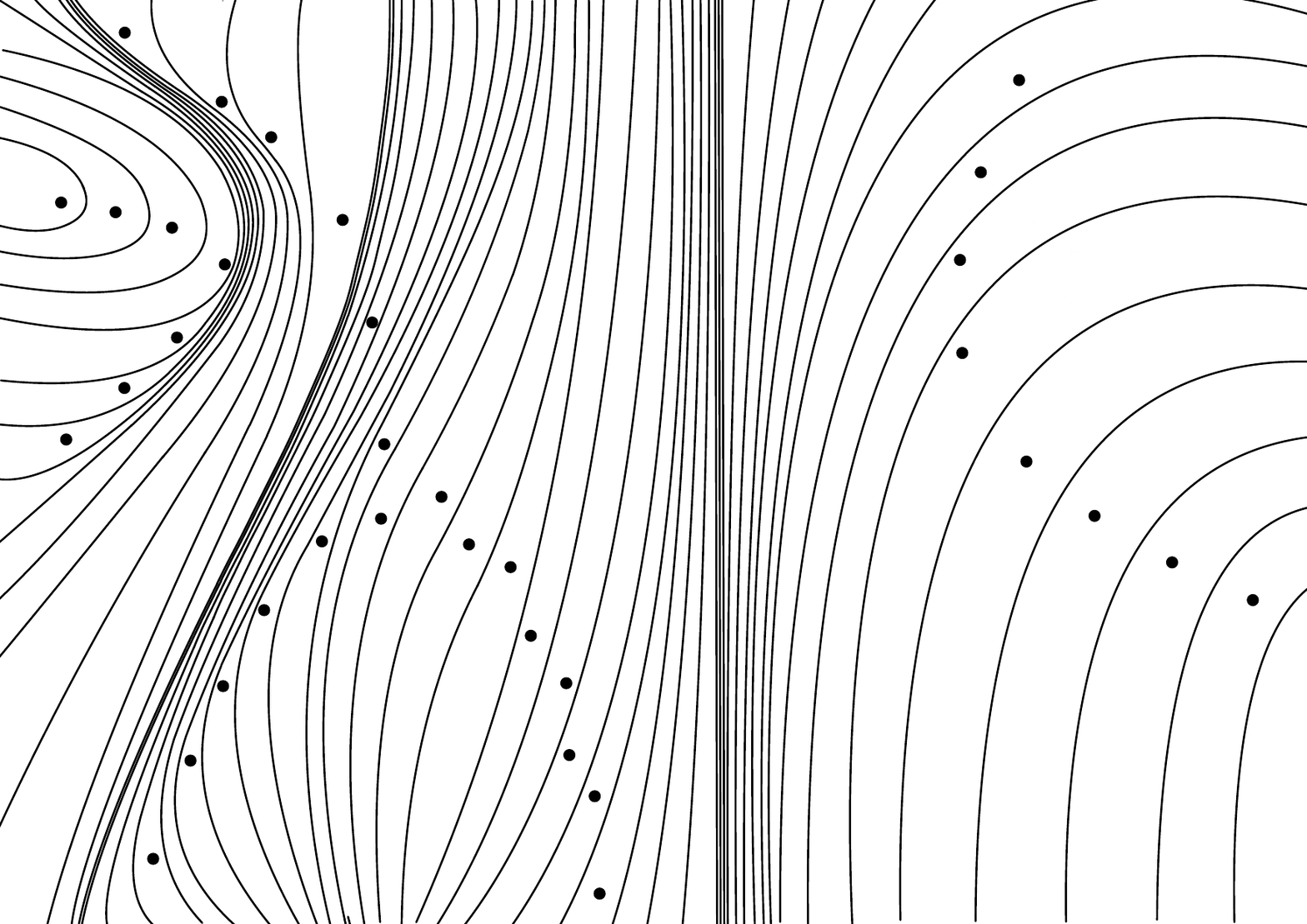}
    \put (90,2) {\colorbox{white}{\color{black}\LARGE$\displaystyle \rule{0cm}{0.48cm}\sspc \G$}}
\end{overpic}
\end{figure}

Observe that, for any geodesic leaf $\lambda \in \G$ and any leaf $\phi \in \F$ such that $\phi^\geo = \lambda$, we have 
 \mycomment{-0.16cm}
$$ R(\lambda) \cap \OO = R(\phi) \cap \OO \quad \text{ and } \quad L(\lambda) \cap \OO = L(\phi) \cap \OO.$$
From this property, we observe that $\G$ replicates several structural features of the foliation $\F$ with respect to $\OO$. By fixing a basepoint $x_\O \in \O$, for each $\O \in \OO$, we have the following:

\vspace*{0.1cm}
\begin{itemize}[leftmargin=0.9cm]
    \item[(1)] An orbit $\O \in \OO$ is said to cross a geodesic leaf $\lambda \in \G$ if it intersects both $R(\lambda)$ and $L(\lambda)$. Moreover, if $\O$ crosses $\lambda$, then there exist $N\in \Z$ such that
     \mycomment{-0.2cm}
    $$f^m(x_\O) \in R(\lambda) \quad \text{and} \quad  f^n(x_\O) \in L(\lambda)\sspc, \quad\  \forall m\leq N<n.$$
    In this context, we say that $\lambda$ is crossed by the orbit $\O$ at time $N$.

    \vspace*{0.1cm}
    \item[(2)] For any $\O \in \OO$ and any $N\in \Z$, there exists $\lambda \in \G$ crossed by $\O$ at time $N$.
    
    \vspace*{0.1cm}
    \item[(3)] For any two distinct orbits $\O,\O^\pp \in \OO$, one of the following alternatives holds:
    \begin{itemize}[leftmargin=1.1cm]
            \item[$\sbullet$] Either there exists a geodesic leaf $\lambda \in \mathcal G$ crossed by both orbits $\O$ and $\O^\pp$.%, meaning that each $\O$ and $\O^\pp$ intersect both sides of $\lambda$.
            \item[$\sbullet$] Or there exists a geodesic leaf $\lambda \in \mathcal G$ that separates $\O$ from $\O^\pp$, in the sense that the orbits $\O$ and $\O^\pp$ are contained in different connected components of $\R^2\setminus \lambda$.
        \end{itemize}
\end{itemize}
 \mycomment{0.2cm}
At last, if the foliation $\F$ is generic with respect to $\OO$, meaning that each leaf of $\F$ intersects at most one orbit in $\OO$, then we have a stronger version of property (2):

\vspace*{0.1cm}
\begin{itemize}[leftmargin=1.1cm]
    \item[(2*)] For any two distinct points $x,x'\in \OO\subset\R^2$, there exists a geodesic leaf $\lambda \in \G$ such that the points $x$ and $x'$ are contained in different connected components of $\R^2\setminus \lambda$.
\end{itemize}

The geodesic leaves in $\G$ can be classified into two types, according to their interaction with the orbits in $\OO$. Motivated by this, we decompose the transverse geodesic lamination into two families $\G= \mathcal P \sqcup \mathcal S$, which are defined as follows:

\vspace*{0.1cm}
\begin{itemize}[leftmargin=1.5cm]
    \item[$\sbullet$]  Every geodesic leaf $\lambda \in \P$ is crossed by at least one orbit in $\OO$.
    \vspace*{0.1cm}
    \item[$\sbullet$] Every geodesic leaf $\sigma \in \mathcal S$ is not crossed by any orbit in $\OO$.
\end{itemize}

\vspace*{0.1cm}
\noindent Throughout the text, we often refer to elements of $\mathcal P$ as \textit{pushing geodesic leaves}. Meanwhile, elements of $\mathcal S$ are called \textit{separating geodesic leaves}, and are represented by the letter $\sigma \in \mathcal S$. 
 \mycomment{0.2cm}
\begin{figure}[h!]
    \center 
    \hspace*{-0.3cm}\begin{overpic}[width=4cm, height=2.8cm, tics=10]{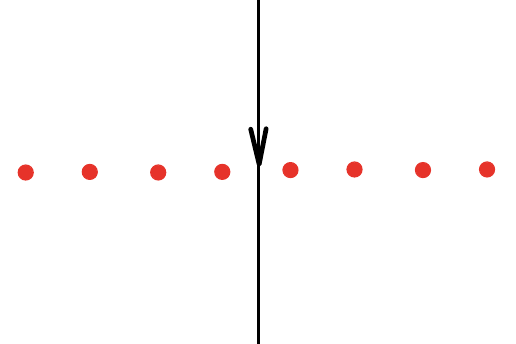}
        \put (56,56) {{\color{black}\Large$\displaystyle \lambda \in \mathcal P$}}
        \put (-14.5,31) {{\color{myRED}\large$\displaystyle \O_{\sspc i}$}}
        
\end{overpic}
\hspace*{3cm}\begin{overpic}[width=4cm, height=2.8cm, tics=10]{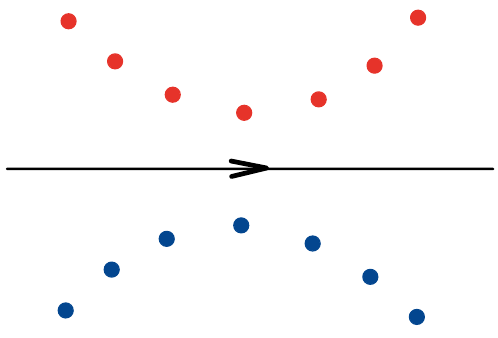}
    \put (102,30.5) {{\color{black}\Large$\displaystyle \sigma \in \mathcal S$}}
    \put (-3,54) {{\color{myRED}\large$\displaystyle \O_{\sspc i}$}}
    \put (-3,10) {{\color{myBLUE}\large$\displaystyle \O_{\sspc j}$}}
\end{overpic}
\end{figure}

 \vspace*{-0.2cm}
Note that, each geodesic leaf $\sigma \in \mathcal S$ separates at least two orbits in $\OO$, meaning that each of its sides $R(\sigma)$ and $L(\sigma)$ contain at least one orbit in $\OO$, since otherwise $\sigma$ would not be essential on $\R^2\setminus \OO$ and this contradicts the fact that every geodesic line on $\R^2\setminus \OO$ is essential. As a consequence, $\mathcal S$ has at most $r-1$ elements, where $r$ denotes the number of orbits in $\OO$.

Let us denote by $f_\geo$ the action of the homeomorphism $f\vert_{\R^2\setminus \OO}$ on the set of geodesics on the surface $\R^2\setminus \OO$ (see Section \ref{sec:hyperbolic_geometry_R2Z} for details). This action on geodesics is invariant under isotopy on $\R^2\setminus \OO$ and, therefore, studying the action of $f_\geo$ on geodesics allows us to understand the behavior of the Brouwer mapping class $\class{f,\OO}$. We recall that $f_\geo$ behaves coherently with geodesic representatives, meaning that, for any essential line $\ell$ on $\R^2\setminus \OO$, 
 \mycomment{-0.11cm}
\begin{equation*}
    f_\geo(\ell^\geo) = (f(\ell))^\geo.
\end{equation*}

The following lemma describes the fundamental dynamical properties of the geodesic leaves under the action of $f_\geo$, highlighting the differences between pushing and separating.

 \mycomment{-0.01cm}
\begin{lemma}\label{lemma:geodesic_leaves}
    The dynamical behavior of geodesic leaves under the action of $f_\geo$ is as follows:

    \vspace*{0.1cm}
    \begin{itemize}[leftmargin=1cm]
        \item[\textup{\textbf{(i)}}] Every pushing geodesic leaf $\lambda \in \P$ satisfies the following property
         \mycomment{-0.12cm}
        $$ L(f^{n+1}_\geo (\lambda)) \subset L(f^{n}_\geo (\lambda)) \quad \forall n \in \Z.$$

         \mycomment{-0.27cm}
       \noindent Moreover, if $\lambda$ is crossed by exactly $k>0$ orbits in $\OO$, then the surface
        \mycomment{-0.11cm}
       $$ S(f^{n}_\geo (\lambda)\sspc,\spc f^{n+1}_\geo (\lambda)) = L(f^{n}_\geo (\lambda)) \cap R(f^{n+1}_\geo (\lambda))\setminus \OO,$$

         \mycomment{-0.25cm}
       \noindent is homeomorphic to a $k$-punctured plane.

        \vspace*{0.1cm}
        \item[\textup{\textbf{(ii)}}] Every separating geodesic leaf $\sigma \in \mathcal S$ is fixed by $f_\geo$, that is,
        $f_\geo(\sigma) = \sigma.$
    \end{itemize}
\end{lemma}

\begin{proof} We begin by proving item (i).
Consider a geodesic leaf $\lambda \in \P$ and a leaf $\phi \in \F$ with geodesic representative $\phi^\geo = \lambda$.  By the definition of $\P$, there are $k>0$ orbits in $\OO$ crossing $\lambda$, and thus crossing $\phi$ as well. Thus, the surface $ S(\phi\sspc,\sspc f(\phi)) =L(\phi) \cap R(f(\phi))\setminus \OO$ is homeomorphic to a $k$-punctured plane. This implies that $\phi$ and $f(\phi)$ are not isotopic on $\R^2\setminus \OO$, therefore they have distinct geodesic representatives $\phi^\geo = \lambda$ and $f(\phi)^\geo = f_\geo(\phi^\geo) = f_\geo(\lambda)$. In addition, since $\phi$ and $f(\phi)$ are disjoint, the geodesics $\lambda$ and $f_\geo(\lambda)$ must also be disjoint, since geodesic representatives minimize intersections in their isotopy classes.

According to the Straightening principle (see Lemma \ref{lemma:straightening_principle}), there exists an isotopy $(h_t)_{t \in [0,1]}$ on $\R^2\setminus \OO$ from the identity $h_0 = \textup{id}_{\R^2\setminus \OO}$ that satisfies
 \mycomment{-0.05cm}
$$h_1(\phi) = \lambda \quad \text{ and } \quad h_1(f(\phi)) = f_\geo(\lambda),$$

 \mycomment{-0.15cm}
\noindent not only as sets, but also with matching orientations. This first implies that
 \mycomment{-0.05cm}
$$ L(\lambda) \cap R(f_\geo(\lambda))\setminus \OO = h_1(L(\phi) \cap R(f(\phi))\setminus \OO),$$ 

 \mycomment{-0.15cm}
\noindent which is homeomorphic to a $k$-punctured plane. Secondly, since $h_1$ maps $\phi$ and $f(\phi)$ into their geodesic representatives with the correct orientation, we have that
 \mycomment{-0.05cm}
$$ h_1(L(\phi)) = L(\lambda) \quad \text{ and } \quad h_1(L(f(\phi))) = L(f_\geo(\lambda)).$$ %= f_\geo(L(\lambda)).$$

 \mycomment{-0.15cm}
\noindent 
Finally, since the leaf $\phi$ satisfies $L(f(\phi))\subset L(\phi)$, we can also conclude that $L(f_\geo (\lambda)) \subset L(\lambda)$. The same reasoning can be applied, for any $n\in \Z$, to the iterates $f^n(\phi)$ and $f^{n+1}(\phi)$, and this leads to the proof of item (i).
We will now prove item (ii).

Consider a geodesic leaf $\sigma \in \mathcal S$ and a leaf $\phi \in \F$ with geodesic representative $\phi^\geo = \sigma$.  By the definition of $\mathcal S$, no orbit in $\OO$ crosses $\sigma$, and thus no orbit crosses the leaf $\phi$ as well. Thus, every orbit in $\OO$ is contained in either $L(\phi)$ or $R(\phi)$. Since each orbit is $f$-invariant,
 \mycomment{-0.05cm}
$$ \O \subset L(\phi) \iff f(\O) \subset L(f(\phi))\sspc, \quad \forall \O \in \OO.$$

 \mycomment{-0.15cm}
\noindent 
This implies that the surface $ S(\phi\sspc,\spc f(\phi)) = L(\phi) \cap R(f(\phi))\setminus \OO$ is homeomorphic to a plane.
By the Homma-Schoenflies theorem, we may suppose up to conjugacy that $\phi$ and $f(\phi)$ are vertical line oriented downwards. This allows us to perform an isotopy $(h_t)_{t\in [0,1]}$ on $\R^2\setminus \OO$ from the identity map that moves points horizontally until we get $h_1(\phi) = f(\phi).$ Consequently, the lines $\phi$ and $f(\phi)$ are isotopic on $\R^2\setminus \OO$, and thus their geodesic representatives, which are given by $\sigma$ and $f_\geo(\sigma)$, must coincide. This concludes the proof of item (ii).
\end{proof}

As a conclusion, we have proved throughout this section the following structural theorem, which is a complete version of Theorem \ref{thmx:structural_thm} stated in the introduction.

\pagebreak

\addtocounter{thmx}{-3}

\begin{thmx} Every Brouwer mapping class $\class{f,\OO}$ admits a family $\G$ formed by pairwise disjoint geodesic planar lines on $\R^2\setminus \OO$ which is decomposable into two subfamilies $\G = \mathcal S\sspc \sqcup \sspc\mathcal P$ in such a way that the following properties hold

    \vspace*{0.1cm}
    \begin{itemize}[leftmargin=1.2cm]
        \item[\textup{\textbf{(i)}} ] Every geodesic in $\G$ is isotopic to a Brouwer line of $f$ on the surface $\R^2\setminus \OO$.
        
        \vspace*{0.05cm}
        \item[\textup{\textbf{(ii)}}$\spc$] Every connected component of $(\R^2\setminus \OO)\setminus \G$ is homeomorphic to $\R^2$ or $\R^2\setminus \{0\}.$
        
        \vspace*{0.05cm}
        \item[\textup{\textbf{(iii)}}] For any two distinct orbits $\O,\O^\pp \in \OO$, we have the following dichotomy:
        
        \vspace*{0.05cm}
        \begin{itemize}
            \item Either there exists a geodesic $\lambda \in \mathcal P$ crossed by both orbits $\O$ and $\O^\pp$.%, meaning that each $\O$ and $\O^\pp$ intersect both sides of $\lambda$.
            
            \vspace*{0.05cm}
            \item Or there exists a geodesic $\sigma \in \mathcal S$ that separates $\O$ from $\O^\pp$.%., meaning that $\O$ and $\O^\pp$ are contained in different sides of $\sigma$.
        \end{itemize}

        \vspace*{0.0cm}
        \item[\textup{\textbf{(iv)}}] Every geodesic $\sigma \in \mathcal S$ is called a \emph{separating geodesic} and satisfies:
        
        \vspace*{0.05cm}
    \begin{itemize}[leftmargin=1.3cm]
            \item[\textup{\textbf{(S1)}}] There are two orbits in $\OO$ that are separated by $\sigma$.
            
            \vspace*{0.05cm}
            \item[\textup{\textbf{(S2)}}] There are no orbits in $ \OO$ that cross $\sigma$.
            
            \vspace*{0.05cm}
            \item[\textup{\textbf{(S3)}}] The action of $f_\geo$ fixes $\sigma$, meaning that $f_\geo(\sigma)=\sigma$.
    \end{itemize}

    \vspace*{0.05cm}
    \noindent In particular, $\mathcal S$ has at most $r-1$ elements, where $r$ is the number of orbits in $\OO$.

    \vspace*{0.1cm}

        \item[\textup{\textbf{(v)}}$\spc$] Every geodesic $\lambda \in \mathcal P$ is called a \emph{pushing geodesic} and satisfies:
        
        \vspace*{0.05cm}
    \begin{itemize}[leftmargin=1.3cm]
            \item[\textup{\textbf{(P1)}}] There are always at least one orbits in $\OO$ that crosses $\lambda$.
            
            \vspace*{0.05cm}
            \item[\textup{\textbf{(P2)}}] For any $n\in \Z$, the geodesics $f^{n}_\geo(\lambda)$ and $f^{n+1}_\geo(\lambda)$ are disjoint and satisfy
            $$\overline{\rule{0pt}{3.5mm}L(f^{n+1}_\geo(\lambda))}\subset L(f^{n}_\geo(\lambda)).$$
            % $$\overline{\rule{0pt}{3.5mm}L(f^{n+1}_\geo(\lambda))}\subset L(f^{n}_\geo(\lambda)) \quad \text{ and } \quad \overline{\rule{0pt}{3.4mm}R(f^{n}_\geo(\lambda))}\subset L(f^{n+1}_\geo(\lambda)),$$

            \vspace*{0.05cm}
            \item[\textup{\textbf{(P3)}}] If there are $k>0$ orbits in $\OO$ crossing $\lambda$, then for any $n\in \Z$ the set
            $$L(f^{n}_\geo(\lambda)) \cap R(f^{n+1}_\geo(\lambda)) \setminus \OO$$ is homeomorphic to a $k$-punctured plane.
    \end{itemize}
    \end{itemize}
\end{thmx}

\begin{proof}
    Item (i) follows immediately from the fact that every geodesic leaf in $\G$ is the geodesic representative of a leaf in $\F$, which is a Brouwer line of $f$. Item (ii) is a direct consequence of property (2*), which holds if we assume that the foliation $\F$ is generic with respect to $\OO$.\break  Item (iii) is a consequence of property (3). Finally, items (iv) and (v) follow from the classification into pushing and separating geodesic leaves $\G = \mathcal S\sspc \sqcup \sspc \mathcal P$, and from the description of their dynamical behavior in Lemma \ref{lemma:geodesic_leaves}.
\end{proof}

\subsection{The pushing lemma}\label{sec:pushing-lemma}

Two geodesic leaves \(\lambda,\lambda' \in \P\) are said to be \emph{pushing-equivalent} if they are equal, or if they bound a subsurface of $\R^2\setminus\OO$ that is homeomorphic to a finitely punctured plane. Equivalently, if $\lambda$ and $\lambda'$ are crossed by exactly the same orbits in $\OO$ and, moreover, if no orbit of $\OO$ is contained in the connected component of $\R^2\setminus(\lambda\cup\lambda')$ bounded by $\lambda$ and $\lambda'$.
The pushing-equivalence relation defines an equivalence relation on $\P$, denoted by \(\sim\).
 \mycomment{0.25cm}

\begin{figure}[h!]
    \center 
    \hspace*{-0.1cm}\begin{overpic}[width=3.4cm, height=3.6cm, tics=10]{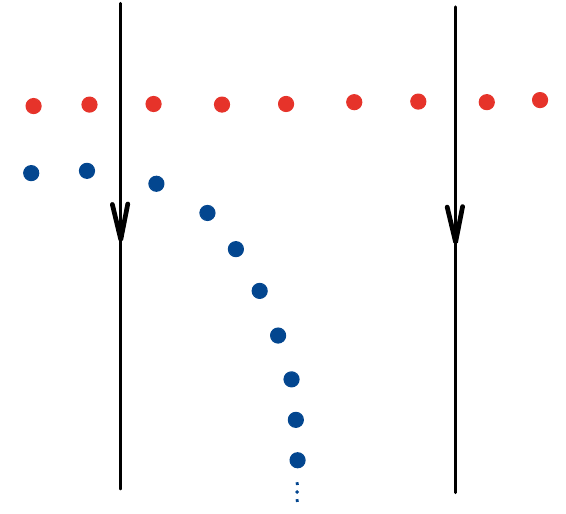}
        \put (34,-15) {{\color{black}\large$\displaystyle \lambda\not\sim\lambda'$}}
        \put (9,95) {{\color{black}\large$\displaystyle \lambda$}}
        \put (66,95) {{\color{black}\large$\displaystyle \lambda'$}}
        \put (-14.5,78.5) {{\color{myRED}\large$\displaystyle \O_{\sspc i}$}}
        \put (-14.5,61) {{\color{myBLUE}\large$\displaystyle \O_{\sspc j}$}}
\end{overpic}
\hspace*{1.5cm}\begin{overpic}[width=3.4cm, height=3.6cm, tics=10]{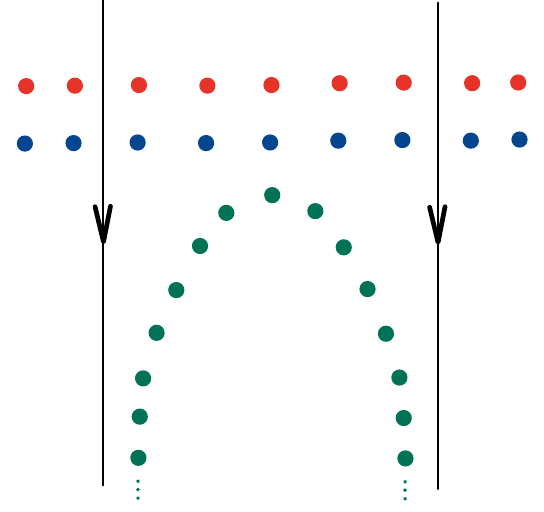}
    \put (9,95) {{\color{black}\large$\displaystyle \lambda$}}
    \put (34,-15) {{\color{black}\large$\displaystyle \lambda\not\sim\lambda'$}}
        \put (66,95) {{\color{black}\large$\displaystyle \lambda'$}}
        \put (-14.6,81.5) {{\color{myRED}\large$\displaystyle \O_{\sspc i}$}}
        \put (-14.6,67.2) {{\color{myBLUE}\large$\displaystyle \O_{\sspc j}$}}
        \put (31,25) {{\color{myGREEN}\large$\displaystyle \O_{\sspc k}$}}
\end{overpic}
\hspace*{1.55cm}\begin{overpic}[width=3.3cm, height=3.6cm, tics=10]{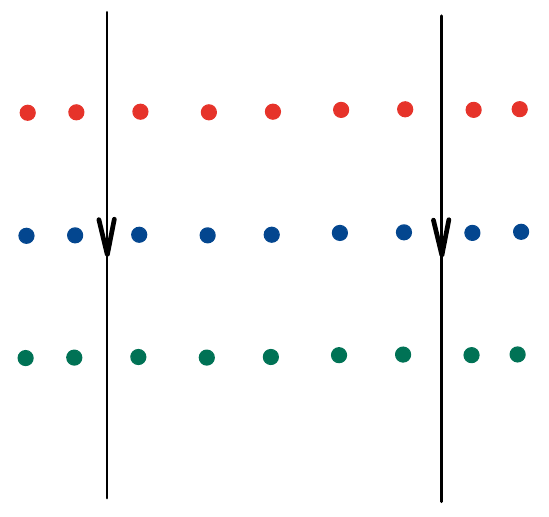}
    \put (9,95) {{\color{black}\large$\displaystyle \lambda$}}
    \put (34,-15) {{\color{black}\large$\displaystyle \lambda\sim\lambda'$}}
        \put (66,95) {{\color{black}\large$\displaystyle \lambda'$}}
        \put (-14,78) {{\color{myRED}\large$\displaystyle \O_{\sspc i}$}}
        \put (-14,53) {{\color{myBLUE}\large$\displaystyle \O_{\sspc j}$}}
        \put (-14,28) {{\color{myGREEN}\large$\displaystyle \O_{\sspc k}$}}
\end{overpic}
\end{figure}
 \mycomment{0.2cm}

The pushing-equivalent relation exhibits the following structural property:\\[1ex]
\indent Let $\lambda, \lambda^\pp \in \P$ be two pushing-equivalent geodesic leaves with \(L(\lambda') \subset L(\lambda)\), where the surface \(S(\lambda,\lambda') = L(\lambda) \cap R(\lambda')\) is homeomorphic to a $k$-punctured plane for some \(k\geq 1\).  Assuming that the foliation $\F$ that induces $\G$ is generic with respect to $\OO$, in the sense that no leaf in $\F$ intersects two or more points in the set $\OO$, we have a finite sequence $\{\lambda_i\}_{i=0}^m$ of pushing-equivalent geodesic leaves in \(\P\) such that, for any \(0 \leq i < m\), the following hold:
\begin{itemize}[leftmargin=1.2cm]
    \item[$\sbullet$] The sequence starts at \(\lambda_0 = \lambda\) and ends at \(\lambda_m = \lambda'\).
    \item[$\sbullet$] It holds the inclusion \(L(\lambda_{i+1}) \subset L(\lambda_i)\).
    \item[$\sbullet$] The surface \(S(\lambda_i, \lambda_{i+1}) = L(\lambda_i) \cap R(\lambda_{i+1})\) is homeomorphic to a $1$-punctured plane.
\end{itemize}

\vspace*{0.4cm}

\begin{figure}[h!]
    \center\begin{overpic}[width=10.4cm, height=3.2cm, tics=10]{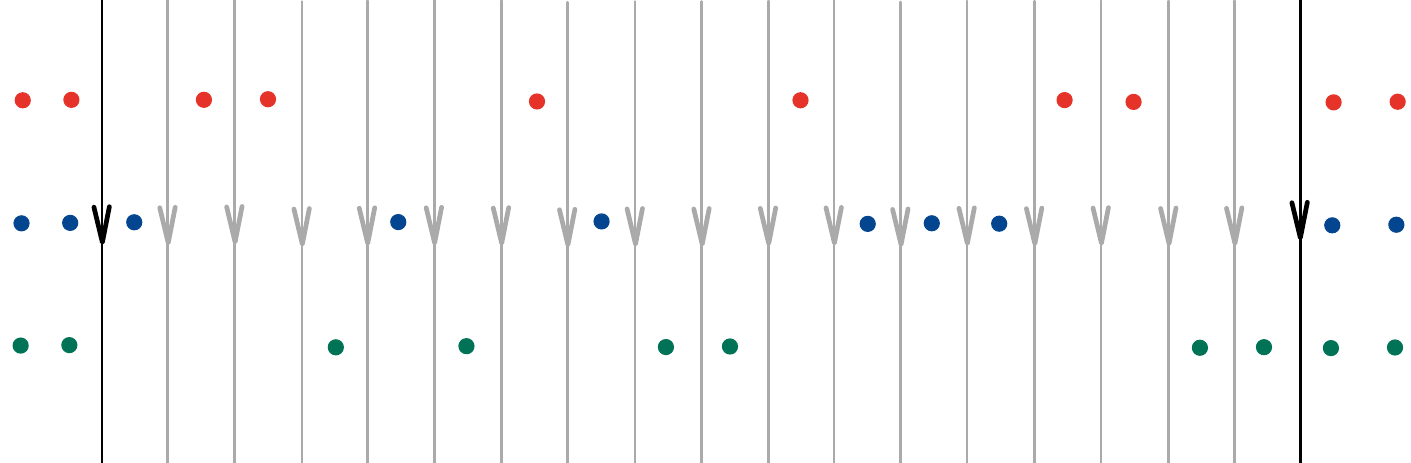}
    \put (2.5,30.5) {{\color{black}\Large$\displaystyle \lambda$}}
    \put (41,30) {\colorbox{white}{\color{gray}\large$\displaystyle \{\lambda_i\}_{i=0}^m$}}
        \put (94,30.5) {{\color{black}\Large$\displaystyle \lambda'$}}
        \put (-5,24) {{\color{myRED}\large$\displaystyle \O_{\sspc i}$}}
        \put (-5,15.5) {{\color{myBLUE}\large$\displaystyle \O_{\sspc j}$}}
        \put (-5,7) {{\color{myGREEN}\large$\displaystyle \O_{\sspc k}$}}
\end{overpic}
\end{figure}

 \mycomment{-0.25cm}

The following lemma characterizes the pushing-equivalence relation by the action of $f_\geo$.

 \mycomment{-0.1cm}
\begin{lemma}[Pushing Lemma]\label{lemma:pushing}
    Two geodesic leaves \(\lambda,\lambda^{\prime} \in \P\) satisfying \(L(\lambda') \subset L(\lambda)\)  are pushing-equivalent if, and only if, there exists an integer $N>0$ that satisfies
     \mycomment{-0.15cm}
    $$ L(f^N_\geo(\lambda)) \subset L(\lambda').
    $$
\end{lemma}

\begin{proof}
    We begin by proving the ``if'' part of the lemma. Let $N>0$ be as stated in the lemma.  According to Lemma \ref{lemma:geodesic_leaves}, the surface $S(\sspc\lambda, f^N_\geo(\lambda)\sspc) = L(\lambda) \cap R(f^N_\geo(\lambda))$ is homeomorphic to a finitely punctured plane. Moreover, the geodesic leaf $\lambda'$ is contained in $S(\sspc\lambda, f^N_\geo(\lambda)\sspc)$, because it satisfies $\lambda'\subset L(\lambda)$ and $\lambda'\subset R(f^N_\geo(\lambda))$. Consequently, the surface $S(\lambda,\lambda') = L(\lambda) \cap R(\lambda')$ is homeomorphic to a finitely punctured plane, showing that $\lambda$ and $\lambda'$ are pushing-equivalent.

    We now prove the ``only if'' part, which is the principal result of the lemma. For that, we assume initially that $S(\lambda,\lambda') = L(\lambda) \cap R(\lambda')$ is homeomorphic to a once punctured plane. 
    
    Let $p \in \OO$ be the puncture of $S(\lambda,\lambda')$, and let $\phi_p \in \F$ be the leaf passing through $p$.  Observe that $\phi_p\setminus\{p\}$ is the union of two lines $\phi_p^-$ and $\phi_p^+$ on the surface $\R^2\setminus \OO$ that satisfy
     \mycomment{-0.12cm}
    $$\omega(\phi_p^-) = \alpha(\phi_p^+) = p.$$
        
 \mycomment{-0.3cm}
    \noindent Even further, the pair $(\phi_p^-,\phi_p^+)$ defines an globally essential proper multiline on $\R^2\setminus \OO$, and therefore it admits a unique geodesic representative, which we denote by $\Lambda_p=(\sspc(\phi_p^-)^\geo, (\phi_p^+)^\geo)$. Since geodesics minimize intersections, $\Lambda_p$ is disjoint from both $\lambda$ and $\lambda'$, and it satisfies 
     \mycomment{-0.12cm}
$$ L(\lambda) \subset L(\Lambda_p) \subset L(\lambda').$$
    
 \mycomment{-0.3cm}
    \noindent According to the Homma-Schoenflies theorem, we can suppose up to conjugacy that each of the geodesics $\lambda$, $\Lambda_p$ and $\lambda'$ is vertical and oriented downwards.

   Next, observe that $f(\phi_p)\setminus\{f(p)\}$ also defines a globally essential proper multiline on $\R^2\setminus \OO$, and its geodesic representative coincides with $f_\geo(\Lambda_p)=(\sspc f_\geo(\phi_p^-), \sspc f_\geo(\phi_p^+)\sspc)$. Moreover, since  geodesics minimize intersections and $L(f(\phi_p)) \subset L(\phi_p)$, it holds that
    \mycomment{-0.12cm}
$$L(f_\geo(\Lambda_p)) \subset L(\Lambda_p).$$
       
 \mycomment{-0.25cm}
    \noindent Now, we remark that, since $L(\Lambda_p) \cap R(\lambda')$ has no punctures and $f(p) \in L(\Lambda_p)$, it follows that the point $f(p)$ consists in a puncture of $\R^2\setminus \OO$ that lies on $L(\lambda')$. In addition, any intersection between $f_\geo(\Lambda_p)$ and $\lambda'$ could be undone by a horizontal isotopy on $\R^2\setminus \OO$. Again, due to the minimal intersection property of geodesics, this means that $f_\geo(\Lambda_p)$ and $\lambda'$ are in fact disjoint, and consequently, they satisfy $L(f_\geo(\Lambda_p)) \subset L(\lambda')$.

   \begin{figure}[h!]
    \center
     \mycomment{0.4cm}\begin{overpic}[width=5.7cm, height=4cm, tics=10]{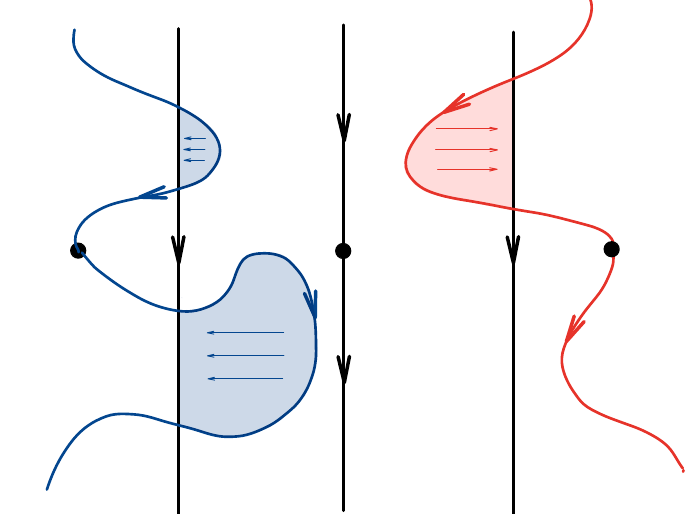}
        \put (47,69) {\color{black}\normalsize$\displaystyle \Lambda_p $}
         \put (24,68.5) {\color{black}\normalsize$\displaystyle \lambda $}
         \put (73,68.5) {\color{black}\normalsize$\displaystyle \lambda'$}
         \put (87,62) {\color{myRED}\large$\displaystyle f_\geo(\Lambda_p)$}
        \put (-15,60) {\color{myBLUE}\large$\displaystyle f^{-1}_\geo(\Lambda_p)$}
        
\end{overpic}
\end{figure}

    By repeating the same argument for $f^{-1}(\phi_p)$, we conclude that $L(\lambda) \subset L(f^{-1}_\geo(\Lambda_p))$. Yielding the series of inclusions $L(f_\geo(\Lambda_p)) \subset L(\lambda') \subset L(\lambda) \subset L(f^{-1}_\geo(\Lambda_p))$. By considering the action of $f_\geo$ on these inclusions, we conclude that $L(f^2_\geo(\Lambda_p)) \subset L(\Lambda_p)$. This implies that
     \mycomment{-0.12cm}
    $$L(f^2_\geo(\lambda)) \subset L(f^2_\geo(f^{-1}_\geo(\Lambda_p))) = L(f_\geo(\Lambda_p)) \subset L(\lambda').$$

     \mycomment{-0.27cm}
    \noindent We have then proved the lemma under the assumption that $S(\lambda,\lambda')$ contain a single puncture.

    Finally, to prove the lemma in the general case where $S(\lambda,\lambda')$ has $k>1$ punctures, we can apply the same argument $k$ times. This is because, as mentioned before, from the generic assumption on the foliation $\F$ with respect to $\OO$, we get a sequence $\{\lambda_i\}_{i=0}^k$ of pushing-equivalent geodesics with $\lambda_0 = \lambda$ and $\lambda_k =\lambda'$, such that $L(\lambda_{i+1}) \subset L(\lambda_i)$ and $S(\lambda_i,\lambda_{i+1})$ is homeomorphic to a $1$-punctured plane, for any $0\leq i <k$. Thus, we can apply the previous argument to each pair $(\lambda_i,\lambda_{i+1})$, to conclude that $L(f^{2}_\geo(\lambda_i)) \subset L(\lambda_{i+1})$.  This completes the proof of the lemma through the series of inclusions
     \mycomment{-0.12cm}
$$ \hspace*{2.8cm} L(f^{2k}_\geo(\lambda))  \subset L(f^{2(k-1)}_\geo(\lambda_1)) \subset \cdots \subset L(f^{2}_\geo(\lambda_{k-1})) \subset L(\lambda'). \hspace*{2.3cm}\qedhere $$
\end{proof}

 \mycomment{-0.2cm}
\begin{remark}
    The Pushing Lemma is a key result within our framework, being one of the main tools behind the proof of Theorem \ref{thmx:II-C}, which is an important result regarding the behavior of the Brouwer mapping class $\class{f,\OO}$.
    It describes an ``up to isotopy'' behavior of leaves in a transverse foliation that cannot be spotted without the geometric features of $\R^2\setminus \OO$. 
    
    For instance, consider a vertical transverse foliation $\F$, and assume that $\OO = \{\O\}$ is formed by a single orbit admitting a horizontal transverse trajectory $\Gamma_{\O} = \prod_{n \in \Z} \gamma_\O^n$. Then, consider any pair of leaves $\phi,\sspc \phi^\pp \in \F$ satisfying $L(\phi^\pp) \subset L(\phi)$. Now, observe that, since the orbit $\O$ crosses both leaves $\phi$ and $\phi^\pp$, we know that sufficiently large iterates $f^n(\phi)$ intersect $L(\phi^\pp)$. However,\break it is completely possible that no iterate $f^n(\phi)$ is entirely contained in $L(\phi^\pp)$.

    \vspace*{0.55cm}
    \begin{figure}[h!]
    \center
     \mycomment{0.65cm}\begin{overpic}[width=7.5cm, height=3.3cm, tics=10]{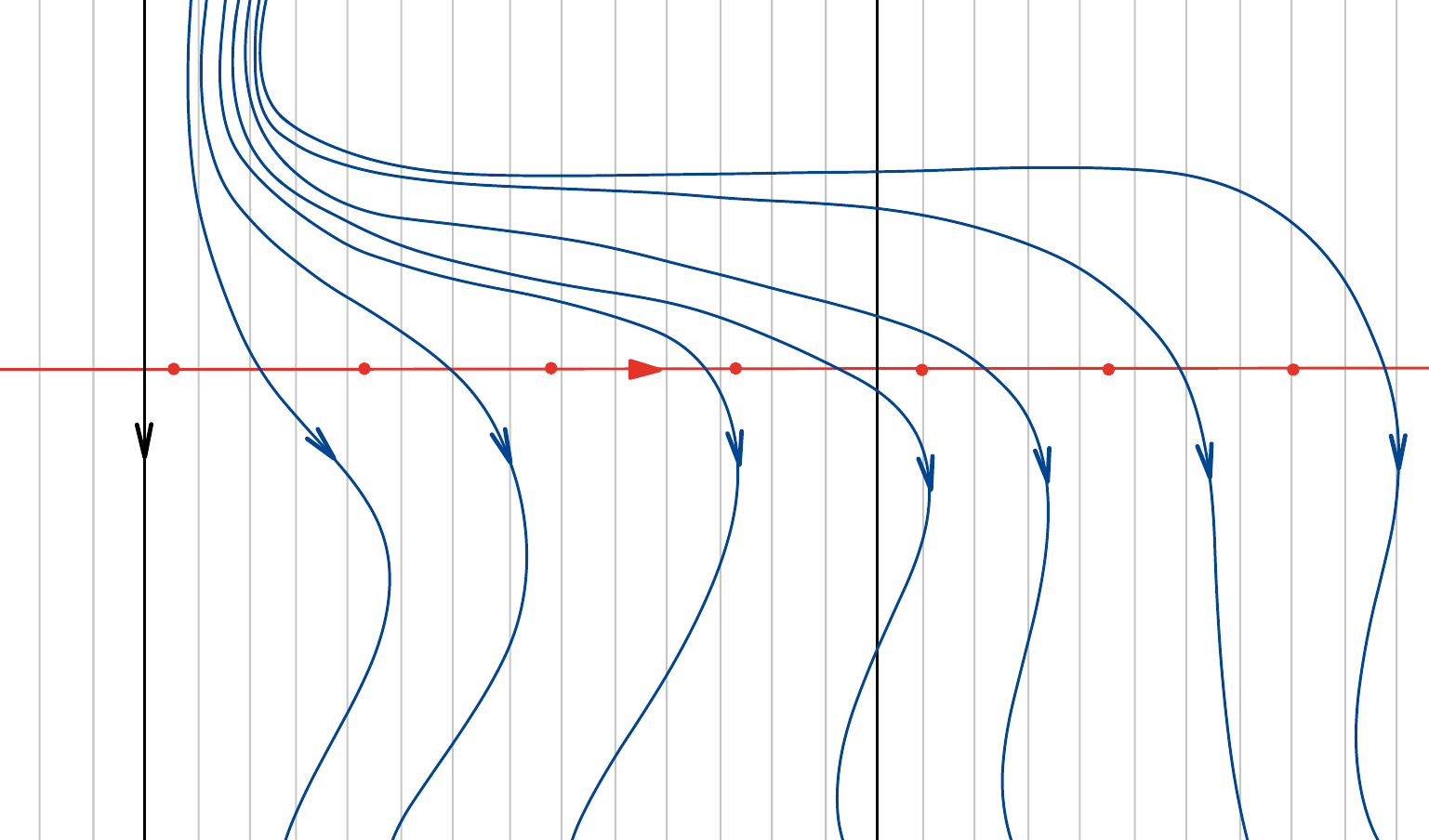}
        \put (9,-5.5) {\color{black}\large$\displaystyle \phi $}
         \put (60,-5.5) {\color{black}\large$\displaystyle \phi^\pp $}
         \put (-7,23.7) {\color{myRED}\large$\displaystyle \Gamma_\O$}
        \put (10,48) {\color{myBLUE}\large$\displaystyle \{f^n(\phi)\}_{n\geq0}$}
        
\end{overpic}
\end{figure}

\end{remark}

\newpage

\subsection{Transverse geodesic trajectories}\label{section:transverse_geodesic_trajectories}

For each $\O \in \OO$, fix a basepoint $x_\O\in \O$ and denote
\vspace*{0.1cm}
$$ \Gamma_\O = \prod_{n \in \Z} \gamma^n_\O,$$
where $\gamma_\O^n:[0,1]\longrightarrow\R^2$ is a path positively transverse to $\F$ joining $f^n(x_\O)$ and $f^{n+1}(x_\O).$\break
By abuse of notation, we can see each path $\gamma_\O^n$ as a line on the surface $\R^2\setminus \OO$ that satisfies
 \mycomment{-0.1cm}
$$\alpha(\gamma_\O^n) = f^n(x_\O) \quad \text{ and } \quad \omega(\gamma_\O^n) = f^{n+1}(x_\O).$$

 \mycomment{-0.2cm}
\noindent
As proven in \cite{schuback1}, there is a choice of paths $\{\gamma_\O^n\}_{n \in \Z}$ that is locally-finite, resulting in a proper transverse trajectory $\Gamma_\O$ for each orbit $\O \in \OO$. On the perspective of $\R^2\setminus \OO$, this allows us to construct $\Gamma_\O=\prod_{n \in \Z} \gamma^n_\O$ in order to be a globally essential proper multiline on $\R^2\setminus \OO$, in the sense that the lines in the family $\{\gamma_\O^n\}_{n \in \Z}$ are essential and pairwise non-isotopic on $\R^2\setminus \OO$.

According to the Straightening principle (see Lemma \ref{lemma:straightening_principle}), for each orbit $\O \in \OO$,
there exists an isotopy $(h_t)_{t\in [0,1]}$ on $\R^2\setminus \OO$ from the identity $h_0 = \textup{id}_{\R^2\setminus \OO}$ that satisfies
 \mycomment{-0.1cm}
$$ h_1(\gamma_\O^n) = (\gamma_\O^n)^\geo, \quad \forall n \in \Z.$$

 \mycomment{-0.2cm}
\noindent
In other words, the proper multiline $\Gamma_\O = \prod_{n \in \Z} \gamma_\O^n$ is globally isotopic to $\Gamma_\O^\geo = \prod_{n \in \Z} (\gamma_\O^n)^\geo$, which is also a proper multiline on $\R^2\setminus \OO$, called its geodesic representative (see Section \ref{sec:hyperbolic_geometry_R2Z}).

The geodesic proper multlines in the family $\{\Gamma_\O^\geo\}_{\O \in \OO}$ are transverse to $\G$, in the sense that
\begin{itemize}
    \item [$\sbullet$] If a geodesic leaf $\lambda \in \G$ is crossed by an orbit $\O \in \OO$, then $\lambda$ intersects $\Gamma_\O^\geo$ once. 
    \item [$\sbullet$] If a geodesic leaf $\lambda \in \G$ is not crossed by an orbit $\O \in \OO$, then $\lambda$ is disjoint from $\Gamma_\O^\geo$.
\end{itemize}
 \mycomment{0.2cm}
In particular, a separating geodesic $\sigma \in \mathcal S$ is disjoint from every geodesic trajectory $\Gamma_\O^\geo$.

In this section we establish two results. The first, Lemma \ref{lemma:limited_winding}, provides a mild control over the iterates of geodesic leaves through their intersections with transverse geodesic trajectories. %This technical result will play a key role in the proof of Theorem \ref{thmx:II-C}, mainly allowing us to control the behavior of iterates of geodesic leaves in $\G$. 

\begin{lemma}\label{lemma:limited_winding}
    For any geodesic leaf $\lambda \in \P$, any orbit $\O \in \OO$, and any $N \in \Z$, we have that 
     \mycomment{-0.05cm}
    $$\# (\sspc f^N_{\geo}(\lambda) \cap \Gamma_\O^\geo\spc) < \infty.$$
\end{lemma}

\begin{proof}[Proof of Lemma \ref{lemma:limited_winding}]
Note that, since $f^N_{\geo}(\lambda)$ is a geodesic planar line, i.e., a geodesic on $\R^2\setminus \OO$ that is a topological line on $\R^2$, we have that $f^N_{\geo}(\lambda)$ intersects each $(\gamma_\O^n)^\geo$ finitely many times.
Therefore, to prove the lemma, one should prove that the following set is finite 
     \mycomment{-0.12cm}
    $$A = \{n \in \Z \ \vert \ f^N_\geo(\lambda) \cap (\gamma_\O^n)^\geo \neq \varnothing\}.$$

     \mycomment{-0.2cm}
    The case $N=0$ holds trivially, while the cases $N>0$ and $N<0$ are completely analogous. Hence, we may initially assume that $N>0$. The orbit $\O\in \OO$ can satisfy one of the following:
    \begin{itemize}[leftmargin=1.7cm]
        \item[\textit{(i)\spc}] The orbit $\O$ crosses $\lambda$.
        \item[\textit{(ii)\sspc}] The orbit $\O$ is contained in $L(\lambda)$.
        \item[\textit{(iii)}] The orbit $\O$ is contained in $R(\lambda)$.
    \end{itemize}
     \mycomment{0.2cm}
    Since $N>0$, the geodesic $f^N_{\geo}(\lambda)$ is contained in $L(\lambda)$, and thus case (iii) becomes trivial,  as in this case $f^N_{\geo}(\lambda)$ and $\Gamma_\O^\geo$ lie on different sides of $\lambda$.

    To prove cases (i) and (ii), we fix some objects. As presented in Lemma \ref{lemma:geodesic_leaves}, the surface 
     \mycomment{-0.1cm}
    $$S(f^{-N}_{\geo}(\lambda), \lambda):=L(f^{-N}_{\geo}(\lambda)) \cap R(\lambda)$$ 
    
     \mycomment{-0.22cm}
    \noindent
    is homeomorphic to a $(Nk)$-punctured plane, where $k>0$ is the number of orbits in $\OO$ crossing the geodesic leaf $\lambda$.
    Let $K\subset \R^2$ be a closed disk containing all punctures of $S(f^{-N}_{\geo}(\lambda), \lambda)$ in its interior and that intersects the geodesic lines $f^{-N}_\geo(\lambda)$ and $\lambda$ at a single point each.
    
    \newpage
     \mycomment{-0cm}
    \begin{claim}\label{claim}
        Any geodesic line $g$ on $\R^2\setminus \OO$ with limit points $\alpha(g),\sspc \omega(g) \in L(\lambda)$ that intersects the geodesic leaf $\lambda$ and that is disjoint from $f^{-N}_\geo(\lambda)$ must also intersect the closed disk $K$.
    \end{claim}

\begin{figure}[h!]
    \center
     \mycomment{0.05cm}\begin{overpic}[width=6cm, height=4.2cm, tics=10]{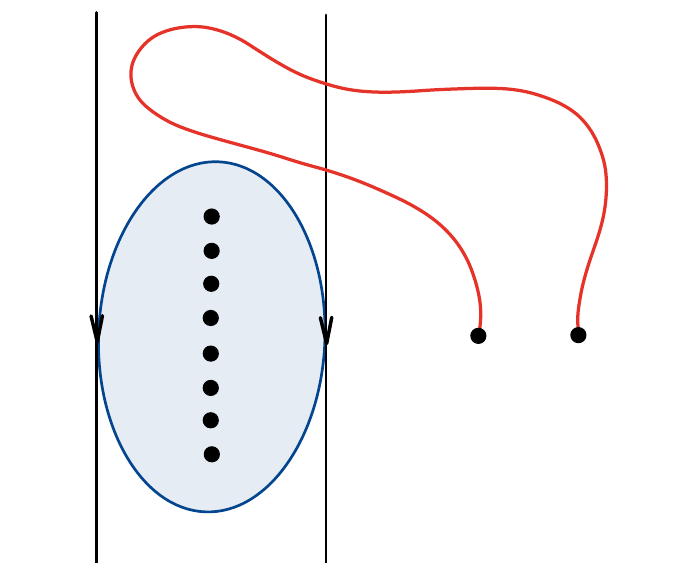}
        \put (49.5,66) {\color{black}\large$\displaystyle \lambda $}
        \put (-11,65.5) {\color{black}\large$\displaystyle f^{-N}_{\geo}(\lambda) $}
         \put (59,19) {\color{black}\normalsize$\displaystyle \alpha(g)$}
         \put (78,19) {\color{black}\normalsize$\displaystyle \omega(g)$}
         \put (87,59) {\color{myRED}\large$\displaystyle g$}
        \put (28,-2) {\color{myBLUE}\large$\displaystyle K$}
        
\end{overpic}
\end{figure}

 \mycomment{-0.5cm}
    \begin{proof}[Proof of Claim \ref{claim}]

    Let $g$ be a geodesic line on $\R^2\setminus \OO$ that satisfies $\alpha(g),\sspc \omega(g) \in L(\lambda)$.  If the line $g$ intersects $\lambda$, then they must form a planar bigon on the right side of $\lambda$. Meaning, there exists a closed disk $B\subset \R^2$ bounded by an arc of $g$ and an arc of $\lambda$ such that $B\subset \overline{R(\lambda)}$.  Since $g$ and $\lambda$ are both geodesics, they minimize the intersection number within their isotopy classes on $\R^2\setminus \OO$, which implies that the interior of $B$ must contain a point $p\in\OO$, because otherwise we would be able to decrease the number of intersections  between $\lambda$ and $g$ by performing an isotopy on $\R^2\setminus \OO$ that eliminates the bigon $B$. 
        
    Now, since the geodesic line $g$ is disjoint from $f^{-N}_\geo(\lambda)$, the point $p$ lies on the right of $\lambda$,  we conclude that $p$ is a puncture of $S(f^{-N}_{\geo}(\lambda), \lambda)$. Since the interior of $K$ contains all the punctures of $S(f^{-N}_{\geo}(\lambda), \lambda)$, we conclude that $p \in \textup{int}(K) \cap \textup{int}(B)$. Because $B$ and $K$ are closed disks with intersecting interiors, there exists a nontrivial interval of $\partial B$ that is contained in the interior of $K$. Since the interior of $K$ is disjoint from $\lambda$, and since $\partial B$ is exclusively formed by arcs of $g$ and of $\lambda$, we conclude that $K$ intersects the geodesic line $g$.
    \end{proof}

     \mycomment{-0.2cm}
    We can now proceed with the proof of Lemma \ref{lemma:limited_winding}, beginning with case (i).
    In this case, up to changing the basepoint $x_\O \in \O$, we can assume that $\lambda$ is crossed by $\O$ at time $0$. This means that $\lambda$ intersects the geodesic line $(\gamma_\O^0)^\geo$ exactly once, and we have
     \mycomment{-0.12cm}
    \begin{align*}
        (\gamma_\O^n)^\geo\subset R(\lambda)\sspc \ \text{ if } n<0 \quad \text{ and } \quad (\gamma_\O^n)^\geo\subset L(\lambda)\sspc \ \text{ if } n>0.
    \end{align*}

     \mycomment{-0.3cm}

\begin{figure}[h!]
    \center\begin{overpic}[width=13cm, height=3.6cm, tics=10]{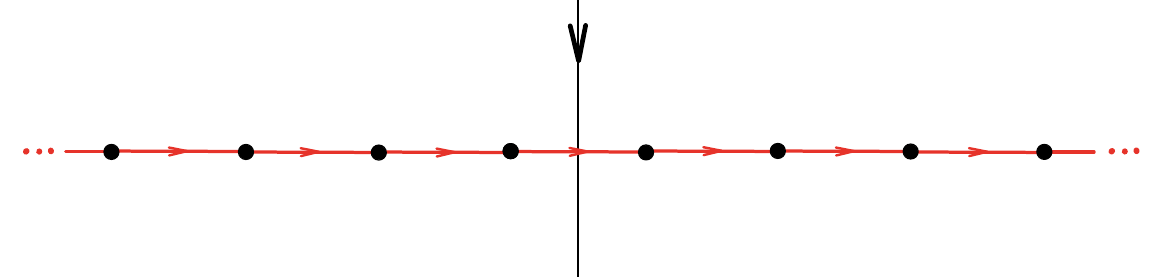}
    \put (52,24) {\color{black}\Large$\displaystyle \lambda $}
    \put (46,7.5) {\colorbox{white}{$\rule{0cm}{0.4cm}  \ \ \ \ $}}
    \put (46.5,8) {{\color{myRED}\normalsize$\displaystyle(\gamma_\O^{0})^\geo $}}
    \put (58,8) {\color{myRED}\normalsize$\displaystyle (\gamma_\O^{1})^\geo $}
    \put (70,8) {\color{myRED}\normalsize$\displaystyle (\gamma_\O^{2})^\geo $}
    \put (81,8) {\color{myRED}\normalsize$\displaystyle (\gamma_\O^{3})^\geo $}
    \put (34.5,8) {\color{myRED}\normalsize$\displaystyle (\gamma_\O^{-1})^\geo $}
    \put (23,8) {\color{myRED}\normalsize$\displaystyle (\gamma_\O^{-2})^\geo$}
    \put (11,8) {\color{myRED}\normalsize$\displaystyle (\gamma_\O^{-3})^\geo$}
\end{overpic}
\end{figure}

     \mycomment{-0.4cm}

    First, since $f^N_\geo(\lambda)$ is contained in $L(\lambda)$, we have that $f^N_\geo(\lambda) \cap (\gamma_\O^n)^\geo = \varnothing$ for all $n<0$. This gives us a lower bound for $A$. In addition, observe that the following properties hold 
     \mycomment{-0.12cm}
    \begin{align*}
        f^{-N}_\geo((\gamma_\O^n)^\geo) \cap f^{-N}_\geo(\lambda) = \varnothing, \quad \forall \spc n>0,\quad \ \ \ \spc  \\
        \alpha(\sspc f^{-N}_\geo((\gamma_\O^n)^\geo)\sspc ),\spc \omega(\sspc f^{-N}_\geo((\gamma_\O^n)^\geo)\sspc ) \in L(\lambda), \quad \forall \spc n>N+1.
    \end{align*}

     \mycomment{-0.21cm}
    \noindent 
    By applying Claim \ref{claim}, we obtain the following equality of sets
     \mycomment{-0.12cm}
    $$ \{n>N+1 \ \vert \ f^{-N}_\geo((\gamma_\O^n)^\geo) \cap K \neq \varnothing\} \ =\  \{n>N+1 \ \vert \ f^{-N}_\geo((\gamma_\O^n)^\geo) \cap \lambda \neq \varnothing\}.$$

     \mycomment{-0.21cm}
    \noindent 
    As mentioned previously, an important consequence of Theorem \ref{theorem:1B-intro} is that the family of transverse paths $\{\gamma_\O^n\}_{n \in \Z}$ can be chosen to be locally finite. This implies that the family of geodesic representatives $\{(\gamma_\O^n)^\geo\}_{n \in \Z}$ is also locally finite, resulting in a proper multiline of geodesic representatives $\Gamma_\O^\geo = \prod_{n \in \Z} (\gamma_\O^n)^\geo$ on $\R^2\setminus \OO$. In particular, this means that the first set on the equality above is finite, thus concluding that the second set is also finite. By considering the action of $f^N_\geo$, we observe that the second set in the equality above coincides with $A$, thus concluding the proof of the lemma for case (i).

    Finally, we can proceed to prove case (ii). In this case, for any integer $n\in \Z$, we have
     \mycomment{-0.12cm}
    $$ f^{-N}_\geo((\gamma_\O^n)^\geo) \cap f^{-N}_\geo(\lambda) = \varnothing,$$

     \mycomment{-0.21cm}
    \noindent 
    and moreover, both limit points of $f^{-N}_\geo((\gamma_\O^n)^\geo)$ lie on the left side $L(\lambda)$. Following the same arguments presented in case (i), we can prove that the set $A$ is also finite in case (ii).
\end{proof}

The second result of this section, Proposition \ref{prop:separated_geodesic_trajectories}, shows that in the fully separated case, meaning $\#\mathcal S = r-1$, then transverse geodesic trajectories are invariant under the action of $f_\geo$.

\begin{proposition}\label{prop:separated_geodesic_trajectories}
    If $\#\mathcal S = r-1$, where $r>0$ is the number of orbits in $\OO$, then we have
     \mycomment{-0.12cm}
    $$ f_\geo(\Gamma_\O^\geo) = \Gamma_\O^\geo, \quad \forall \O \in \OO.$$

     \mycomment{-0.23cm}
    \noindent
    More specifically, for any $\O \in \OO$ and  $n \in \Z$, it holds that $f_\geo((\gamma_\O^n)^\geo) = (\gamma_\O^{n+1})^\geo$.
\end{proposition}

\begin{figure}[h!]
    \center\begin{overpic}[width=9cm, height=4.5cm, tics=10]{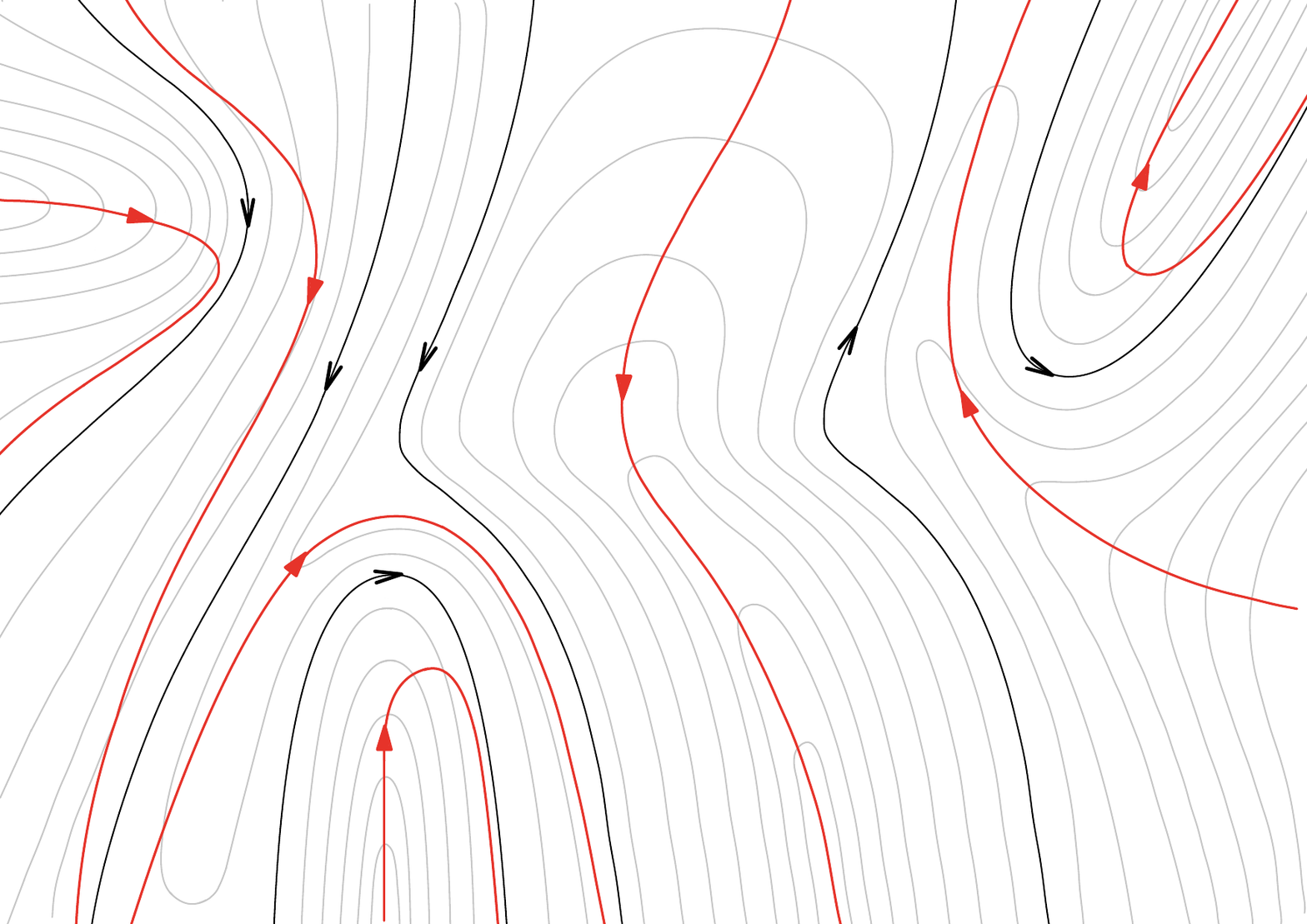}
         % \put (123,3) {{\color{black}\normalsize$\displaystyle \Lbot \O$}{\color{black}\normalsize$\displaystyle \ \thinsupsetneq \ $}{\color{black}\normalsize$\displaystyle \Lbot {\O^\pp}$}}
\end{overpic}
\end{figure}

 \vspace*{-0.15cm}
\begin{proof}
    Let $X\subset \R^2\setminus \OO$ be the union of all geodesics in $\mathcal S$. Each orbit $\O \in \OO$ is contained in a unique connected component of $\R^2\setminus X$, which we denote by $M_\O$. The family $\{M_\O\}_{\O \in \OO}$ is formed by pairwise disjoint 2-submanifolds, and each $M_\O$ has a geodesic boundary $\partial M_\O\subset X$.
    According to the Straightening Principle (see Lemma \ref{lemma:straightening_principle}), there exists an isotopy $(f_t)_{t\in [0,1]}$ on $\R^2\setminus \OO$ such that $f_0 = f\vert_{\R^2\setminus \OO}$ and $f_1=:f^*$ is a homeomorphism fixing $X$ pointwise. 

    Fix an orbit $\O \in \OO$. Consider a collar neighborhood $N_\O$ of the boundary $\partial M_\O$ on the surface $M_\O\setminus \O$ that is disjoint from both $\Gamma_\O^\geo$ and $f^*(\Gamma_\O^\geo)$. Consider also a homeomorphism
     \mycomment{-0.12cm}
    $$ \varphi_\O: M_\O\setminus\O \longrightarrow \R^2\setminus \O$$

     \mycomment{-0.23cm}
    \noindent
    coinciding with the inclusion map $M_\O\setminus \O \hookrightarrow \R^2\setminus \O$ outside the collar neighborhood $N_\O$.

    Now, we need the following claims. The first, Claim \ref{claim: inspired on handel}, is inspired by an argument of Handel in \cite{handel99}. The second, Claim \ref{claim: f_geo invariance}, is the key idea of the proof, as it shows that proper transverse trajectories are ideal substitutes for homotopy translation arcs in this setting.

    \begin{claim}\label{claim: inspired on handel}
        The homeomorphism $f^*\vert_{M_\O\setminus \O}$ is isotopic to $\varphi_\O^{-1} \circ f \circ \varphi_\O$ on the surface $M_\O\setminus \O$.
    \end{claim}

     \mycomment{-0.15cm}
    \begin{proof}[Proof of Claim \ref{claim: inspired on handel}]
        First, note that after an isotopy on $M_\O\setminus \O$, we may assume that $f^*\vert_{M_\O\setminus\O}$ leaves the collar neighborhood $N_\O$ invariant. Next, since $f$ and $f^*$ are isotopic on $\R^2\setminus \OO$,  nwe have that $\varphi_\O^{-1} \circ f \circ \varphi_\O$ is isotopic to $\varphi_\O^{-1} \circ f^* \circ \varphi_\O$ on $M_\O\setminus \O$. Finally, we observe that the homeomorphism $\varphi_\O^{-1} \circ f^* \circ \varphi_\O$ is isotopic to $f^*\vert_{M_\O\setminus \O}$ on $M_\O\setminus \O$, because they coincide outside the collar neighborhood $N_\O$. This concludes the proof of the claim.
    \end{proof}

    \vspace*{0.2cm}

    \begin{claim}\label{claim: f_geo invariance}
        For any $n \in \Z$, the image $f(\gamma_\O^{\sspc n})$ is isotopic to $\gamma_\O^{n+1}$ on the surface $\R^2\setminus \O$.
    \end{claim}

    \vspace*{0.1cm}

     \mycomment{-0.4cm}
    \begin{proof}[Proof of Claim \ref{claim: f_geo invariance}]
        Let $x_\O \in \O$ be the initially chosen basepoint in $\O$. For each $n \in \Z$,  let $\phi_n \in \F$ be the leaf containing $f^n(x_\O)$, and let
            $ S(\phi_n,f(\phi_{n+2}))= L(\phi_n) \cap R(f(\phi_{n+2}))\setminus \O$. Since we are in the totally separated case, i.e. $\#\mathcal S=r-1$, the leaves $\phi_n$ and $f(\phi_{n+2})$ are crossed exclusively by the orbit $\O$. This means that $S(\phi_n,f(\phi_{n+2}))$ is homeomorphic to a 2-punctured plane, where the punctures correspond to the points $f^{n+1}(x_\O)$ and $f^{n+2}(x_\O)$.
            Note that both paths $\gamma_\O^{n+1}$ and $f(\gamma_\O^n)$ are contained in $S(\phi_n,f(\phi_{n+2}))$, and they satisfy
             \mycomment{-0.1cm}
            $$ \alpha(\gamma_\O^{n+1}) = \alpha(f(\gamma_\O^n)) = f^{n+1}(x_\O) \quad \text{ and } \quad \omega(\gamma_\O^{n+1}) = \omega(f(\gamma_\O^n)) = f^{n+2}(x_\O).$$

           \begin{figure}[h!]
    \center
     \mycomment{0.2cm}\begin{overpic}[width=9cm, height=5.5cm, tics=10]{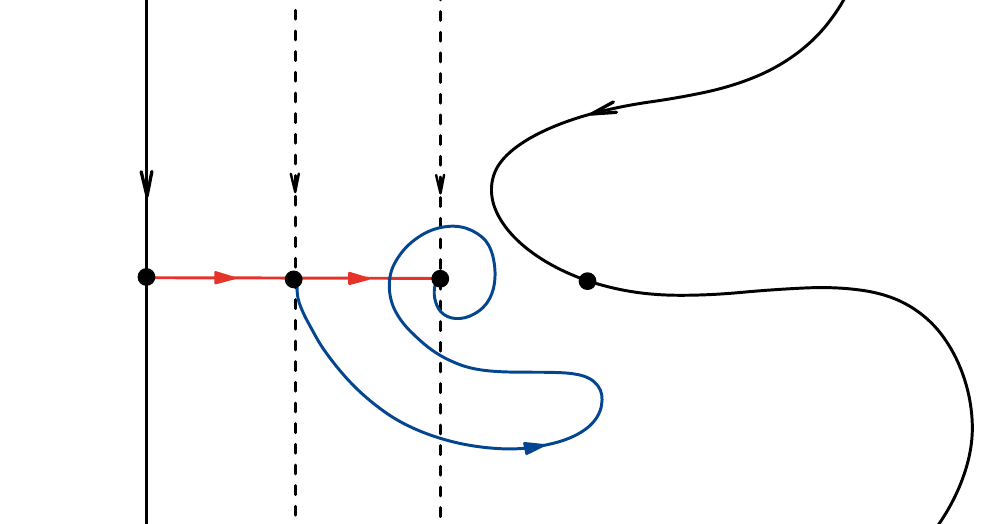}
        \put (86,55) {\color{black}\large$\displaystyle f(\phi_{n+2})$}
        \put (7,55) {\color{black}\large$\displaystyle \phi_{n}$}
        \put (41,63) {\color{black}\small$\displaystyle \phi_{n+2}$}
        \put (24.5,63) {\color{black}\small$\displaystyle \phi_{n+1}$}
        \put (63,10) {\color{myBLUE}\large$\displaystyle f(\gamma_\O^n)$}
        \put (32,31.5) {\color{myRED}\large$\displaystyle \gamma_\O^{n+1}$}
        \put (18,31.5) {\color{myRED}\large$\displaystyle \gamma_\O^{n}$}
        
\end{overpic}
\end{figure}

 \mycomment{-0.2cm}
            Now, let $p_1,p_2 \in \R^2$ be the two punctures of the 2-punctured plane $\R^2\setminus \{p_1, p_2\}$. It is a classical  and well-known fact that
            there exists a unique isotopy class of lines on $\R^2\setminus \{p_1, p_2\}$
            that connect the puncture $p_1$ to the puncture $p_2$ while still remaining injective. 
            \begin{figure}[h!]
    \center
     \mycomment{0cm}\begin{overpic}[width=3.3cm, height=3.3cm, tics=10]{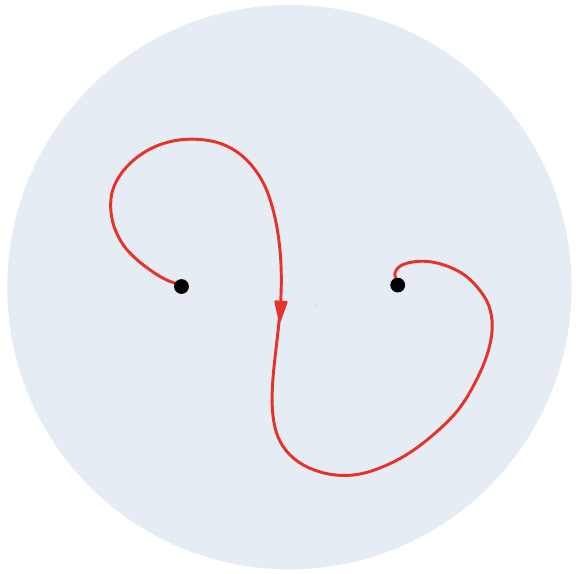}
        \put (62,40) {\color{black}\small$\displaystyle p_2$}
        \put (26,40) {\color{black}\small$\displaystyle p_1$}
        %\put (18,31) {\color{myRED}\large$\displaystyle \gamma_\O^{n}$}
        
\end{overpic}
\end{figure}
            
 \mycomment{-0.5cm}
            This can be applied to the current setting, implying that there exists an isotopy $(h_t)_{t\in[0,1]}$ on the 2-punctured plane $S(\phi_n,f(\phi_{n+2}))$ from the identity that satisfies $h_1(\gamma_\O^{n+1}) = f(\gamma_\O^n)$. Moreover, note that this isotopy can be chosen to be the identity on a collar neighborhood of the boundary $\partial S(\phi_n,f(\phi_{n+2})) = \phi_n \cup f(\phi_{n+2})$.
            This means that the isotopy $(h_t)_{t\in[0,1]}$ can be extended to an isotopy on the entire surface $\R^2\setminus \O$, which proves the claim.
    \end{proof} 

     \mycomment{-0.1cm}
    To conclude the proof of the proposition, we recall that $(\gamma_\O^n)^\geo$ denotes the geodesic representative of $\gamma_\O^n$ on the surface $\R^2\setminus \OO$. Since $\R^2\setminus \O$ has fewer punctures than $\R^2\setminus \OO$, we still have that $\gamma_\O^n$ is isotopic to $(\gamma_\O^n)^\geo$ on $\R^2\setminus \O$. Similarly, $\gamma_\O^{n+1}$ is isotopic to $(\gamma_\O^{n+1})^\geo$ on $\R^2\setminus \O$. Thus, by applying
    the result provided by Claim \ref{claim: f_geo invariance}, combined with the observation above, we conclude that $f((\gamma_\O^n)^\geo)$ is isotopic to $(\gamma_\O^{n+1})^\geo$ on the surface $\R^2\setminus \O$.

    Next, since $f$ and $f^*$ are isotopic on $\R^2\setminus \OO$, by a similar argument passing to $\R^2\setminus \O$,
    we conclude that $f((\gamma_\O^n)^\geo)$ is isotopic to $f^*((\gamma_\O^n)^\geo)$ on $\R^2\setminus \O$. Combined with the previous conclusion, we have that $f^*((\gamma_\O^n)^\geo)$ is isotopic to $(\gamma_\O^{n+1})^\geo$ on $\R^2\setminus \O$.

    Now, by applying the inverse homeomorphism $\varphi_\O^{-1}$, we conclude that $\varphi^{-1} \circ f^* ((\gamma_\O^n)^\geo)$ is isotopic to $\varphi_\O^{-1}((\gamma_\O^{n+1})^\geo)$ on the surface $M_\O\setminus \O$. Recall that the collar neighborhood $N_\O$ was chosen to be disjoint from both $\Gamma_\O^\geo$ and $f^*(\Gamma_\O^\geo)$, and that $\varphi_\O$ coincides with the inclusion map outside $N_\O$. This implies that 
    $$ \varphi_\O^{-1} \circ f^* ((\gamma_\O^n)^\geo) = f^*((\gamma_\O^n)^\geo) \quad \text{ and } \quad \varphi_\O^{-1}((\gamma_\O^{n+1})^\geo) = (\gamma_\O^{n+1})^\geo.$$
    Thus, we conclude that $f^*((\gamma_\O^n)^\geo)$ is isotopic to $(\gamma_\O^{n+1})^\geo$ on the surface $M_\O\setminus \O$. Note that an isotopy between $f^*((\gamma_\O^n)^\geo)$ and $(\gamma_\O^{n+1})^\geo$ on $M_\O\setminus \O$ can be adjusted to be the identity on the collar neighborhood $N_\O$. Therefore, such isotopy can be extended to an isotopy on the entire surface $\R^2\setminus \OO$. This concludes that $f^*((\gamma_\O^n)^\geo)$ is isotopic to $(\gamma_\O^{n+1})^\geo$ on $\R^2\setminus \OO$.

    We conclude the proof by recalling that $f^*((\gamma_\O^n)^\geo)$ is isotopic to $f_\geo((\gamma_\O^n)^\geo)$ on $\R^2\setminus \OO$, which implies that $f_\geo((\gamma_\O^n)^\geo)$ is isotopic to $(\gamma_\O^{n+1})^\geo$ on $\R^2\setminus \OO$. Since both $f_\geo((\gamma_\O^n)^\geo)$ and $(\gamma_\O^{n+1})^\geo$ are geodesics on $\R^2\setminus \OO$, they must coincide, concluding the proof of the proposition.
\end{proof}

\section{Proof of theorem \ref{thmx:II-B}}\label{chapter:proof_THM_B}

We remark that Theorem \ref{thmx:II-B} follows from applying Proposition \ref{lemma:simplification_wandering_leaves} below repeatedly to each wandering leaf, in order to simplify the transverse foliation as much as possible, thus yielding a transverse geodesic lamination $\G$ that satisfies the properties required in Theorem \ref{thmx:II-B}.

\begin{proposition}\label{lemma:simplification_wandering_leaves}
    Let $\F$ be a transverse foliation for $f$, and let $\phi\in \F$ be a leaf that is an essential line on $\R^2\setminus \OO$. Assume that the family of geodesics $\{f^n_\geo(\phi^\geo)\}_{n\geq0}$ is locally-finite. Then, there exists an isotopy $(f_t)_{t\in [0,1]}$ on the surface $\R^2\setminus \OO$ that starts at $f_0=f\vert_{\R^2\setminus \OO}$ and satisfies:
    \begin{itemize}[leftmargin=1.2cm]
        \item[\textup{\textbf{(i)}}]\ $\overline{\rule{0cm}{0.3cm} f_t}$ is a Brouwer homeomorphism, for all $t\in [0,1]$,
        \item[\textup{\textbf{(ii)}}]\ $\overline{\rule{0cm}{0.3cm} f_1}$ admits a transverse foliation $\F^\pp$ that satisfies
        $$ \F^\pp\vert_{\overline{\rule{0cm}{0.23cm} R(\phi)}}=\F\vert_{\overline{\rule{0cm}{0.23cm} R(\phi)}} \quad \ \text{ and } \ \quad \F^\pp\vert_{\overline{\rule{0cm}{0.23cm} L(\phi)}} \ \text{ is trivial and } \overline{\rule{0cm}{0.3cm} f_1}\text{-invariant}.$$
    \end{itemize}
\end{proposition}
\begin{figure}[h!]
    \center 
    \hspace*{-0.3cm}\begin{overpic}[width=7cm, height=4.2cm, tics=10]{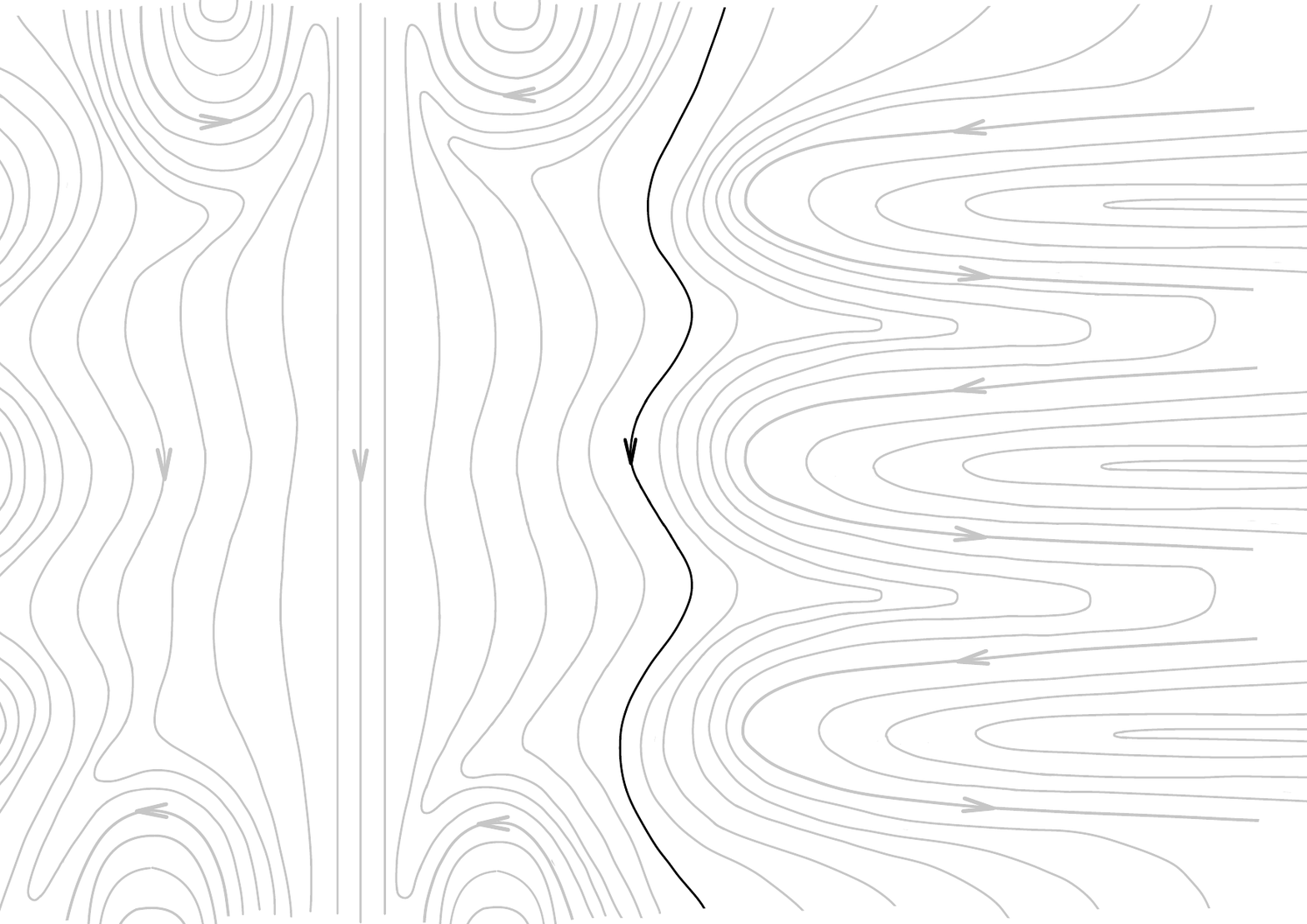}
         \put (58.5,59) {\colorbox{white}{\color{black}\large$\displaystyle \phi$}}
         \put (47,-10) {{\color{gray}\Large$\displaystyle \F$}}
        % \put (-14.5,31) {{\color{myRED}\large$\displaystyle \O_{\sspc i}$}}
        
\end{overpic}
\hspace*{1cm}\begin{overpic}[width=7cm, height=4.2cm, tics=10]{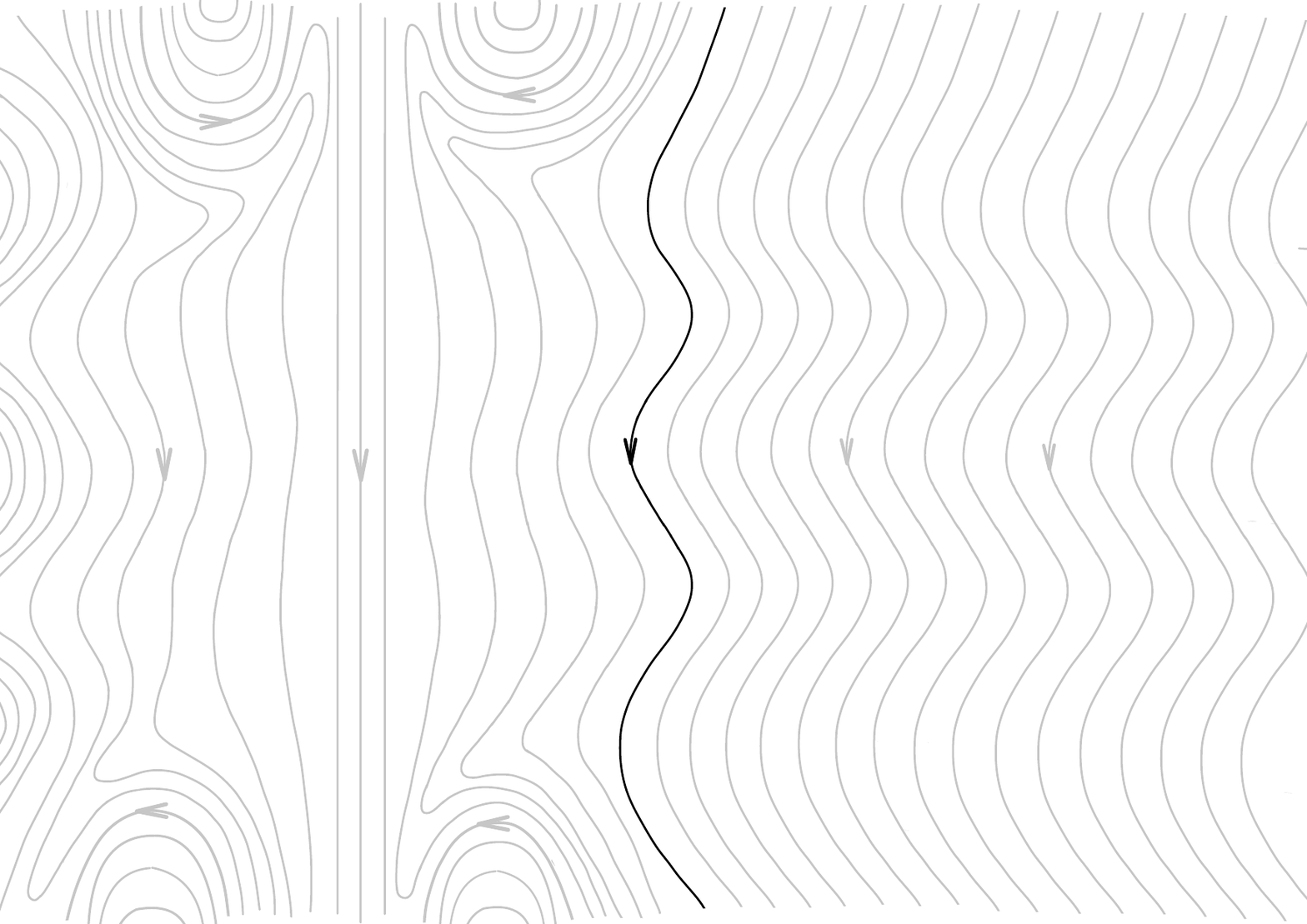}
    \put (58.5,59) {\colorbox{white}{\color{black}\large$\displaystyle \phi$}}
    \put (47,-10) {{\color{gray}\Large$\displaystyle \F^\pp$}}
    % \put (102,30.5) {{\color{black}\Large$\displaystyle \sigma \in \mathcal S$}}
    % \put (-3,54) {{\color{myRED}\large$\displaystyle \O_{\sspc i}$}}
    % \put (-3,10) {{\color{myBLUE}\large$\displaystyle \O_{\sspc j}$}}
\end{overpic}
\end{figure}

 \vspace*{0.5cm}

\begin{proof}[Proof of Proposition \ref{lemma:simplification_wandering_leaves}]
Let $\lambda = \phi^\geo \in \G$ be the geodesic representative of the leaf $\phi\in \F$ given in the hypothesis of the proposition. We may assume up to an identity isotopy on $\R^2\setminus \OO$  that the leaf $\phi$ and its geodesic representative $\lambda$ coincide, i.e., $\phi = \lambda$. By the hypothesis of the proposition, the family $\{f^n_\geo(\lambda)\}_{n\geq0}$ is locally-finite, however note that the family $\{f^n(\lambda)\}_{n\geq0}$ is not necessarily locally-finite. For sake of simplicity, we may assume up to conjugacy, by the Homma-Schoenflies theorem, that $f^n_\geo(\lambda)=\{n\} \times \R$ oriented downwards, for all $n\geq0$.

By the Straightening Principle, there exists an isotopy $(h_t)_{t\in [0,1]}$ on the surface $\R^2\setminus \OO$ that starts at the identity map and ends at a homeomorphism $h_1$ that satisfies
 \mycomment{-0.1cm}
$$ h_1(\lambda) = \lambda \quad \text{ and } \quad h_1(f(\lambda)) = f_\geo(\lambda).$$

 \mycomment{-0.2cm}
\noindent As mentioned in Remark \ref{remark:straightening_principle}, we can take this isotopy to fix $\lambda$, meaning that
 \mycomment{-0.1cm}
$$ h_t(x) = x, \quad \forall x \in \lambda,\ \forall t\in [0,1].$$

 \mycomment{-1cm}
\noindent Observe that, the way the isotopy $(h_t)_{t\in [0,1]}$ moves points on the right of $\lambda$ does not affect the points on the left of $\lambda$. This means that, since $f(\lambda)$ and $f_\geo(\lambda)$ are both located on the left side of $\lambda$, we can take the isotopy $(h_t)_{t\in [0,1]}$ to fix $\overline{\rule{0cm}{0.23cm}R(\lambda)}$ pointwise, meaning that
 \mycomment{-0.12cm}
$$ h_t\vert_{\overline{\rule{0cm}{0.23cm}R(\lambda)}} = \text{Id}\sspc, \quad \forall \sspc t \in [0,1].$$

 \mycomment{-0.23cm}
\noindent By defining $f_t:= h_t \circ f\vert_{\R^2\setminus \OO} \circ h^{-1}_t$, we obtain an isotopy $(f_t)_{t\in [0,1]}$ on the surface $\R^2\setminus \OO$ such that
$ \overline{\rule{0cm}{0.3cm} f_t}$ is a Brouwer homeomorphism conjugate to $f$ that coincides with $f$ on $f^{-1}(\overline{\rule{0cm}{0.23cm}R(\lambda)})$,  for every  $t\in [0,1]$. In addition, the homeomorphism $f_1$ at the end of the isotopy satisfies
 \mycomment{-0.15cm}
$$ f_1(\lambda) = f_\geo(\lambda).$$

\begin{figure}[h!]
    \center
     \vspace*{0.3cm}\begin{overpic}[width=9cm, height=4.7cm, tics=10]{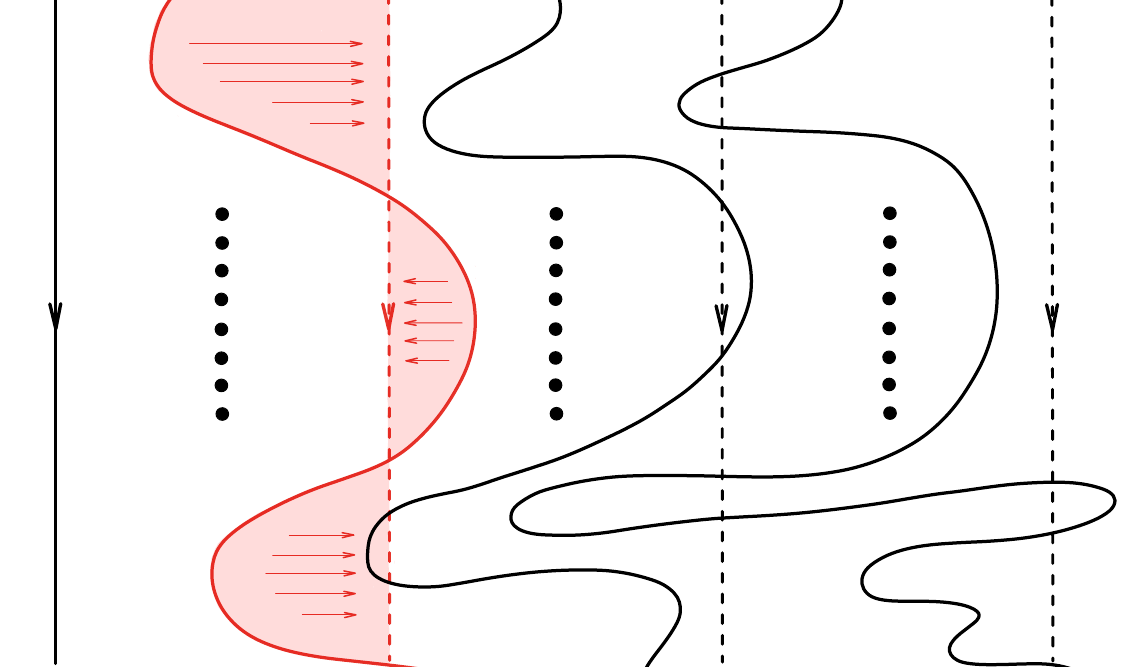}
         \put (87,54.5) {\color{black}$\displaystyle f^3_\geo(\lambda)$}
         \put (58,54.5) {\color{black}$\displaystyle f^2_\geo(\lambda)$}
         \put (-8,26) {\color{black}$\displaystyle \phi = \lambda$}
         \put (30,54.5) {\color{myRED}$\displaystyle f_\geo(\lambda)$}
         \put (12,54.5) {\color{myRED}$\displaystyle f(\lambda)$}
         \put (43,54.5) {\color{black}$\displaystyle f^2(\lambda)$}
         \put (70,54.5) {\color{black}$\displaystyle f^3(\lambda)$}
        
\end{overpic}
\end{figure}

\vspace*{0.1cm}
We end up with a setting similar to the one presented in the beginning of the proof, where the line $f_1(\lambda)$ coincides with its geodesic representative $f_\geo(\lambda)$. By repeating the same arguments presented above, we can extend the isotopy $(f_t)_{t\in [0,1]}$ into an isotopy $(f_t)_{t\in [0,2]}$  
where each $\overline{\rule{0cm}{0.3cm} f_t}$ is a Brouwer homeomorphism conjugate to $f$ and the map $f_2$ satisfies
\vspace*{0.1cm}
$$ f_2\vert_{\overline{\rule{0cm}{0.23cm}R(\lambda)}} = f_1\vert_{\overline{\rule{0cm}{0.23cm}R(\lambda)}} \quad \text{ and } \quad f_2^{\sspc n}(\lambda) = f^{\sspc n}_\geo(\lambda), \quad \forall \sspc 0 \leq n \leq 2.$$

  \vspace*{0.1cm}
 \noindent
We end up with a similar setting, again, and we can keep extending the isotopy $(f_t)_{t\in [0,1]}$  to obtain an isotopy $(f_t)_{t\in [0,\infty)}$ where each map $\overline{\rule{0cm}{0.3cm} f_t}$ is a Brouwer homeomorphism conjugate to the original map $f$ and, for any $N\geq 0$ and $t\in [N,N+1]$, the following properties hold:
  \vspace*{0.05cm}
    $$ f_t\vert_{\spc\overline{\rule{0cm}{0.23cm}f^{\sspc N-1}_\geo(R(\lambda))}\spc} = f_N\vert_{\spc\overline{\rule{0cm}{0.23cm}f^{\sspc N-1}_\geo(R(\lambda))}\spc} \quad \text{ and } \quad f_N^{\sspc n}(\lambda) = f^{\sspc n}_\geo(\lambda)\sspc ,\quad \forall\spc  0\leq n \leq N.$$

      \vspace*{0.05cm}

\begin{claim}\label{claim2}
    The isotopy $(f_t)_{t\in [0,\infty)}$ converges to a homeomorphism $f_\infty \in \homeo(\R^2\setminus \OO)$ (with respect to the compact-open topology) as $t\to \infty$, such that $\overline{\rule{0cm}{0.32cm}f_\infty}$ is a Brouwer homeomorphism and, moreover, $f_\infty^{\sspc n}(\lambda) = f^{\sspc n}_\geo(\lambda)$ for every $n\geq 0$.
    % that satisfies the following properties:
    % \begin{itemize}[leftmargin=1.2cm]
    %     \item The extension $\overline{\rule{0cm}{0.32cm}f_\infty}$ is a Brouwer homeomorphism,
    %     \item For any $n\geq 0$, it holds that $f_\infty^{\sspc n}(\lambda) = f^{\sspc n}_\geo(\lambda)$.
    % \end{itemize}
\end{claim}

 \mycomment{-0.4cm}
\begin{proof}[Proof of Claim \ref{claim2}]
    Observe that, for any $N\geq0$ and any $x \in  \overline{\rule{0cm}{0.35cm}f^{\sspc N-1}_\geo(R(\lambda))}$, we have that
      \mycomment{-0.13cm}
    $$ f_t(x) = f_{N} (x)\sspc, \quad  \forall t\geq N.$$

      \mycomment{-0.23cm}
 \noindent For any compact $K\subset \R^2\setminus \OO$, there exists an $N>0$ large enough such that $K\subset \overline{\rule{0cm}{0.35cm}f^{\sspc N-1}_\geo(R(\lambda))}$. Consequently, there exists a homeomorphism $f_\infty \in \homeo(\R^2\setminus \OO)$ given by $ f_\infty = \lim_{t\to \infty} f_t\sspc$, where the limit is taken on uniform convergence on compact sets of $\R^2\setminus \OO$, and satisfies
    \mycomment{-0.1cm}
    \begin{equation}\label{eq:convergence_F_infty}
        f_\infty(x) = f_{N} (x)\sspc,  \quad \forall N\geq0\sspc ,\quad \forall x \in \overline{\rule{0cm}{0.355cm}f^{\sspc N-1}_\geo(R(\lambda))}\sspc.
    \end{equation}

     \mycomment{-0.2cm}
    \noindent
Since the set $\bigcup_{n\geq 0} \overline{\rule{0cm}{0.32cm}R(f^n_\geo(\lambda))}$ covers the plane, for any $x \in \R^2$ there exists $N>0$ such that
 \mycomment{-0.1cm}
   $$ \overline{\rule{0cm}{0.32cm}f_\infty}(x) = \overline{\rule{0cm}{0.32cm}f_N}(x) \neq x.$$

    \mycomment{-0.23cm}
    \noindent
   This implies that $\overline{\rule{0cm}{0.32cm}f_\infty}$ is itself a Brouwer homeomorphism.
    
    Finally, since $f^n_\geo(\lambda)\subset \overline{\rule{0cm}{0.32cm}f^{\sspc n}_\geo(R(\lambda))}$ for every $n\geq 0$, we can apply (\ref{eq:convergence_F_infty}) to conclude that
     \mycomment{-0.1cm}
    $$ f_\infty^{\sspc n}(\lambda) = f_{n+1}^{\sspc n}(\lambda) = f^{\sspc n}_\geo(\lambda)\sspc, \quad \forall \sspc n \geq 0.$$

     \mycomment{-0.2cm}
    \noindent
    Thus, we conclude the proof of the claim.
\end{proof}

 \mycomment{-0.2cm}
It is worth highlighting that homeomorphism $f_\infty$ given by Claim \ref{claim2} is isotopic to $f\vert_{\R^2\setminus \OO}$ but not necessarily conjugate to it, even though $f_t$  is conjugate to $f\vert_{\R^2\setminus \OO}$ for every $0\leq t <\infty$.
If the family $\{f^n(\lambda)\}_{n\geq 0}$ is locally finite, then in this case $f_\infty$ is also conjugate to $f\vert_{\R^2\setminus \OO}$.

At last, we have that $\phi = \lambda$ and $f^{\sspc n}_\infty(\phi) = f^{\sspc n}_\geo(\lambda)$ for every $n\geq 0$, and this allows us to define a transverse foliation of $\overline{\rule{0cm}{0.33cm}f_\infty}$, denoted by $\F^\pp$, as follows:

\vspace*{0.1cm}
\begin{itemize}[leftmargin=1.4cm]
    \item On the right side $\overline{\rule{0cm}{0.33cm}R(\phi)}$, the foliation $\F^\pp$ coincides with $\F$.
    
    \vspace*{0.1cm}
    \item For every $n\geq 0$, the oriented line $f^n_\infty(\phi)$ is a leaf in $\F^\pp$.
    
    \vspace*{0.1cm}
    \item For every $n\geq 0$, the set $D_n:=\overline{\rule{0cm}{0.38cm}L(f^n_\infty(\phi)) \cap R(f^{n+1}_\infty(\phi))}$ is trivially foliated by $\F^\pp$.
    
    \vspace*{0.1cm}
    \item For every $n\geq 0$, we have that $\F^\pp\vert_{D_{n+1}} = f_\infty(\F^\pp\vert_{D_n})$.
\end{itemize}

\begin{figure}[h!]
    \center
     \mycomment{0.2cm}\begin{overpic}[width=3.8cm, height=3.8cm, tics=10]{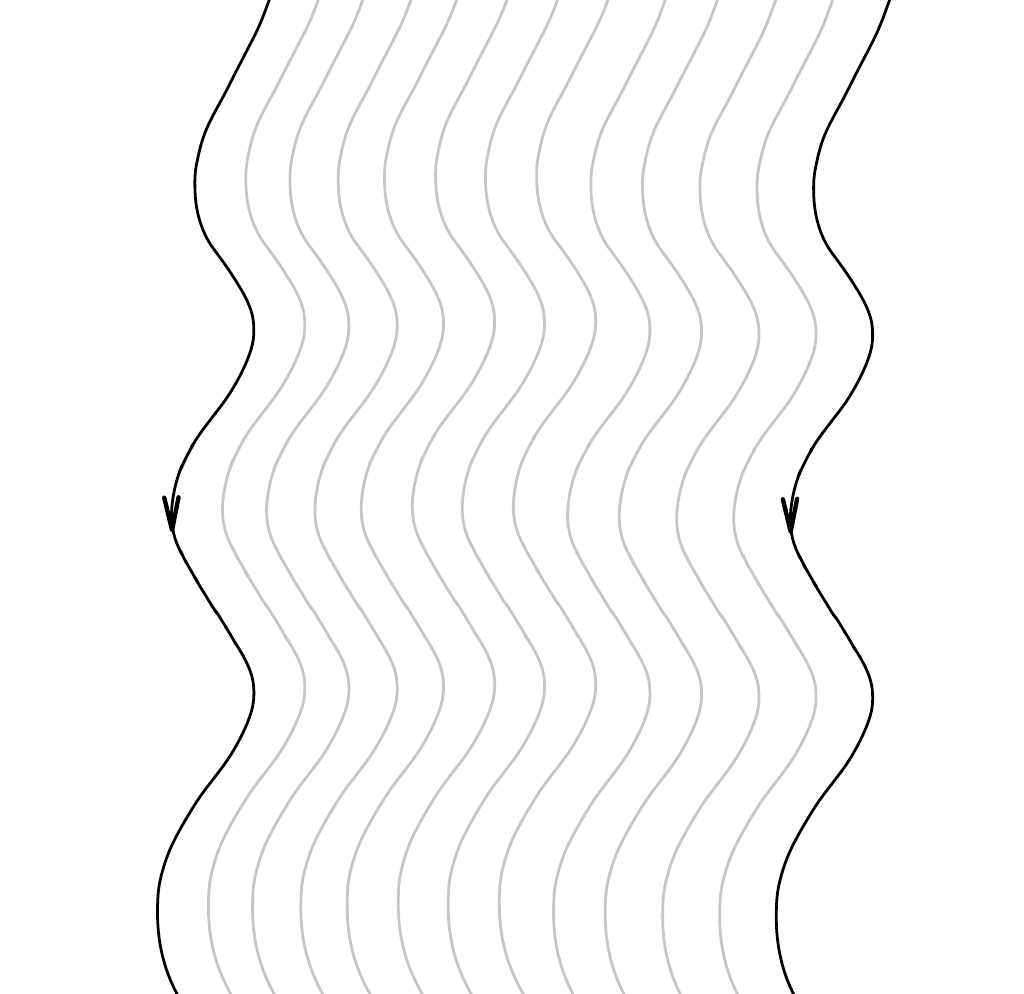}
         \put (-14,47) {\color{black}$\displaystyle f_\infty^n(\phi)$}
         \put (43.8,47) {\colorbox{white}{\large\color{gray}$\rule{0cm}{0.3cm} \ \ \sspc$}}
         \put (44,47) {\color{gray}\large$\displaystyle D_n$}
         \put (85,47) {\color{black}$\displaystyle f_\infty^{n+1}(\phi)$}
        %  \put (58,54.5) {\color{black}$\displaystyle f^2_\geo(\lambda)$}
        %  \put (-8,26) {\color{black}$\displaystyle \phi = \lambda$}
        %  \put (30,54.5) {\color{myRED}$\displaystyle f_\geo(\lambda)$}
        %  \put (12,54.5) {\color{myRED}$\displaystyle f(\lambda)$}
        %  \put (43,54.5) {\color{black}$\displaystyle f^2(\lambda)$}
        %  \put (70,54.5) {\color{black}$\displaystyle f^3(\lambda)$}
\end{overpic}
\end{figure}

 \mycomment{-0.3cm}

The isotopy $(f_t)_{t \in [0,\infty]}$ and the foliation $\F^\pp$ complete the proof of the proposition.
\end{proof}

\section{Proof of Theorem \ref{thmx:II-C}}

 \mycomment{-1cm}
\begin{remark}\label{remark:about_thm_II_A_and_II_B} Before heading to the proof of Theorem \ref{thmx:II-C}, we do  some important remarks.
    
    \vspace*{0.1cm}

    \begin{itemize}
        \item[\textbf{(i)}]  First, note that geodesics in $\mathcal S$ are fixed by $f_\geo$, and thus cannot be wandering.
        
        \vspace*{0.1cm}
        \item[\textbf{(ii)}] Second, note that every geodesic leaf $\lambda \in \P$  trivially satisfies both implications
         \mycomment{-0.17cm}
$$
\ \ \lambda \in \P \text{ is forward-wandering } \ \ \implies \text{ no orbit in } \OO \text{ lies entirely on } L(\lambda).
$$

\vspace*{-0.6cm}
$$
\ \ \lambda \in \P \text{ is backward-wandering } \implies \text{ no orbit in } \OO \text{ lies entirely on } R(\lambda).
$$

\vspace*{0.1cm}

\noindent Indeed, since $\O$ is an $f$-invariant set, if an orbit $\O \in \OO$ lies entirely on $L(\lambda)$, it must in fact lie on $\bigcap_{n>0} L(f^n_\geo(\lambda))$. Thus, any path $\alpha:[0,1]\longrightarrow\R^2$ joining $\O$ to $\lambda$ must also intersect $f^n_\geo(\lambda)$ for all $n>0$. Thus, the family $\{f^n_\geo(\lambda)\}_{n>0}$ cannot be locally-finite.  An analogous reasoning applies to the implication in the backward-wandering case.

\newpage

        \item[\textbf{(iii)}] Finally, we illustrate why the structural properties given by Theorem \ref{thmx:structural_thm} are not sufficient on their own to prove Theorem \ref{thmx:II-C}. Suggesting that the Pushing Lemma is the crucial dynamical ingredient that allows us recover \cite[Case (1) in the proof of Proposition 6.6]{handel99}\break without relying on Handel's machinery of Fitted families introduced in \cite{handel99}.   

        \vspace*{0.2cm}

    \textbf{Illustative example:}\ 

    \vspace*{0.1cm}
    Fix a basepoint $x = (0,0) \in \R^2$, and suppose that $f^n(x) = (n,0)$ for each $n \in \Z$.  Let $\OO:=\{\O=\O(f,x)\}$, and  let $\G$ be a transverse geodesic lamination $\G$ for $\class{f,\OO}$. Consider $\lambda \in \P$ to be the geodesic leaf crossed by $\O$ at time $0$. %For simplicity, we suppose that the geodesic leaf $\lambda$ is the vertical $\{1/2\} \times \R$ oriented downwards.
    
    All we can obtain from Theorem \ref{thmx:structural_thm} is that, for any $N \geq 0$, the following hold: 
    \vspace*{0.1cm}
$$ L(f^{N+1}_\geo(\lambda)) \subset L(f^N_\geo(\lambda)) \quad \text{ and } \quad \begin{cases}
(m, 0) \in R(f^N_\geo(\lambda)) \quad \text{ if } m\leq N,\\
(m, 0) \in L(f^N_\geo(\lambda)) \quad \text{ if } m> N.
\end{cases}
$$

 \mycomment{0cm}
 \vspace*{0.1cm}
\noindent As tillustrated in the below, this information is not sufficient to conclude that the family of iterates $\{f^N_\geo(\lambda)\}_{N\geq0}$ is locally-finite, suggesting that the Pushing Lemma is the key.
\end{itemize}

\vspace*{0.1cm}

\begin{figure}[h!]
    \center
    \hspace*{0.5cm}\begin{overpic}[width=14cm, height=8cm, tics=10]{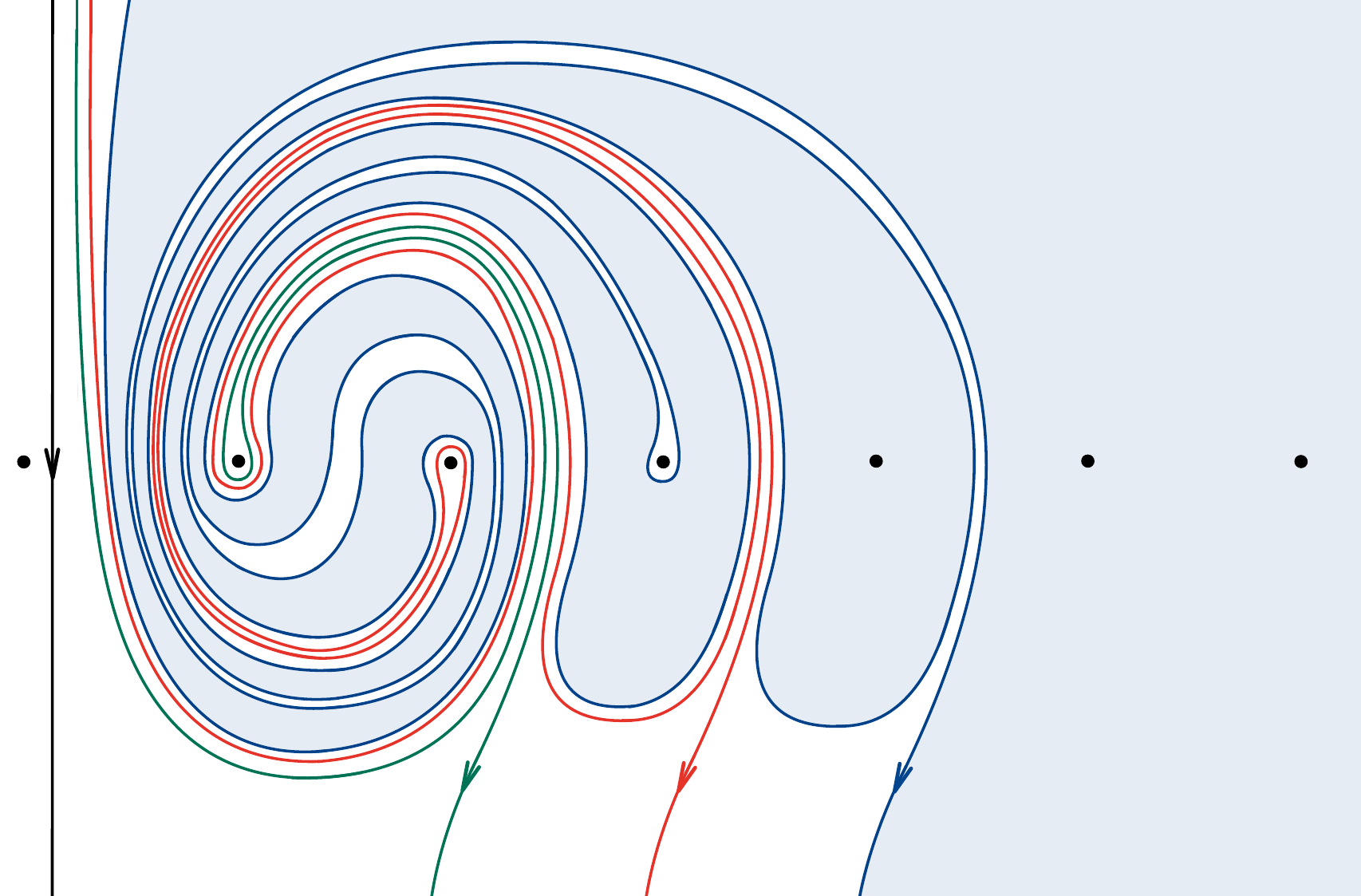}
          \put (2.5,-5) {\Large\color{black}$\displaystyle \lambda$}
          \put (26.3,-5) {\Large\color{myGREEN}$\displaystyle f_\geo(\lambda)$}
          \put (42.7,-5) {\Large\color{myRED}$\displaystyle f^2_\geo(\lambda)$}
          \put (58,-5) {\Large\color{myBLUE}$\displaystyle f^3_\geo(\lambda)$}
        %  \put (58,54.5) {\color{black}$\displaystyle f^2_\geo(\lambda)$}
        %  \put (-8,26) {\color{black}$\displaystyle \phi = \lambda$}
        %  \put (30,54.5) {\color{myRED}$\displaystyle f_\geo(\lambda)$}
        %  \put (12,54.5) {\color{myRED}$\displaystyle f(\lambda)$}
        %  \put (43,54.5) {\color{black}$\displaystyle f^2(\lambda)$}
        %  \put (70,54.5) {\color{black}$\displaystyle f^3(\lambda)$}
\end{overpic}
\end{figure}

\end{remark}

\vspace*{0.4cm}
\begin{proof}[Proof of Theorem \ref{thmx:II-C}]
    We will prove only the forward-wandering case, as the backward-wandering case can be proved through a symmetric argument.
    The ``only if'' implication was already proved in item (ii) of Remark \ref{remark:about_thm_II_A_and_II_B}. Now, we prove the ``if'' implication using an induction argument.  We assume that the theorem holds for any pushing geodesic leaf crossed by at most $r-1$ orbits in $\OO$, and then demonstrate that the theorem also holds for those crossed by $r$ orbits in $\OO$. Of course, the base case $r=1$ must be proved initially.

    \vspace*{0.3cm}
    \textbf{Proof of the base case $\mathbf{r=1}$:}
    Let $\lambda \in \P$ be crossed by exactly the orbit $\O_1 \in \OO$. Clearly, no orbit in $\OO$ lies entirely on $L(\lambda)$. Let $\Gamma_1^\geo$ be a transverse geodesic trajectory of $\O_1$, which is a proper multiline on $\R^2\setminus \OO$ because $\Gamma_1$ is constructed to be properly embedded (Theorem \ref{theorem:1B-intro}).

    According to the Homma-Schoenflies theorem, we can assume up to conjugacy that $\lambda$ is the vertical line $\{1/2\}\times \R$ oriented downwards, that  $\O_1 = \{f^n(x_1) = (n,1) \in \R^2 \mid n \in \Z\}$, and that the completion of $\Gamma_1^\geo$ is the horizontal line $\R\times\{1\}$ oriented in the positive $x$-direction.

    Observe that, in this particular case, every geodesic leaf in $\G$ must be crossed by $\O_1\sspc$.
    This means that every geodesic leaf in $\spc\G\vert_{\overline{\rule{0cm}{0.22cm}L(\lambda)}}\spc$ is pushing-equivalent to $\lambda$ and, thus, we can enumerate them as $\spc\G\vert_{\overline{\rule{0cm}{0.22cm}L(\lambda)}}\spc = \{\lambda_k\}_{k \in \N}$ such that $\lambda_0 = \lambda$ and $L(\lambda_{k+1}) \subset L(\lambda_k)$, for all $k \in \N$. Since $\lambda$ is crossed by $\O_1$ at time $0$, we have for each $k \in \N$ that $\lambda_k$ is crossed by $\O_1$ at time $k$.

    From the hypothesis given by the Homma-Schoenflies theorem, we note that for $k \in \N$, the geodesic leaf $\lambda_k$ is disjoint from $
    \lambda = \{1/2\}\times \R$ and satisfies the following properties:
    \begin{itemize}[leftmargin=1cm]
        \item[$\sbullet$] The geodesic leaf $\lambda_k$ intersects the horizontal line $\Delta_1 \cup \O_1 = \R\times\{1\}$ only once.
        \item[$\sbullet$] The forward orbit $\O_1 \cap L(\lambda_k)$ correspond to the set $\{ (j\sspc , 1) \in \R^2 \mid j \in \Z, \ j > k\}$.
        %\item[$\sbullet$] The backward orbit $\O_1 \cap R(\lambda_k)$ correspond to the set $\{ (j\sspc , 1) \in \R^2 \mid j \in \Z, \ j \leq k\}$.
    \end{itemize}

     \mycomment{0.23cm}
    \noindent With these properties, we deduce that each $\lambda_k$ must intersect the horizontal line $\R\times\{1\}$ somewhere in the arc $(k,k+1) \times \{1\}$, which is the geodesic $(\gamma_{\O_1}^k)^\geo$ constituting $\Gamma_1^\geo$. 

\begin{figure}[h!]
    \center
     \mycomment{0cm}\begin{overpic}[width=9.6cm, height=5cm, tics=10]{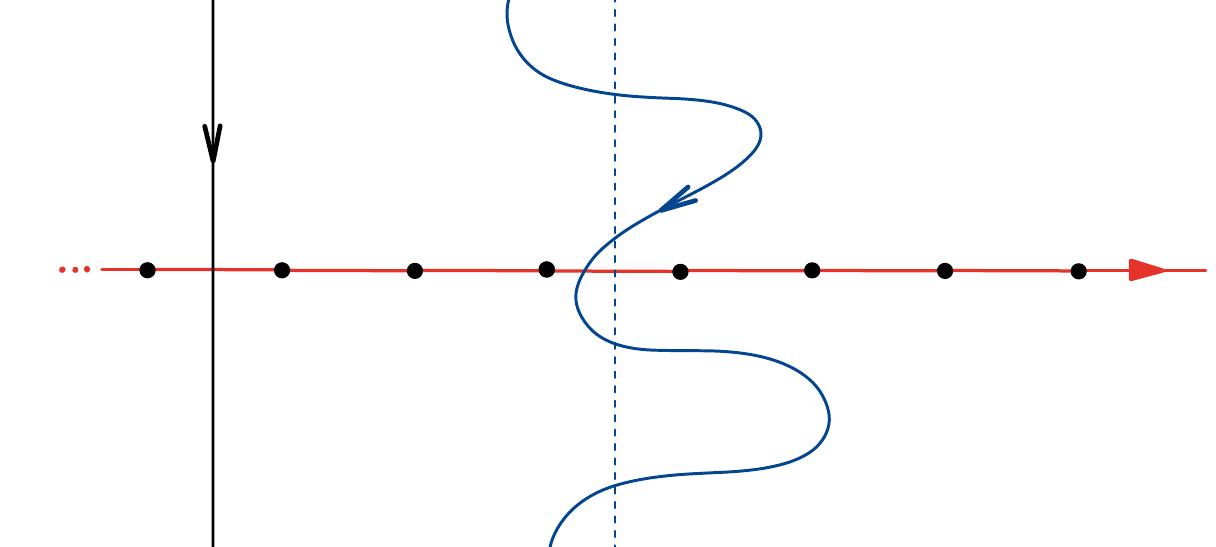}
         \put (11.5,33) {\large\color{black}$\displaystyle \lambda$}
         \put (35,43) {\large\color{myBLUE}$\displaystyle \lambda_k$}
         \put (53,44) {\large\color{myBLUE}$\displaystyle v_k$}
         \put (101.5,22.5) {\large\color{myRED}$\displaystyle \Gamma_1^\geo$}
        
\end{overpic}
\end{figure}

 \mycomment{-0.2cm}
    
    Then, for each $k \in \N$, we can perform an isotopy on $\R^2\setminus \OO$ that moves $\lambda_k$ to the vertical line $v_k = \{1/2+k\}\times \R$ oriented downwards, with matching orientations.
    Now, since the family of verticals $\{v_k\}_{k \in \N}$ is locally-finite, and each $\lambda_k$ is the geodesic representative of the vertical line $v_k$, we have that the family $\{\lambda_k\}_{k \in \N}$ is also locally-finite (see Section \ref{sec:hyperbolic_geometry_R2Z}).

    According to the Pushing Lemma, for any $k\in \N$, there exists some $M_k>0$ such that
     \mycomment{-0.1cm}
    $$ L(f^{M_k}_\geo(\lambda)) \subset L(\lambda_k).$$

     \mycomment{-0.25cm}
 \noindent Since the family $\{\lambda_k\}_{i\in \N}$ is a locally-finite, the set $\bigcap_{k\in \N} L(\lambda_k)$ is empty. This implies that \vspace*{0.1cm}
 $$\bigcap_{n\in \N} L(f^{n}_\geo(\lambda)) = \bigcap_{k\in \N} L(f^{M_k}_\geo(\lambda)) \subset \bigcap_{k\in \N} L(\lambda_k) = \varnothing\sspc.$$

 \vspace*{0cm}
 \noindent This proves that the family $\{f^n_\geo(\lambda)\}_{n\in \N}$ is locally-finite, thus concluding the base case $r=1$.

 \vspace*{0.4cm}
    \textbf{Proof of the induction step $\mathbf{(r-1 \implies r\sspc)}\sspc$:}
    Suppose the theorem holds for any geodesic leaf crossed by at most $r-1$ orbits in $\OO$.
    Let $\lambda \in \P$ be a geodesic leaf crossed by exactly $r$ orbits $\{\O_1,...,\O_r\}\subset \OO$, and assume that no orbit in $\OO$ is lies entirely on $L(\lambda)$. %For each $i \in \{1,...,r\}$, let $\Gamma_i^\geo$ be the transverse geodesic trajectory associated with the orbit $\O_{\sspc i}$.

    First, we partition the set $I= \{1,...,r\}$ into two disjoint subsets, denoted  $I_\text{out}$ and $I_\omega$. The set $I_\text{out}$ is formed by the indices $i \in I$ such that the orbit $\O_{\sspc i}$ crosses a geodesic leaf  that lies on the side $L(\lambda)$ and that is not pushing-equivalent to $\lambda$. Meanwhile, $I_\omega = I\setminus I_\text{out}$,
    meaning that $I_\omega$ is formed by the indices $i \in I$ such that every geodesic leaf lying on $L(\lambda)$ crossed by the orbit $\O_{\sspc i}$ is pushing-equivalent to $\lambda$. 

    By the definition of $I_\text{out}$, for each $i \in I_\text{out}$, we have a well-defined integer $M(i)$ given by
    \vspace*{0.2cm}
\[ M(i)= \max \Biggl\{\ k \in \Z \ \ \Bigg\vert \ \  \parbox[c]{.46\linewidth}{%\raggedright%
   $\exists \lambda^\pp \in \G$ pushing-equivalent to $\lambda$ such that  $\lambda^\pp$ is crossed by the orbit $\O_{\sspc i}$ at time $k$%
   }\ \Biggr\}.
   \]

   \vspace*{0.2cm}
     \mycomment{0.1cm}
    \begin{claim}\label{claim3}
        For each index $i \in I_\text{out}$, there exists a geodesic leaf $\eta_i \in \G$ that is not pushing-equivalent to $\lambda$ and crossed by the orbit $\O_{\sspc i}$ at time $M(i)$. Moreover, there exists $\lambda^* \in \G$ that is pushing-equivalent to $\lambda$ and crossed by the orbit $\O_{\sspc i}$ at time $M(i)$, for every $i \in I_\text{out}$.
    \end{claim}

    \vspace*{-0.1cm}
\begin{figure}[h!]
    \center
    \hspace*{4cm}\begin{overpic}[width=8.6cm, height=7.5cm, tics=10]{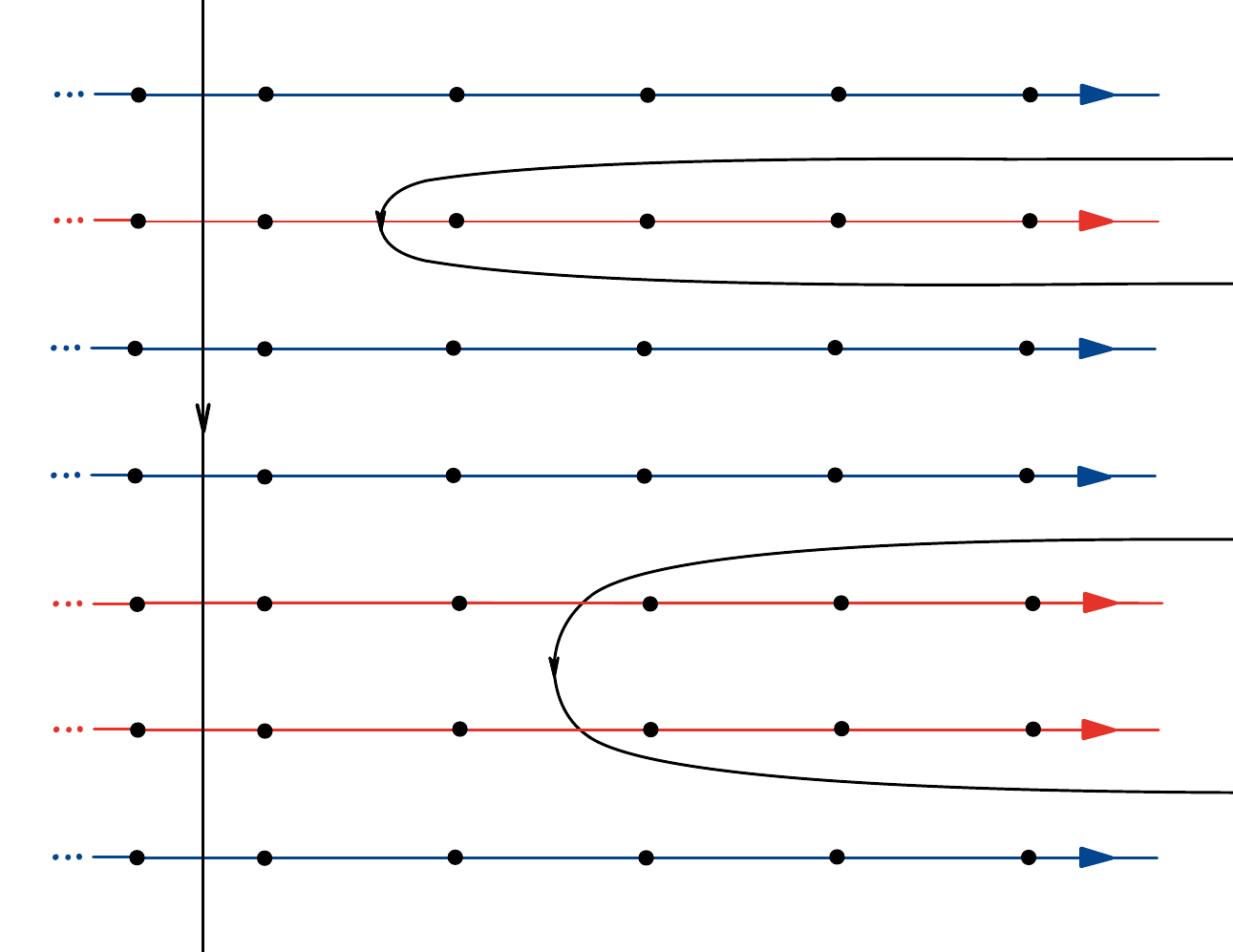}
         \put (10,48) {\large\color{black}$\displaystyle \lambda$}
         \put (103,37) {\large\color{black}$\displaystyle \eta_2 =\eta_3$}
         \put (103,72) {\large\color{black}$\displaystyle \eta_6$}
         \put (-5,19.5) {\large\color{myRED}$\displaystyle \Gamma_2^\geo$}
         \put (-5,8) {\large\color{myBLUE}$\displaystyle \Gamma_1^\geo$}
         \put (-5,31) {\large\color{myRED}$\displaystyle \Gamma_3^\geo$}
         \put (-5,43) {\large\color{myBLUE}$\displaystyle \Gamma_4^\geo$}
         \put (-5,55) {\large\color{myBLUE}$\displaystyle \Gamma_5^\geo$}
         \put (-57,49) {\large\color{myRED}$\displaystyle I_\text{out} = \{2,3,6\}$}
         \put (-57,36) {\large\color{myBLUE}$\displaystyle I_\omega = \{1,4,5,7\}$}
         \put (-5,78) {\large\color{myBLUE}$\displaystyle \Gamma_7^\geo$}
         \put (-5,66) {\large\color{myRED}$\displaystyle \Gamma_6^\geo$}
        
\end{overpic}
\end{figure}

     \vspace*{-0.2cm}
    \begin{proof}[Proof of Claim \ref{claim3}] We need to look at the underlying transverse foliation $\F$ defining $\G$.  We briefly recall some notions from \cite{schuback1}: Each orbit $\O_{\sspc i}$ in $\OO$ is associated to a set $C_{\O_{\sspc i}}\subset \R^2$, that consists in the union of all leaves $\phi \in \F$ such that $\O_{\sspc i} \cap L(\phi)$ and $\O_{\sspc i} \cap R(\phi)$ are non-empty. Each set $C_{\O_{\sspc i}}$ is trivially foliated by $\F$ and homeomorphic to a plane. Note that  there exists an essential leaf $\phi_\lambda \in \F$ lying on the set $C_{\O_{\sspc i}}$ such that $\phi_\lambda^\geo = \lambda$. Hence, the set $$D:= \bigcap_{i=1}^r C_{\O_{\sspc i}} $$
        is non-empty and, thus, $D$ is also homeomorphic to a trivially foliated plane (see \cite{schuback1}). Finally, we recall that the set $\partial_L D \subset \partial D$, which is defined as the union of all leaves $\phi \in \partial D$ satisfying $\phi \subset L(\phi')$ for any $\phi' \in D$, is the union of a locally-finite collection of leaves in $\F$.
        Since $\partial_L D$ is formed by a locally-finite collection of leaves in $\F$, we can apply the argument in Section \ref{sec:foliated_brouwer} to assume, up to perturbing $\F$, that each leaf in $\partial_L D$ is disjoint from $\OO$.
        For each index $i \in I_\text{out}$, there exists a leaf $\phi_i \in \F$ contained in $\partial_L D \cap C_{\O_{\sspc i}}$. It suffices to define $\eta_i := \phi_i^\geo$ to obtain a geodesic leaf $\eta_i \in \P$ that is not pushing-equivalent to $\lambda$ and crossed by the orbit $\O_{\sspc i}$ at time $M(i)$. At last, let $\phi^* \in \F$ be a leaf contained in $D \cap L(\phi_\lambda)$ which is sufficiently close to $\partial_L D$ so to satisfy $L(\phi^*) \cap D \cap \O_{\sspc i} = \varnothing$ for every index $i \in I_\text{out}$.  Then, it suffices to define $\lambda^* := (\phi^*)^\geo$ to obtain a geodesic leaf $\lambda^* \in \P$ that is pushing-equivalent to $\lambda$ and crossed by the orbit $\O_{\sspc i}$ at time $M(i)$, for every index $i \in I_\text{out}$.
    \end{proof}

     \mycomment{-0.1cm}
    For each index $i \in I$, we denote by $\Gamma_i^\geo = \{(\gamma_i^k)^\geo\}$ to be the transverse geodesic trajectory associated with the orbit $\O_{\sspc i}$, where $(\gamma_i^k)^\geo$ satisfies $\alpha((\gamma_i^k)^\geo) =f^k(x_i)$ and $\omega((\gamma_i^k)^\geo) = f^{k+1}(x_i)$ for a fixed basepoint $x_i \in \O_{\sspc i}$. 
    Using both Theorem \ref{theorem:1B-intro} and the Homma-Schoenflies theorem, we can assume, up to conjugacy and a choice of basepoint $x_i \in \O_{\sspc i}\sspc$, that the following hold:

    \vspace*{0.1cm}
    \begin{itemize}[leftmargin=0.9cm]
        \item[$\sbullet$] $\lambda^*=\{1/2\}\times \R$ oriented downwards.
        
        \vspace*{0.1cm}
        \item[$\sbullet$] $\O_{\sspc i}=\{f^n(x_i) = (n,1) \in \R^2 \mid n \in \Z\}$, for all $i \in I$.
        
        \vspace*{0.1cm}
        \item[$\sbullet$] $\overline{\rule{0cm}{3.6mm}\Gamma_i^\geo} =\R\times\{i\}$ oriented positively in the first coordinate.
    \end{itemize}

    \vspace*{0.1cm}
    \noindent Observe that, with the choice of basepoints above, we have that $M(i) =0$ for all $i \in I_\text{out}$.
    Moreover, for each $i \in I_\text{out}$, we know that the geodesic leaf $\eta_i$ is crossed by at most $r-1$ orbits in $\OO$, otherwise $\eta_i$ would be pushing-equivalent to $\lambda$. By the induction hypothesis, this implies that for each $i \in I_\text{out}$, the geodesic leaf $\eta_i$ is forward-wandering. 

    \vspace*{0.2cm}

    The current setting is illustrated in the figure below.

    \vspace*{-0.1cm}
    \begin{figure}[h!]
    \center
     \mycomment{0.2cm}\begin{overpic}[width=8cm, height=7cm, tics=10]{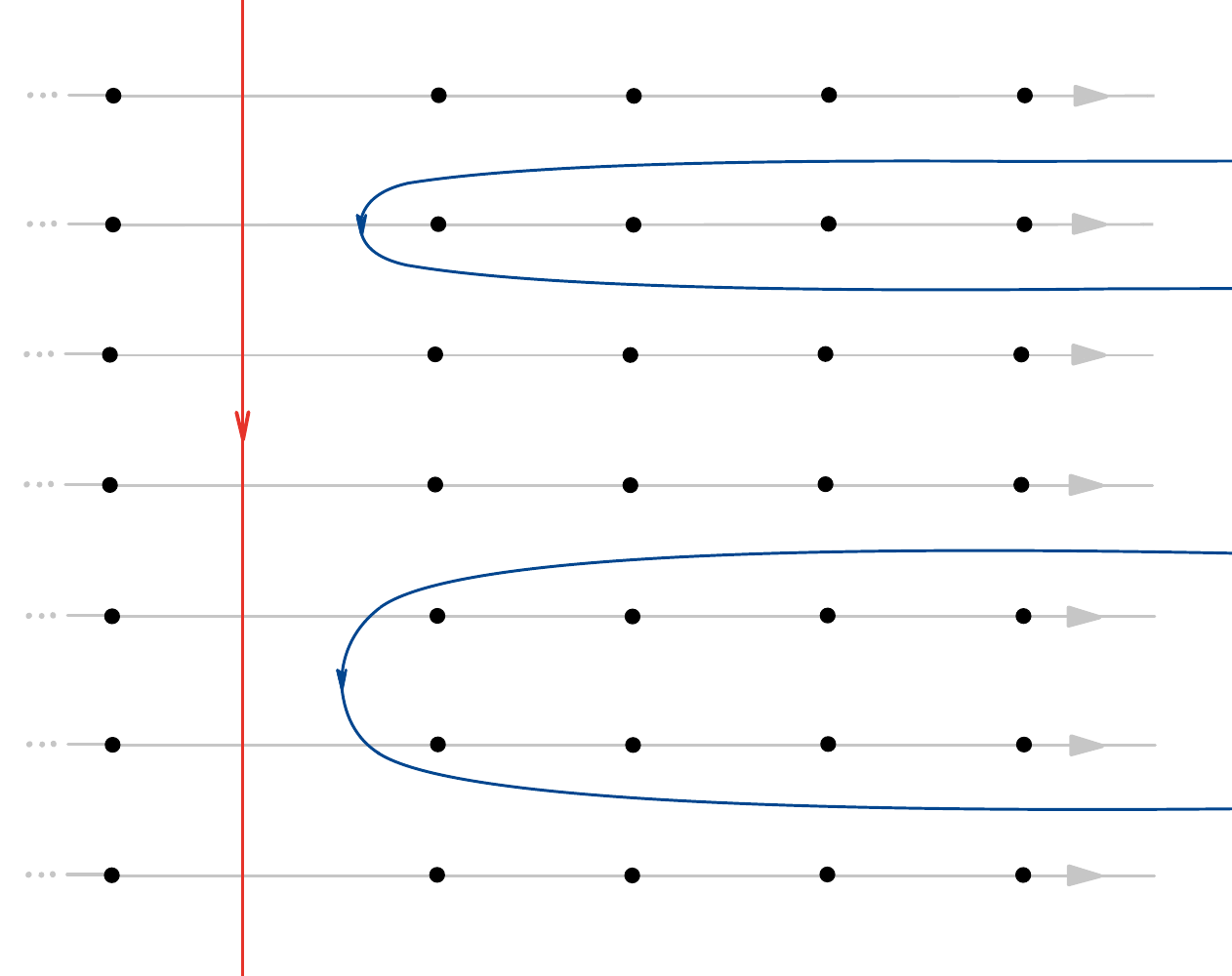}
         \put (11.4,50) {\Large\color{myRED}$\displaystyle \lambda^*$}
         \put (104,77) {\large\color{myBLUE}$\displaystyle \{\eta_i\}_{i \in I_\text{out}}$}
        
\end{overpic}
\end{figure}
    
\vspace*{0.1cm}
    According to Proposition \ref{lemma:simplification_wandering_leaves}, we can assume that for any $i \in I_\text{out}$ the following hold:

    \vspace*{0.1cm}
    \begin{itemize}
        \item The family of geodesic leaves $\G$ is trivial and $f_\geo$-invariant on the left side $L(\eta_i)$,  meaning that, for every geodesic leaf $\eta^\pp \in \G$ lying on the left side $\overline{\rule{0cm}{0.33cm}L(\eta_i)}\sspc$, it holds 
         \mycomment{-0.5cm}
        $$ \eta^\pp \sspc \text{ is pushing-equivalent to }  \sspc \eta_i\spc, \quad  \quad f_\geo(\eta^\pp) \in \G\sspc. \quad \quad $$

        \vspace*{0.1cm}
        \item The transverse geodesic trajectory $\Gamma_i^\geo$ is $f_\geo$-invariant on the left side $L(\eta_i)$, meaning that for any integer $k>0$, it holds that $f_\geo((\gamma_i^k)^\geo) = (\gamma_i^{k+1})^\geo$.
    \end{itemize}

    \vspace{0.2cm}
    % \noindent Therefore, we end up with the setting as illustrated in the figure below.
    \noindent Now that our setting is properly defined, we can finally prove the theorem.
\vspace*{0.3cm}

    Let $N>0$ be a positive integer and, for any index $i\in I$, define $T_N(i)$ to be 
     \mycomment{-0.13cm}
    $$ T_N(i) := \max \{\spc k \geq 0 \spc\mid\spc i(\sspc f^N_\geo (\lambda^*)\sspc, \sspc (\gamma_i^k)^\geo\sspc) \neq 0\spc\}.$$

     \mycomment{-0.25cm}
    \noindent
    Due to the Bounded Intersection Property (Lemma \ref{lemma:limited_winding}), we have that $T_N(i)$ is well-defined for each index $i \in I$. Observe that, since $f^N(x_i) \in R(f^N_\geo(\lambda^*))$ and $f^{N+1}(x_i) \in L(f^N_\geo(\lambda^*))$,  we have that $i(f^N_\geo(\lambda^*),(\gamma_i^M)^\geo)$ must be odd for all $i \in I$. In particular, it follows that
         \mycomment{-0.123cm}
    $$ T_N(i) \geq N\sspc, \quad \forall i \in I.$$

     \mycomment{-0.25cm}
    \noindent
    Next, remark that for all $i \in I_\text{out}$, it holds that $f^N_\geo(\lambda^*) \subset R(f^N_\geo(\eta_i))$. Recall that $f^N_\geo(\eta_i)\in \G$, due to the simplification of $\G$ through the forward-wandering leaf $\eta_i$. Consequently, it follows that
    $i(f^N_\geo(\lambda^*),(\gamma_i^k)^\geo) = 0$ for all $i \in I_\text{out}$ and $k \leq N+1$. This allows us to conclude that
     \mycomment{-0.13cm}
    $$ T_N(i) = N, \quad \forall i \in I_\text{out}.$$

     \mycomment{-0.25cm}
    Let $\lambda_N \in \G$ be the geodesic leaf pushing-equivalent to $\lambda^*$ such that, for all indices $i \in I_\omega$, the orbit $\O_{\sspc i}$ crosses the geodesic leaf $\lambda_N$ at time $T_N^+(i)$ for some integer $T_N^+(i)\geq T_N(i)$.
    We remark that the existence of such a geodesic leaf $\lambda_N$ follows from the definition of $I_\omega$.  In order to unify the notation, we will denote $T_N^+(i) = T_N(i)$ for all $i \in I_\text{out}$.

    \newpage

    \begin{figure}[h!]
    \center
     \mycomment{0.4cm}\begin{overpic}[width=11.7cm, height=8.5cm, tics=10]{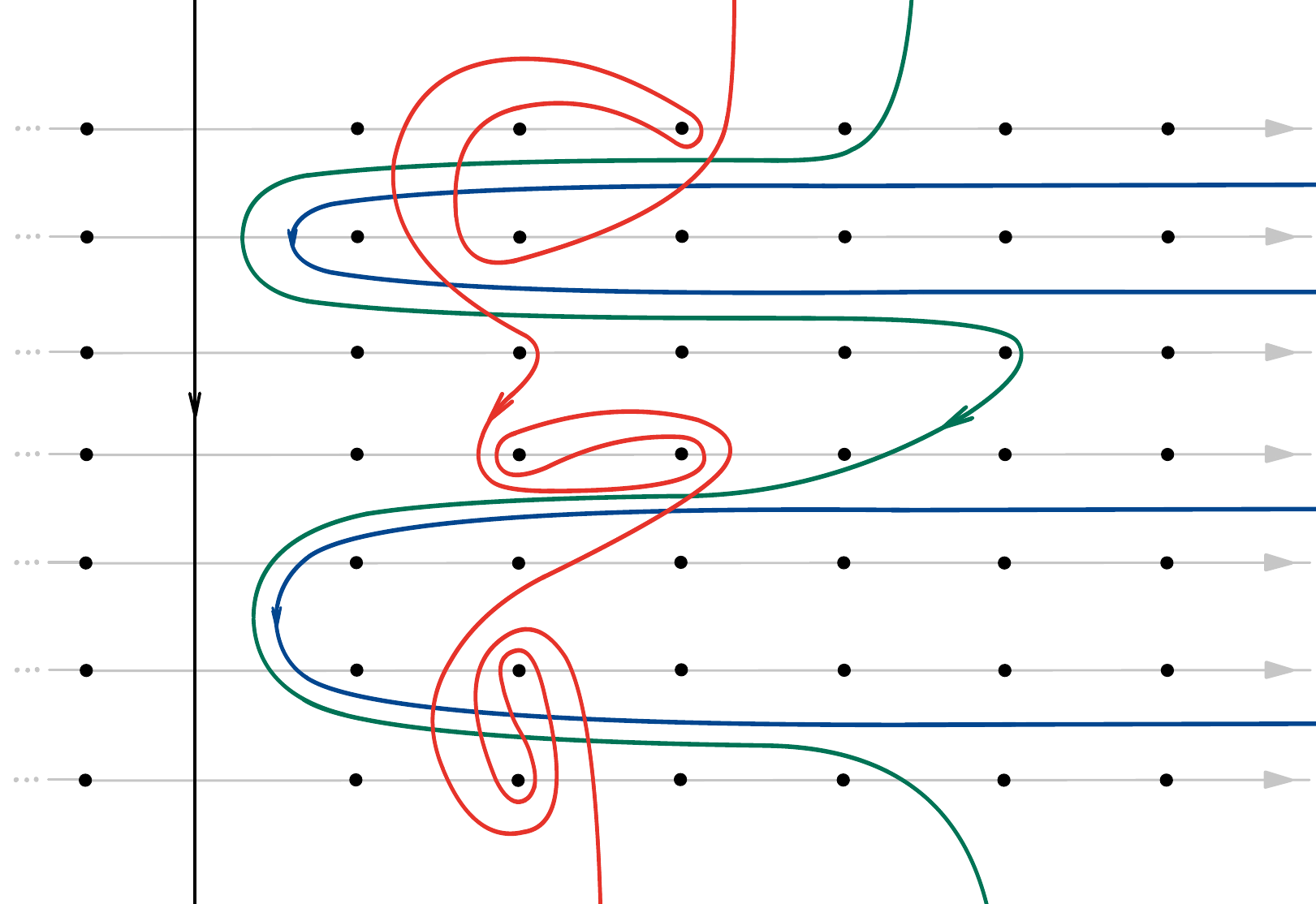}
        \put (9,38.5) {\Large\color{black}$\displaystyle \lambda^*$}
        \put (42,71) {\large\color{myRED}$\displaystyle f^N_\geo(\lambda^*)$}
        \put (63,71) {\large\color{myGREEN}$\displaystyle \lambda_N$}
         \put (104,57) {\large\color{myBLUE}$\displaystyle \{\eta_i\}_{i \in I_\text{out}}$}
\end{overpic}
\end{figure}

    Let us define, for each index $i \in I$, the line segment $v_{N,\sspc i}:[0,1] \longrightarrow \R^2$ that connects the point $(\sspc T_N^+(i)+ \frac{1}{2}\sspc, \sspc i\sspc)\in \R^2$ to the point $(\sspc T_N^+(i+1) + \frac{1}{2}\sspc, \sspc i+1\sspc)\in\R^2$.  Next, we define two half-verticals $v_{N,\sspc \alpha}:(-\infty,0\sspc] \longrightarrow \R^2$ and $v_{N,\sspc \omega}:[\sspc 0,\infty) \longrightarrow \R^2$ determined by the formulas
     \mycomment{0.1cm}
    $$ v_{N,\sspc \alpha}(t) = \left(\sspc T_N^+(r)+\frac{1}{2}\sspc, \sspc r - t\sspc\right) \quad\text{ and } \quad v_{N,\sspc \omega}(t) = \left(\sspc T_N^+(1)+\frac{1}{2}\sspc, \sspc 1 - t\sspc\right).$$

    \noindent Finally, let $V_N:\R \longrightarrow \R^2$ be the piecewise linear topological line given by the concatenation
     \mycomment{-0.5cm}
    $$V_N =  v_{N,\sspc \alpha} \sspc \cdot \sspc v_{N,\sspc r} \sspc \cdot \sspc v_{N,\sspc r-1} \sspc \cdot \sspc ... \sspc \cdot \sspc v_{N,\sspc 1} \sspc \cdot \sspc v_{N,\sspc \omega}.$$

     \mycomment{-0.23cm}
    \noindent
    Note that, since $T_N^+(i)\geq N$ for all $i \in I$, the family of vertical lines $\{V_N\}_{i \in I}$ is locally-finite.
    
 \begin{figure}[h!]
    \center
     \mycomment{0.4cm}\begin{overpic}[width=12cm, height=8.6cm, tics=10]{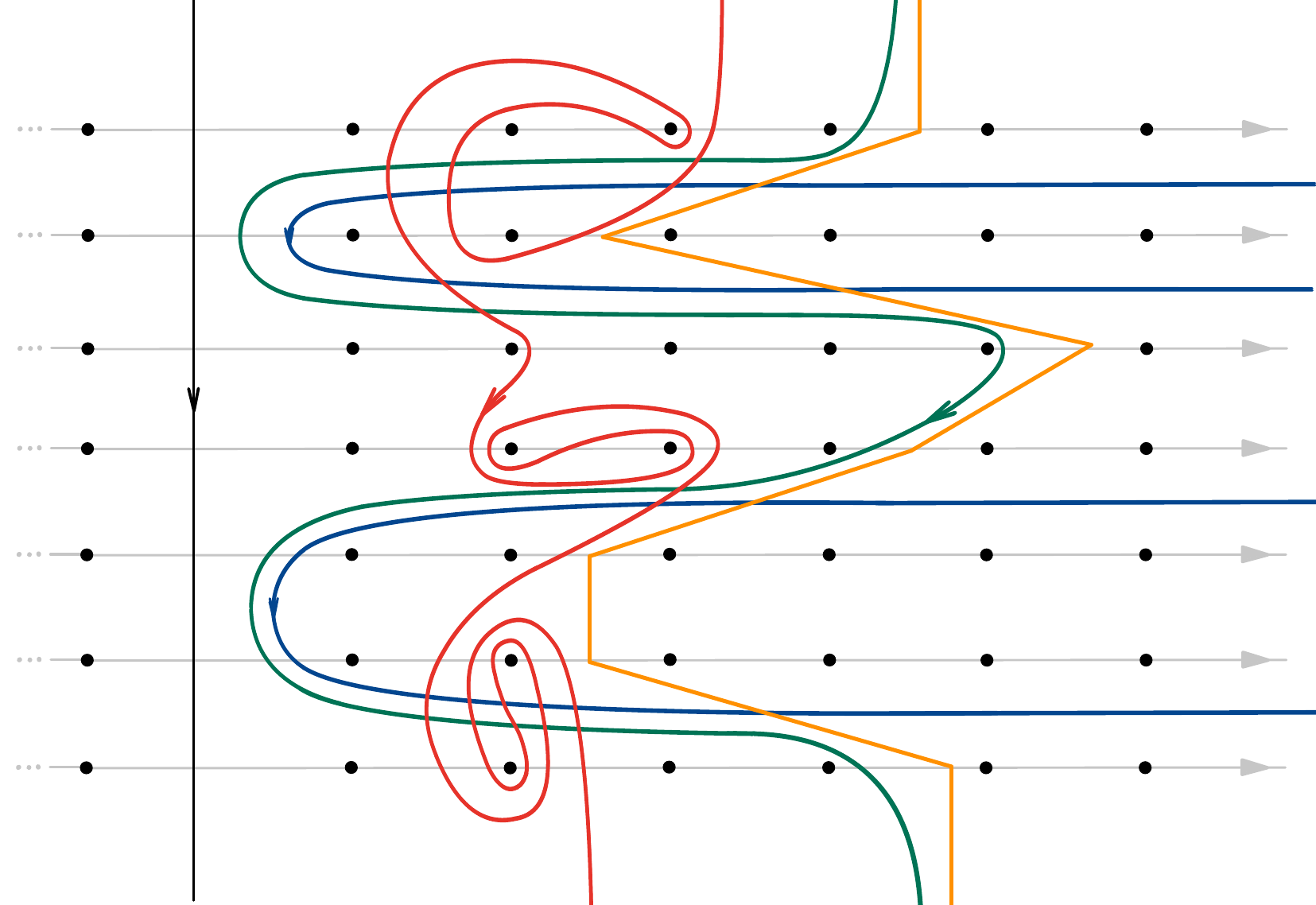}
         \put (9,38.5) {\Large\color{black}$\displaystyle \lambda^*$}
        \put (42,70) {\large\color{myRED}$\displaystyle f^N_\geo(\lambda^*)$}
        \put (62,69) {\large\color{myGREEN}$\displaystyle \lambda_N$}
        \put (72,69) {\large\color{myDARKYELLOW}$\displaystyle V_N$}
         \put (104,56.5) {\large\color{myBLUE}$\displaystyle \{\eta_i\}_{i \in I_\text{out}}$} 
\end{overpic}
\end{figure}
    
 \mycomment{-0.2cm}
    Let $V_N^\geo$ be the geodesic representative of the line $V_N$. It is worth noting that $V_N^\geo$ do not necessarily belong to $\G$, and it will only be used as a tool to prove the theorem. Note that the family $\{V_N^\geo\}_{N \geq 0}$ must be locally-finite, since $\{V_N\}_{i \in I}$ is locally-finite (see Section \ref{sec:hyperbolic_geometry_R2Z}). 
    
    Using the definition of $V_N$ and the fundamental properties of geodesics, we can deduce that each geodesic $V_N^\geo$ satisfies the following basic properties:

    \vspace*{0.1cm}
    \begin{itemize}[leftmargin=1.4cm]
        \item[(0)] $V_N^\geo \cap \lambda^* = \varnothing$.
        
        \vspace*{0.1cm}
        \item[(1)] $V_N^\geo \cap (\gamma_i^k)^\geo \neq \varnothing \iff k = T_N^+(i)$, for all $i \in I$.
        
        \vspace*{0.1cm}
        \item[(2)] $i(\sspc V_N^\geo\sspc, \sspc (\gamma_i^{ T_N^+(i)})^\geo) = 1$, for all $i \in I$.
        
        \vspace*{0.1cm}
        \item[(3)] $L(V_N^\geo) \cap R(\lambda_N) \cap \OO = \varnothing$.
        
        \vspace*{0.1cm}
        \item[(4)] $p \in R(V_N^\geo) \cap L(\lambda^*)\cap \OO \iff p=f^n(x_i)$ for some $i \in I$ and $0<n \leq T_N^+(i)$.
        
        \vspace*{0.1cm}
        \item[(5)] $p \in R(V_N^\geo) \cap L(\lambda_N)\cap \OO \iff p=f^n(x_i)$ for some $i \in I_\text{out}$ and $0<n \leq T_N^+(i)$.
        
        \vspace*{0.1cm}
        \item[(6)] $p \in L(V_N^\geo) \in \OO \iff p=f^n(x_i)$ for some $i \in I$ and $ n > T_N^+(i)$.
    \end{itemize}
    
    \vspace*{0.1cm}
    \noindent Now, we use properties (0)-(6) to prove that the properties (i)-(iii) of Claim \ref{claim4} hold as well.
    
    \vspace*{0.2cm}
    \begin{claim}\label{claim4}
        The geodesic $V_N^\geo$ satisfies the following extra properties:
        \begin{itemize}[leftmargin=1.4cm]
            \item[\textup{\textbf{(i)}}] $L(V_N^\geo) \subset L(\lambda_N)$.
            \item[\textup{\textbf{(ii)}}] $L(V_N^\geo) \subset L(f^N_\geo(\lambda^*))$.
            \item[\textup{\textbf{(iii)}}] $L(\lambda_N) \cap L(f^N_\geo(\lambda^*)) \cap R(V_N^\geo) \cap \OO= \varnothing$.
        \end{itemize}
    \end{claim} 

     \mycomment{-0.4cm}
    \begin{proof}[Proof of Claim \ref{claim4}] 
        \textbf{Proof of item (i):}
        Assume by contradiction that $L(V_N^\geo)\not\subset L(\lambda_N)$.  For that to occur, there should exist at least one point $p \in V_N^\geo$ lying on the right side $R(\lambda_N)$. We have three cases to consider:
        \begin{itemize}[leftmargin=1.4cm]
            \item[(a)] The whole line $V_N^\geo$ lies on the right side $R(\lambda_N)$.
            \item[(b)] The point $p$ lies on a sub-arc of $V_N^\geo$ that is contained in $R(\lambda_N)$, except for its two endpoints, which lie on $\lambda_N$.
            \item[(c)] The point $p$ lies on a sub-halfline of $V_N^\geo$ that is contained in $R(\lambda_N)$, except for its endpoint, which lies on $\lambda_N$.
        \end{itemize}

        \begin{figure}[h!]
    \center
     \mycomment{0.3cm}\begin{overpic}[width=3.6cm, height=3.8cm, tics=10]{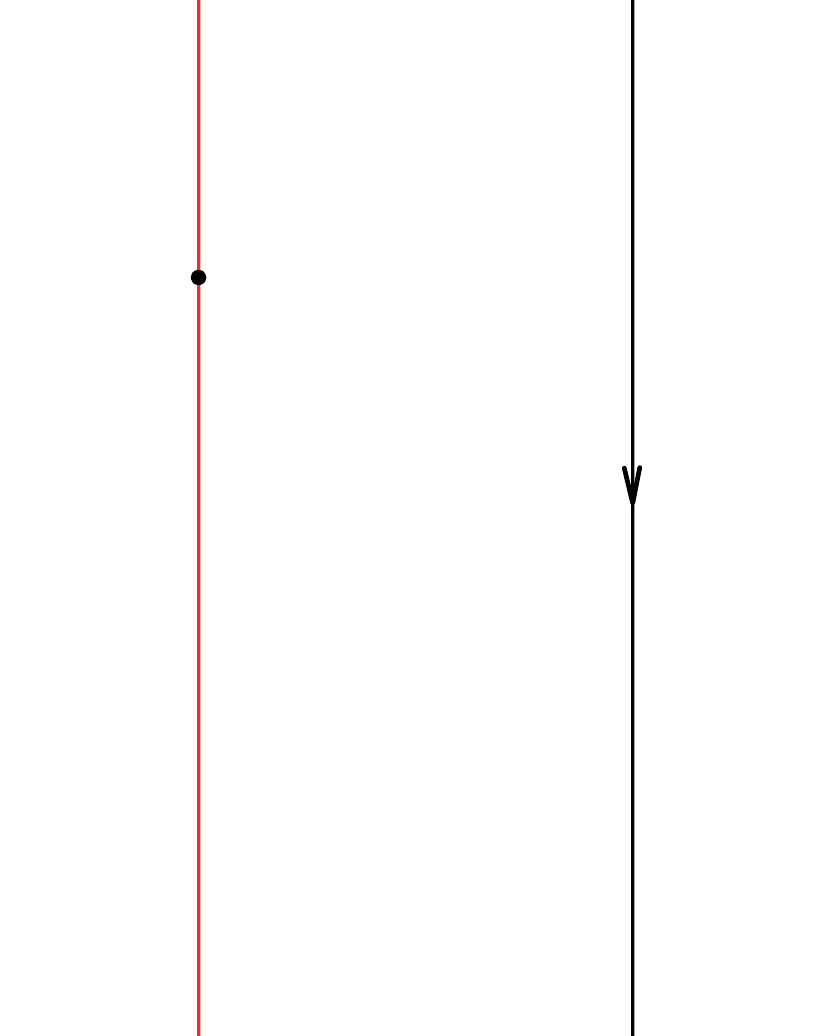}
         \put (-5,90) {\large\color{black}$\displaystyle $(a)}
         \put (26,10) {\large\color{myRED}$\displaystyle V_N^\geo$}
         \put (77,10) {\large\color{black}$\displaystyle \lambda_N$}
         \put (27,72) {\large\color{black}$\displaystyle p$}
\end{overpic}\hspace*{1cm}
\begin{overpic}[width=3.6cm, height=3.8cm, tics=10]{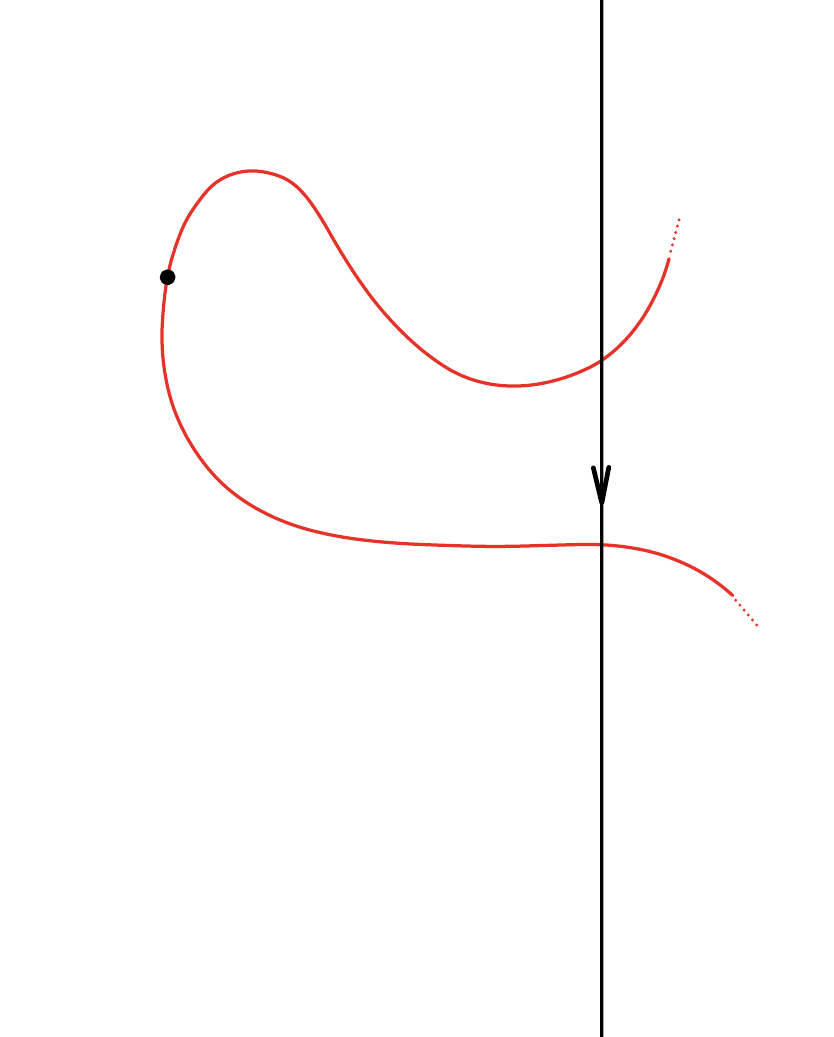}
         \put (-5,90) {\large\color{black}$\displaystyle $(b)}
         \put (26,71) {\large\color{black}$\displaystyle p$}
         \put (26,37) {\large\color{myRED}$\displaystyle V_N^\geo$}
         \put (72,10) {\large\color{black}$\displaystyle \lambda_N$}
\end{overpic}\hspace*{1.5cm}
\begin{overpic}[width=3.6cm, height=3.8cm, tics=10]{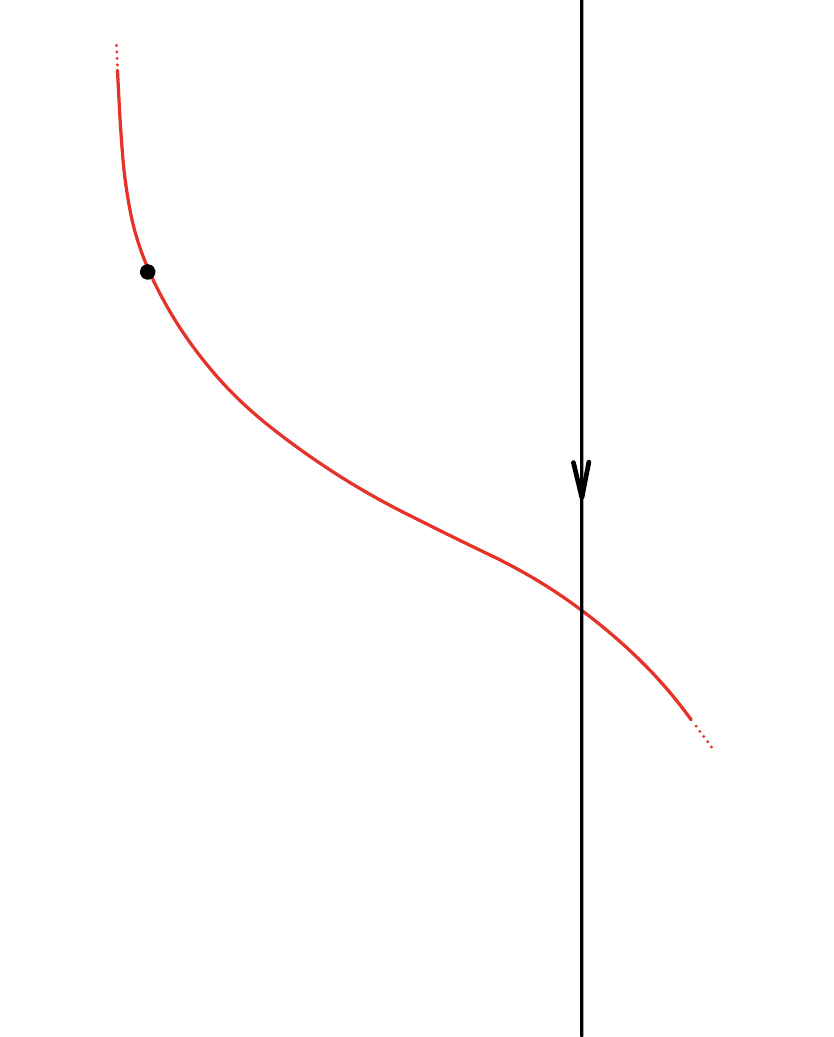}
         \put (-5,90) {\large\color{black}$\displaystyle $(c)}
         \put (24,73.5) {\large\color{black}$\displaystyle p$}
         \put (26,40) {\large\color{myRED}$\displaystyle V_N^\geo$}
         \put (70,10) {\large\color{black}$\displaystyle \lambda_N$}
\end{overpic}
\end{figure}

         \mycomment{-0.1cm}
        We begin by analyzing case (a). %Observe that, if $I_\omega=\varnothing$, then $\lambda_N = \lambda^*$ and the claim follows trivially. Thus, we can assume that $I_\omega \neq \varnothing$. 
        From the definition of $T_N^+(i)$ and property (6), we know that the points in $\{f^n(x_i)\}_{i \in I, n>T_N^+(i)}$ lie simultaneously on $L(\lambda_N)$ and $L(V_N^\geo)$. 
        This means that, in case (a), the only acceptable orientation for the line $V_N^\geo$ to have while being in $R(\lambda_N)$  is the one implying $L(\lambda_N) \subset L(V_N^\geo)$. Using Homma-Schoenflies theorem, we can see both geodesic lines $V_N^\geo$ and $\lambda_N$ as disjoint vertical lines.
        Due to property (3), we can perform a horizontal isotopy on $\R^2 \setminus \OO$ that moves $V_N^\geo$ into $\lambda_N$. In other words, $V_N^\geo$ and $\lambda_N$ are geodesic lines in the same isotopy class, which implies $V_N^\geo = \lambda_N^\geo$.  This contradicts (a).
    
            Now we analyze case (b). In this case, $V_N^\geo$ and $\lambda_N$ form a bigon on the right side $R(\lambda_N)$. This bigon is a closed disk $B\subset \R^2$ with its interior contained in $R(\lambda_N)$, and its boundary composed of a sub-arc of $V_N^\geo$ and a sub-arc of $\lambda_N$. By properties (1) and (2), we see that $V_N^\geo$ intersects each $\Gamma_i^\geo$ exactly once and this intersection occurs at $(\gamma_i^{ T_N^+(i)})^\geo \in \Gamma_i^\geo$. Revisiting the definition of the geodesic $\lambda_N$, we see that $\lambda_N$ intersects each $\Gamma_i^\geo$ exactly once and, when $i \in I_\omega$,  this intersection occurs exactly at the line $(\gamma_i^{ T_N^+(i)})^\geo \in \Gamma_i^\geo$. This implies that, for every $i \in I_\omega$, if $\textup{int}(B) \cap \Gamma_i^\geo \neq\varnothing$, then $B \cap \Gamma_i^\geo$ must be a sub-arc of $(\gamma_i^{ T_N^+(i)})^\geo$. Consequently, we get that
             \mycomment{-0.13cm}
            $$ \textup{int}(B) \cap \O_i = \varnothing, \quad \forall i \in I_\omega.$$
            
             \mycomment{-0.2cm}
            From property (0), we know that $V_N^\geo$ is disjoint from $\lambda^*$. This implies that $\textup{int}(B)$ is not only contained in $R(\lambda_N)$, but in fact $\textup{int}(B) \subset R(\lambda_N) \cap L(\lambda^*)$. From the definition of $\lambda_N$, we know that $R(\lambda_N) \cap L(\lambda^*) \cap \OO \subset \bigcup_{i \in I_\omega} \O_i$. This allows us to conclude that
            $$ \textup{int}(B) \cap \OO = \bigcup_{i \in I_\omega} (\textup{int}(B) \cap \O_i) = \varnothing.$$

            Since $V_N^\geo$ and $\lambda_N$ are geodesic lines, the bigon $B$ needs to contain a point $q \in \OO$ in its interior, otherwise we would be able to perform an isotopy on $\R^2 \setminus \OO$ that decreases the intersection number between $V_N^\geo$ and $\lambda_N$, contradicting the fact that geodesic lines minimize intersections in their isotopy class. However, this is not what happens, since we proved above that $\textup{int}(B) \cap \OO = \varnothing$. This contradicts (b).

        Finally, we analyze case (c). In this case, let $\ell$ be a sub-halfine of $V_N^\geo$ contained in $R(\lambda_N)$, except for its endpoint that lies on $\lambda_N$. Since $V_N^\geo$ is disjoint from $\lambda^*$, we can see that the set $L(\lambda^*) \cap R(\lambda_N)\setminus \ell$ has two connected components, one bounded by $\lambda^*$, $\lambda_N$ and $\ell$, and the other only bounded by $\lambda_N$ and $\ell$. Let $B\subset \R^2$ denote the second component. We can use a similar argument to the one in case (b) to show that $\textup{int}(B) \cap \OO$ is empty. And this property leads to a contradiction of (c) by the same reasoning as in case (b). This proves item (i).

         \mycomment{0.2cm}

        \textbf{Proof of item (ii):}
        Similarly, we assume by contradiction that $L(V_N^\geo)\not\subset L(f^N_\geo(\lambda^*))$.  For that to occur, there should exist at least one point $p \in f^N_\geo(\lambda^*)$ lying on the side $L(V_N^\geo)$. 
        Again, we have three cases to consider:
        \begin{itemize}[leftmargin=1.4cm]
            \item[(d)] The whole line $f^N_\geo(\lambda^*)$ lies on the left side $L(V_N^\geo)$.
            \item[(e)] The point $p$ lies on a sub-arc of $f^N_\geo(\lambda^*)$ that is contained in $L(V_N^\geo)$, except for its two endpoints, which lie on $V_N^\geo$.
            \item[(f)] The point $p$ lies on a sub-halfline of $f^N_\geo(\lambda^*)$ that is contained in $L(V_N^\geo)$, except for its endpoint, which lies on $V_N^\geo$.
        \end{itemize}

        \begin{figure}[h!]
    \center
     \mycomment{0.3cm}\begin{overpic}[width=3.6cm, height=3.8cm, tics=10]{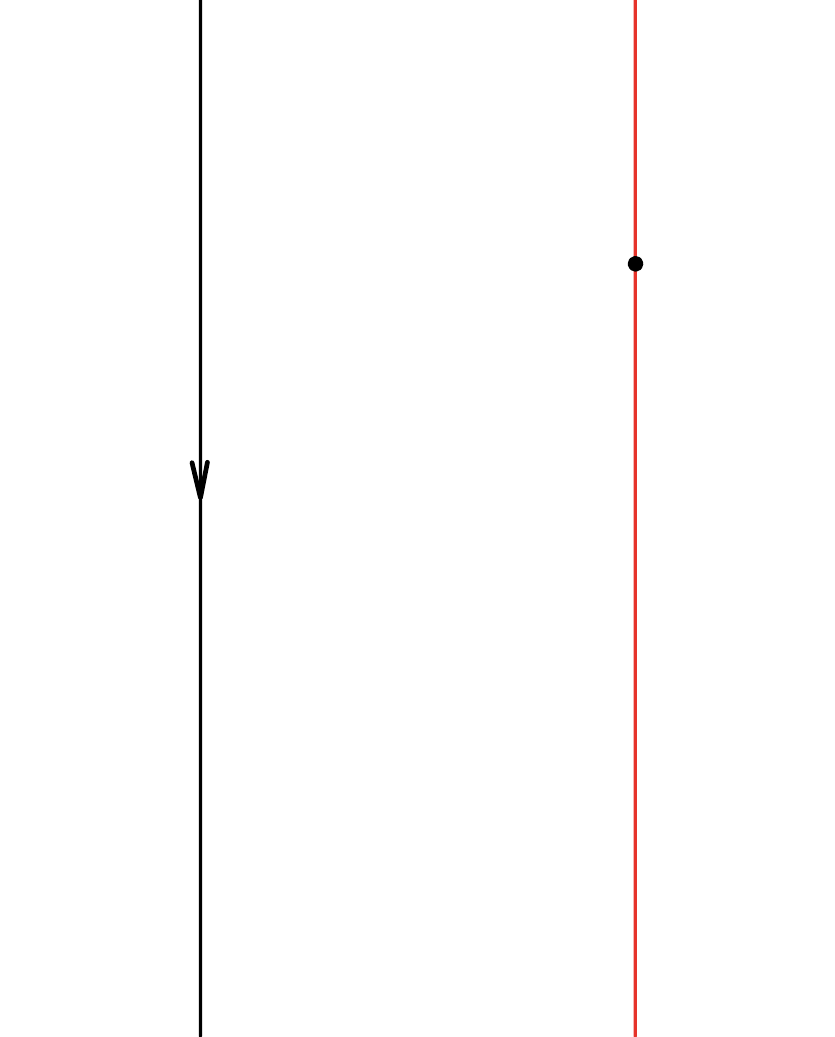}
         \put (-5,90) {\large\color{black}$\displaystyle $(d)}
         \put (76,10) {\large\color{myRED}$\displaystyle f^N_\geo(\lambda^*)$}
         \put (26.5,5) {\large\color{black}$\displaystyle V_N^\geo$}
         \put (76,72) {\large\color{black}$\displaystyle p$}
\end{overpic}\hspace*{1.5cm}
\begin{overpic}[width=3.6cm, height=3.8cm, tics=10]{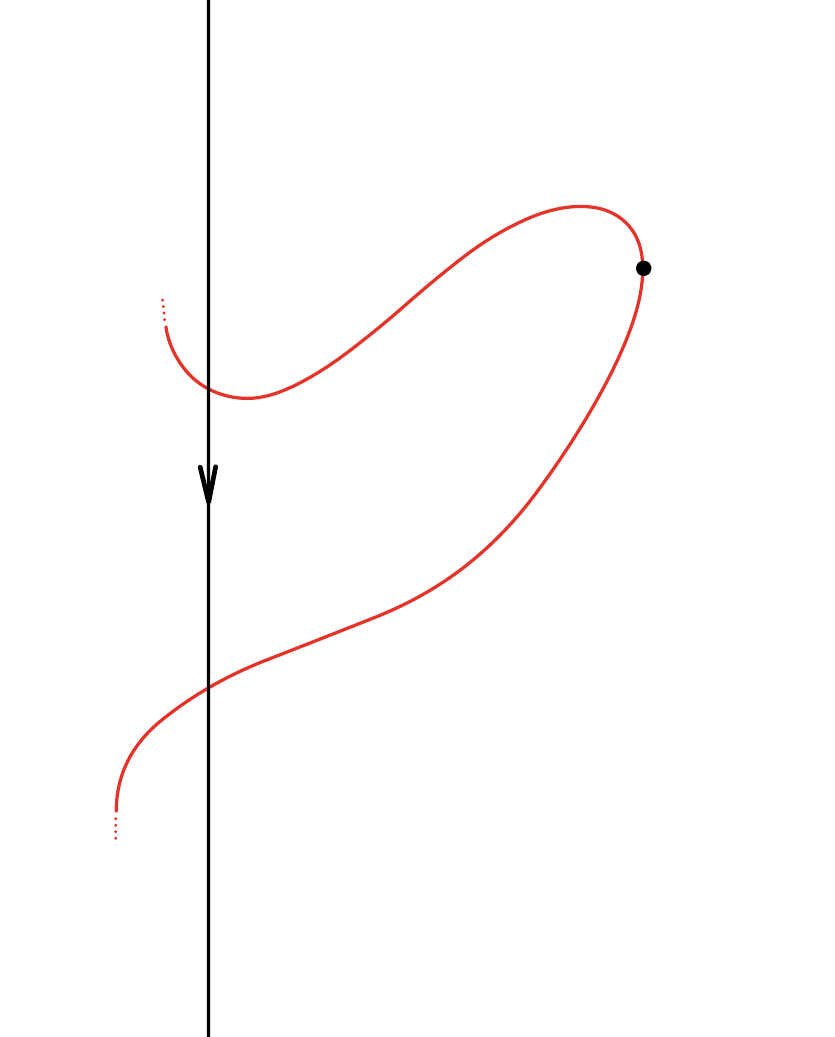}
         \put (-5,90) {\large\color{black}$\displaystyle $(e)}
         \put (78,71) {\large\color{black}$\displaystyle p$}
         \put (50,32) {\large\color{myRED}$\displaystyle f^N_\geo(\lambda^*)$}
         \put (27,5) {\large\color{black}$\displaystyle V_N^\geo$}
\end{overpic}\hspace*{1cm}
\begin{overpic}[width=3.6cm, height=3.8cm, tics=10]{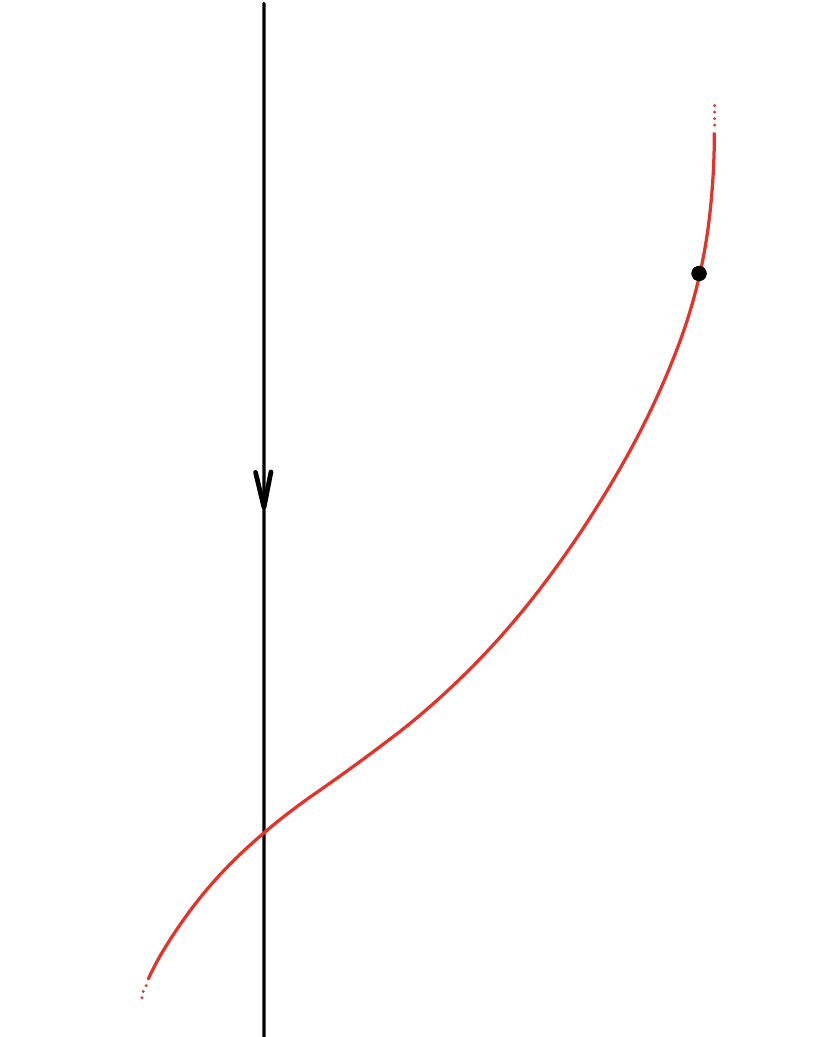}
         \put (-5,90) {\large\color{black}$\displaystyle $(f)}
         \put (83,72) {\large\color{black}$\displaystyle p$}
         \put (55,27) {\large\color{myRED}$\displaystyle f^N_\geo(\lambda^*)$}
         \put (33,5) {\large\color{black}$\displaystyle V_N^\geo$}
\end{overpic}
\end{figure}

        Case (d) is very similar to the case (a). The only difference is that, in case (d), the only acceptable orientation for the line $f^N_\geo(\lambda^*)$ to have while being in $L(V_N^\geo)$ is the one implying $L(f^N_\geo(\lambda^*)) \subset L(V_N^\geo)$. Apart from that, we can use the same argument as in case (a)  to deduce that $f^N_\geo(\lambda^*)$ and $V_N^\geo$ are geodesic lines in the same isotopy class. From unicity of geodesic representatives, we conclude that $f^N_\geo(\lambda^*) = V_N^\geo$. This contradicts (d).

         \mycomment{0cm}

        Case (e) resembles case (b) but it relies on distinct properties to attain a contradiction. 
        In case (e), the geodesics $f^N_\geo(\lambda^*)$ and $V_N^\geo$ form a bigon on the left side $L(V_N^\geo)$. This bigon is a closed disk $B\subset \R^2$ with its interior contained in $L(V_N^\geo)$, and its boundary composed of a sub-arc of $f^N_\geo(\lambda^*)$ and a sub-arc of $V_N^\geo$. By the definition of $T_N^+(i)$ and property (1), we have 
         \mycomment{-0.4cm}
        $$ f^N_\geo(\lambda^*) \cap (\gamma_i^k)^\geo = V_N^\geo \cap (\gamma_i^k)^\geo = \varnothing\sspc, \quad \forall i \in I, \ k > T_N^+(i).$$

         \mycomment{-0.2cm}
        \noindent
        Similarly to case (b), we can use the fact that $\textup{int}(B) \subset L(V_N^\geo)$ and property (6) to prove
         \mycomment{-0.15cm}
        $$ \textup{int}(B) \cap \OO = \varnothing.$$

         \mycomment{-0.25cm}
        However, as in case (b), this leads to a contradiction, as it would allow us to perform an isotopy on $\R^2 \setminus \OO$ that reduces the intersection number between $f^N_\geo(\lambda^*)$ and $V_N^\geo$, thus contradicting the minimal intersection property of geodesics. This contradicts (e).

        Finally, we analyze case (f). The argument is similar to that of case (c), but more involved.
        In this case, let us denote by $\ell$ the sub-halfline of $f^N_\geo(\lambda^*)$ contained in $L(V_N^\geo)$, except for its endpoint that lies on $V_N^\geo$. Note that the set $L(V_N^\geo)\setminus \ell$ has exactly two connected components, denoted by $B_1$ and $B_2$. First, we argue that for each $i \in I$, the forward-orbit $\O_i \cap L(V_N^\geo)$ is entirely contained in either $B_1$ or $B_2$. Indeed, otherwise there would exist $i \in I$ and $k > T_N^+(i)$  such that $f^k(x_i) \in B_1$ and $f^{k+1}(x_i) \in B_2$, or vice versa. This would imply that $f^N_\geo(\lambda^*)$ intersects $(\gamma_i^k)^\geo$, contradicting the definition of $T_N^+(i)$. This proves the property above.

        Therefore, some forward orbits $\O_i \cap L(V_N^\geo)$ are contained in $B_1$, while others are in $B_2$. Moreover, both sets $B_1\cap \OO$ and $B_2\cap \OO$ must be non-empty, otherwise we would be able to perform an isotopy on $\R^2 \setminus \OO$ that reduces an intersection between $f^N_\geo(\lambda^*)$ and $V_N^\geo$, just as in case (c), contradicting the minimal intersection property of geodesics.

        Now, by the definition of $V_N$, we observe that $\OO \cap L(V_N^\geo) \subset L(f^N_\geo(\lambda^*))$. This means that both $B_1\cap \OO$ and $B_2 \cap \OO$ must lie on the left side $L(f^N_\geo(\lambda^*))$. Implying that $\ell$ cannot be the only segment of $f^N_\geo(\lambda^*)$ contained in $L(V_N^\geo)$. In case (e), we showed that $f^N_\geo(\lambda^*)$ cannot form bigons with $V_N^\geo$ on the left side $L(V_N^\geo)$. Thus, we conclude that there must exist another sub-halfline $\ell^\pp$ of $f^N_\geo(\lambda^*)$ that is disjoint from $\ell$ and contained in $L(V_N^\geo)$, except for its endpoint lying on $V_N^\geo$. 
        
        This new sub-halfline $\ell^\pp$ divides either $B_1$ or $B_2$ into two new connected components. Without loss of generality, assume that $\ell^\pp$ divides $B_1$, and these new connected components are denoted by $B_{1,L}:=B_1 \cap L(f^N_\geo(\lambda^*))$ and $B_{1,R}:=B_1 \cap R(f^N_\geo(\lambda^*))$. Additionally, define the set $C:= R(f^N_\geo(\lambda^*))\cap R(V_N^\geo)$. Observe that $B_{1,R} \cap \OO = \varnothing$, because $\OO \cap L(V_N^\geo) \subset L(f^N_\geo(\lambda^*))$. At last, since $f^N_\geo(\lambda^*)$ cannot form bigons with $V_N^\geo$, we have that $C \cup B_{1,R} = R(f^N_\geo(\lambda^*))$. 
        
        This setting is illustrated below.

        \vspace*{-0.1cm}

    \begin{figure}[h!]
    \center
     \mycomment{0.2cm}\begin{overpic}[width=8.2cm, height=6cm, tics=10]{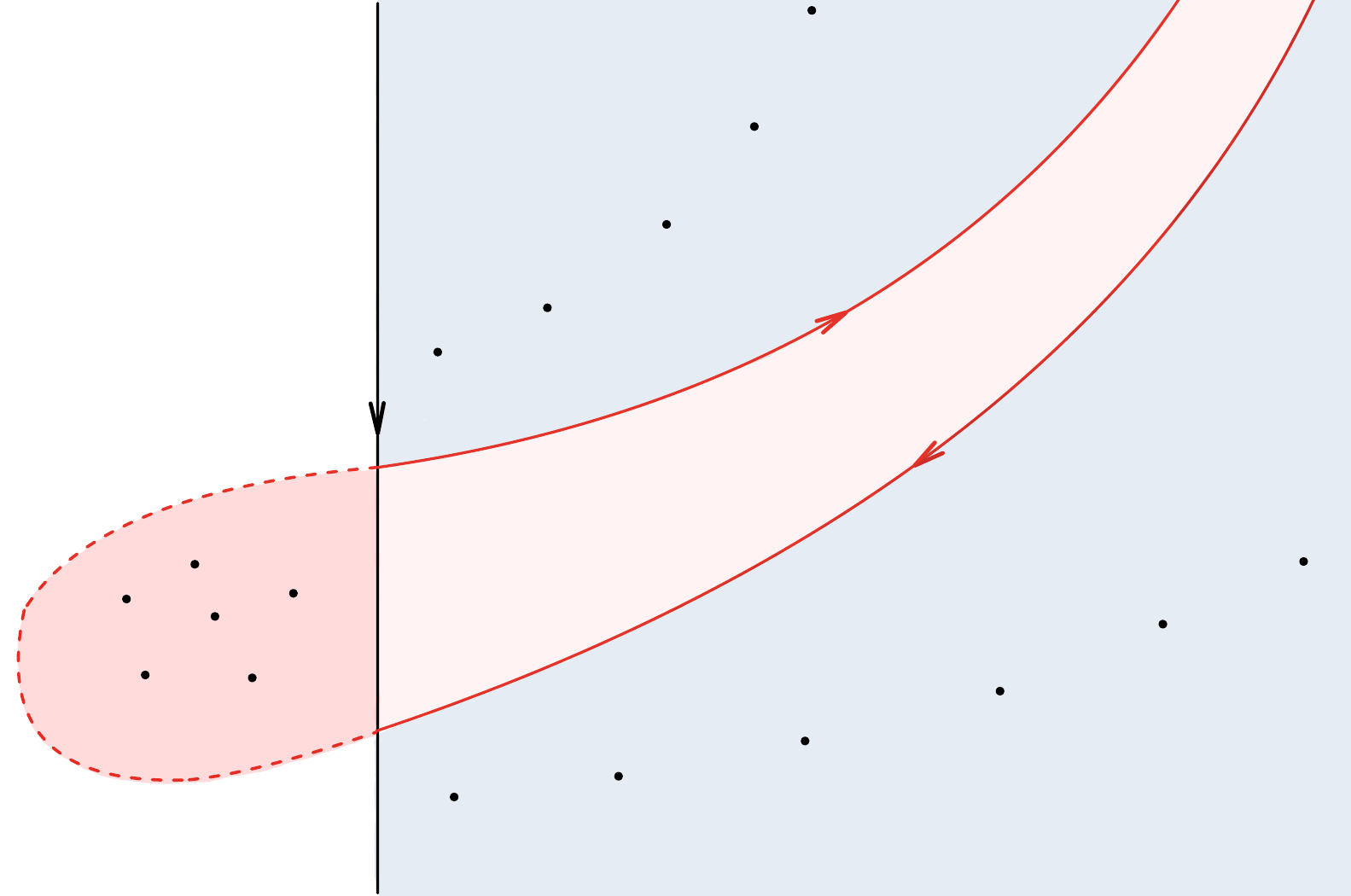}
        \put (12,12.5) {\large\color{myRED}$\displaystyle C$}
        \put (20,58.5) {\large\color{black}$\displaystyle V_N^\geo$}
        \put (50,31.5) {\large\color{myRED}$\displaystyle B_{1,R}$}
        \put (62,58.5) {\large\color{myBLUE}$\displaystyle B_{1,L}$}
        \put (80,30.5) {\large\color{myBLUE}$\displaystyle B_2$}
        %\put (-14.5,18) {\large\color{myRED}$\displaystyle f^N_\geo(\lambda^*)$}
    \end{overpic}
    \end{figure}

        Observe that $C$ is a compact set. Since $\OO$ is a locally-finite subset of the plane, we conclude that $\OO \cap C$ is a finite set. Since $B_{1,R}$ is disjoint from $\OO$, and $R(f^N_\geo(\lambda^*)) = C \cup B_{1,R}$, this implies that $\OO \cap R(f^N_\geo(\lambda^*))$ is finite. This is a contradiction, and finally proves item (ii).

        \vspace*{0.3cm}

        \textbf{Proof of item (iii):} Let $p \in R(V_N^\geo)\cap L(\lambda^*) \cap \OO$. According to property (4), we know that $p=f^n(x_i)$ for some index $i \in I$ and $0<n \leq T_N^+(i)$. According to Lemma \ref{lemma:geodesic_leaves}, we have
         \mycomment{-0.1cm}
        $$f^n(x_i) \in R(f^N_\geo(\lambda^*))\sspc, \quad \forall n\leq N, \ i \in I.$$ 

         \mycomment{-0.2cm}
        \noindent Since $T_N^+(i) = N$ for all $i \in I_\text{out}$, we conclude that
         \mycomment{-0.1cm}
        $$R(V_N^\geo) \cap L(f^N_\geo(\lambda^*)) \cap \O_i = \varnothing, \quad \forall i \in I_\text{out}.$$

         \mycomment{-0.2cm}
        \noindent Meanwhile, it follows from property (5) that the geodesics $V_N^\geo$ and $\lambda_N$ satisfy
         \mycomment{-0.1cm}
        $$ R(V_N^\geo) \cap L(\lambda_N) \cap \O_i = \varnothing, \quad \forall i \in I_\omega.$$

         \mycomment{-0.2cm}
        \noindent 
        Therefore, we get
        $ R(V_N^\geo) \cap L(\lambda_N) \cap L(f^N_\geo(\lambda^*)) \cap \OO = \varnothing\sspc$, and this concludes the proof.
    \end{proof}

     \mycomment{-0.1cm}
    Now we can return within the proof of the main theorem.

    \noindent According to the Pushing Lemma, there exists a positive integer $M_N>N$ such that
     \mycomment{-0.12cm}
    $$ L(f^{M_N}_\geo(\lambda^*)) \subset L(\lambda_N).$$

        The next claim establishes a crucial property of the geodesic line $f^{M_N}_\geo(\lambda^*)$.

    \begin{claim}\label{claim6}
        The geodesic line $f^{M_N}_\geo(\lambda^*)$ satisfies $L(f^{M_N}_\geo(\lambda^*)) \subset L(V_N^\geo)$.
    \end{claim}

     \mycomment{-0.3cm}

    \begin{proof}[Proof of Claim \ref{claim6}]
        Similarly to Claim \ref{claim4}, assume by contradiction $L(f^{M_N}_\geo(\lambda^*)) \not\subset L(V_N^\geo)$.  For that to occur, there should exist at least one point $p \in f^{M_N}_\geo(\lambda^*)$ lying on the side $R(V_N^\geo)$. 
        
        \noindent This leaves us with the following three cases to consider:
        \begin{itemize}[leftmargin=1.4cm]
            \item[(a)] The whole line $f^{M_N}_\geo(\lambda^*)$ lies on the right side $R(V_N^\geo)$.
            \item[(b)] The point $p$ lies on a sub-arc of $f^{M_N}_\geo(\lambda^*)$ that is contained in $R(V_N^\geo)$, except for its two endpoints, which lie on $V_N^\geo$.
            \item[(c)] The point $p$ lies on a sub-halfline of $f^{M_N}_\geo(\lambda^*)$ that is contained in $R(V_N^\geo)$, except for its endpoint, which lies on $V_N^\geo$.
        \end{itemize}

        \begin{figure}[h!]
    \center
     \mycomment{0.3cm}\begin{overpic}[width=3.5cm, height=3.6cm, tics=10]{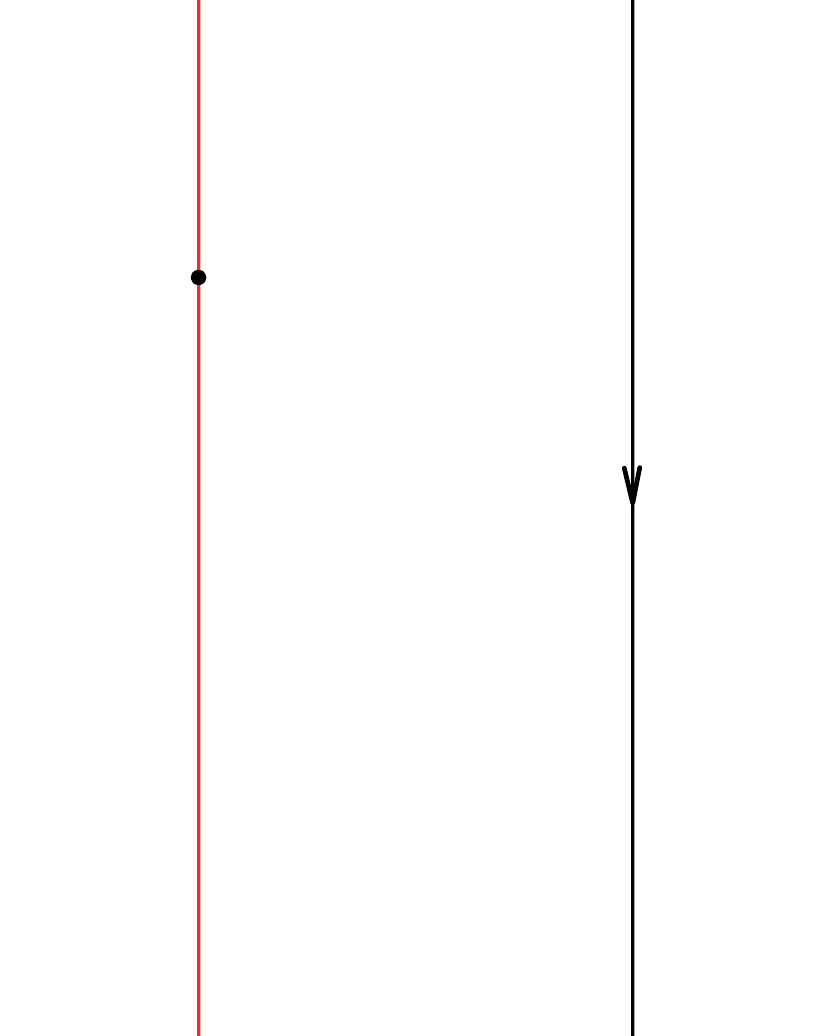}
         \put (-5,90) {\large\color{black}$\displaystyle $(a)}
         \put (26,10) {\large\color{myRED}$\displaystyle f^{M_N}_\geo(\lambda^*)$}
         \put (77,10) {\large\color{black}$\displaystyle V_N^\geo$}
         \put (27,72) {\large\color{black}$\displaystyle p$}
\end{overpic}\hspace*{1cm}
\begin{overpic}[width=3.5cm, height=3.6cm, tics=10]{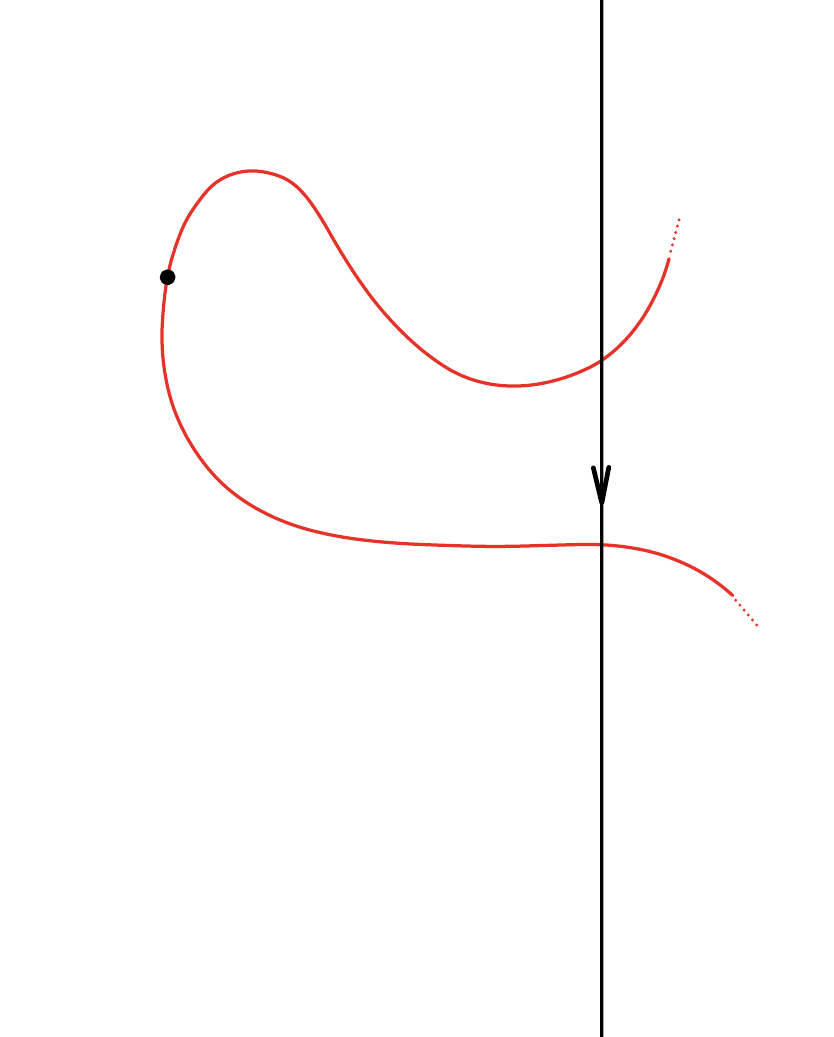}
         \put (-5,90) {\large\color{black}$\displaystyle $(b)}
         \put (25,70) {\large\color{black}$\displaystyle p$}
         \put (22,35) {\large\color{myRED}$\displaystyle f^{M_N}_\geo(\lambda^*)$}
         \put (72,10) {\large\color{black}$\displaystyle V_N^\geo$}
\end{overpic}\hspace*{1.5cm}
\begin{overpic}[width=3.5cm, height=3.6cm, tics=10]{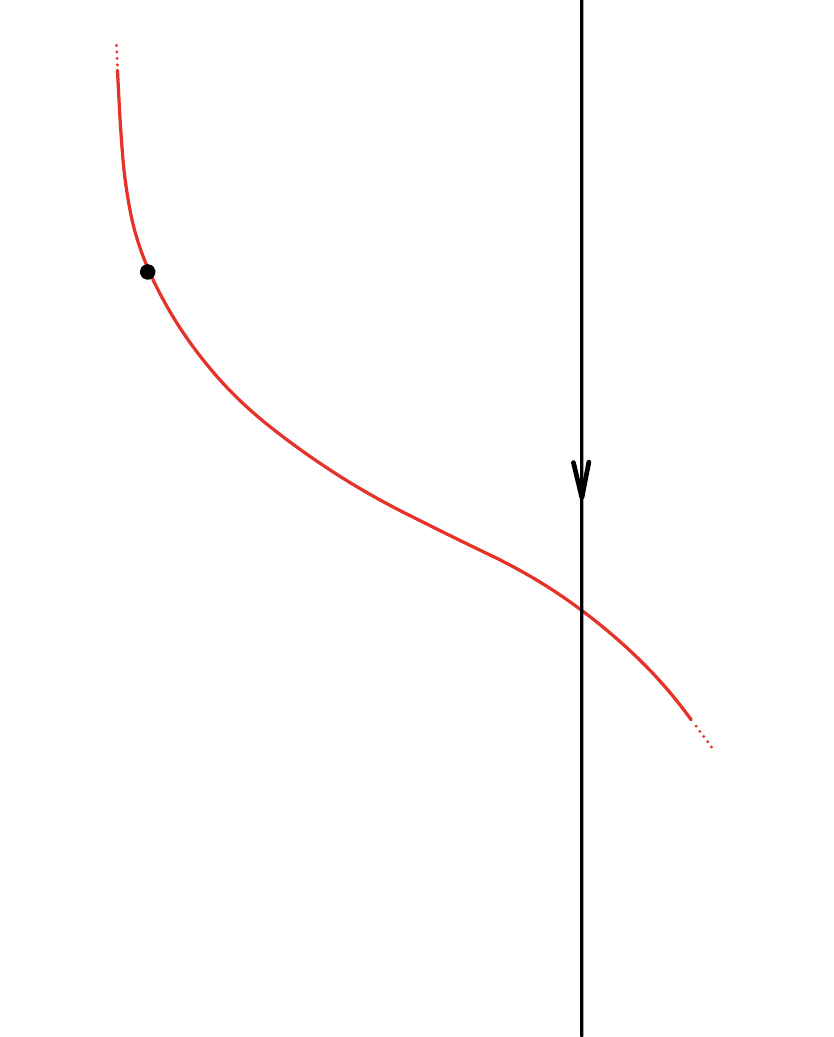}
         \put (-5,90) {\large\color{black}$\displaystyle $(c)}
         \put (24,73.5) {\large\color{black}$\displaystyle p$}
         \put (19,35) {\large\color{myRED}$\displaystyle f^{M_N}_\geo(\lambda^*)$}
         \put (70,10) {\large\color{black}$\displaystyle V_N^\geo$}
\end{overpic}
\end{figure}

        We begin by analyzing case (a). %Observe that, if $I_\omega=\varnothing$, then $\lambda_N = \lambda^*$ and the claim follows trivially. Thus, we can assume that $I_\omega \neq \varnothing$. 
        Observe that, for any index $i \in I$ and any sufficiently large integer $k>0$ satisfies $f^k(x_i) \in L(V_N^\geo) \cap L(f^{M_N}_\geo(\lambda^*))$. 
        This means that, in case (a), the only acceptable orientation for the line $f^{M_N}_\geo(\lambda^*)$ to have while being in $R(V_N^\geo)$ is the one implying $L(V_N^\geo)\subset L(f^{M_N}_\geo(\lambda^*))$. 
        By construction, the geodesic $f^{M_N}_\geo(\lambda^*)$ satisfies
         \mycomment{-0.1cm}
        $$ L(f^{M_N}_\geo(\lambda^*)) \subset L(\lambda_N) \cup L(f^N_\geo(\lambda^*)).$$

         \mycomment{-0.2cm}
        \noindent Using item (iii) of Claim \ref{claim4}, we obtain that
         \mycomment{-0.1cm}
        $$ L(f^{M_N}_\geo(\lambda^*)) \cap L(V_N^\geo) \cap \OO \subset (L(\lambda_N) \cup L(f^N_\geo(\lambda^*))) \cap R(V_N^\geo) \cap \OO = \varnothing\sspc.$$
        
         \mycomment{-0.2cm}
        \noindent Thus, we can use Homma-Schoenflies theorem, we can see both lines $f^{M_N}_\geo(\lambda^*)$ and $V_N^\geo$ as disjoint vertical lines with no points in $\OO$ between them. This means that we can perform a horizontal isotopy on $\R^2 \setminus \OO$ that moves $V_N^\geo$ into $f^{M_N}_\geo(\lambda^*)$, thus showing that $V_N^\geo$ and $f^{M_N}_\geo(\lambda^*)$ are in the same isotopy class. Since they are geodesic lines, this implies that $V_N^\geo = f^{M_N}_\geo(\lambda^*)$ and, consequently, this contradicts case (a).

        Now we analyze case (b). In this case, $f^{M_N}_\geo(\lambda^*)$ and $V_N^\geo$ form a bigon on the side $R(V_N^\geo)$. This bigon is a closed disk $B\subset \R^2$ with its interior contained in $R(\lambda_N)$, and its boundary composed of a sub-arc of $V_N^\geo$ and a sub-arc of $f^{M_N}_\geo(\lambda^*)$ containing the point $p$. 
        Since both lines $f^{M_N}_\geo(\lambda^*)$ and $V_N^\geo$ are geodesics, the interior of the bigon $B$ needs to contain a point $q \in \OO$,  otherwise we would be able to perform an isotopy on $\R^2 \setminus \OO$ that decreases the number of intersections between $f^{M_N}_\geo(\lambda^*)$ and $V_N^\geo$, contradicting the fact that geodesic lines are in minimal position. Since $q \in \textup{int}(B) \subset R(V_N^\geo)$, we can use (iii) of Claim \ref{claim4} to conclude that
        $$ q \in R(\lambda_N) \cup R(f^N_\geo(\lambda^*))\sspc.$$
        Let $\alpha:[0,1]\longrightarrow\R^2$ be a continuous path contained in $B$ that joins the points $p$ and $q$. Since
        $$ p \in f^{M_N}_\geo(\lambda^*) \subset L(\lambda_N) \cap L(f^N_\geo(\lambda^*)) \quad\text{ and }\quad q \in R(\lambda_N) \cup R(f^N_\geo(\lambda^*)),$$
        we conclude that there exists $t\in [0,1]$ such that $\alpha(t)$ belongs to either $\lambda_N$ or $f^N_\geo(\lambda^*)$. Since $\lambda_N$ and $f^N_\geo(\lambda^*)$ are topological lines on the plane, this means that either $\lambda_N$ and $f^N_\geo(\lambda^*)$ must intersect the boundary of $B$ at some point. However, since $\partial B$ is formed by a sub-arc of $V_N^\geo$ and a sub-arc of $f^{M_N}_\geo(\lambda^*)$, items (i) and (ii) of Claim \ref{claim4} imply that $\partial B$ should be disjoint from both $\lambda_N$ and $f^N_\geo(\lambda^*)$. Thus, this contradicts case (b).

        We analyze case (c). In this case, let $\ell$ be a sub-halfine of $f^{M_N}_\geo(\lambda^*)$ contained in $R(V_N^\geo)$,  except for its endpoint that lies on $V_N^\geo$. Since $V_N^\geo$ is disjoint from $\lambda^*$, we can see that the set $L(\lambda^*) \cap R(V_N^\geo)\setminus \ell$ has two connected components, one bounded by $\lambda^*$, $V_N^\geo$ and $\ell$, and the other only bounded by $V_N^\geo$ and $\ell$. Let $B\subset \R^2$ be the second component. As in case (b), we know that $B$ must contain a point $q \in \OO$ in its interior, otherwise we would be able to perform an isotopy on $\R^2 \setminus \OO$ that reduces the intersection number between $f^{M_N}_\geo(\lambda^*)$ and $V_N^\geo$, thus contradicting the fact that they are geodesic lines in minimal position.
        Through a similar argument to the one in case (b), we can show that either $f^N_\geo(\lambda^*)$ or $\lambda_N$ are contained in $B$. However, this leads to a contradiction, as it would imply that both $\lambda^*$ and $f^{M_N}_\geo(\lambda^*)$ are contained in the same connected component of $\R^2\setminus f^N_\geo(\lambda^*)$ or $\R^2\setminus \lambda_N$, which is not true. This contradicts (c) and, thus, we proved the claim.
    \end{proof}

     \vspace*{-0cm}
    Observe that, if we repeat the same construction for $N+1$, it suffices to consider the geodesic leaf $\lambda_{N+1} \in \G$ far enough to the left to ensure that
    $$ T_{N+1}^+(i) > T_N^+(i), \quad \forall i \in I_\omega.$$
    Since $T_{N+1}^+(i) =N+1$ for all $i \in I_\text{out}$, we observe that this choice of geodesic leaf $\lambda_{N+1}$ in fact implies that $T_{N+1}^+(i) > T_N^+(i)$ for all $i \in I$. By constructing the piecewise linear line $V_{N+1}$  based on such choice of integers $\{T_{N+1}^+(i)\}_{i \in I}$ we ensure that $V_{N+1}$ satisfies
    $ L(V_{N+1}) \subset L(V_N).$
    This means that, by constructing such piecewise linear lines for all $N>0$, we can obtain a family of piecewise linear lines $\{V_N\}_{N>0}$ with nested left sides.
    
    Finally, observe that since $T_N^+(i) \geq N$ for all $i \in I$ and all $N>0$, we conclude that
     \mycomment{-0.1cm}
    $$ T_N^+(i) \longrightarrow \infty \quad \text{as} \quad  N \longrightarrow \infty\sspc.$$
    By the construction of the piecewise linear lines $V_N$, this implies that
    $ \bigcap_{N>0} L(V_N) = \varnothing\sspc.$

    \newpage
    \noindent This proves that the family $\{V_N\}_{N>0}$ is locally-finite and, therefore, the family of geodesic representatives $\{V_N^\geo\}_{N>0}$ is also locally-finite (see Section \ref{sec:hyperbolic_geometry_R2Z}), implying $ \bigcap_{N>0} L(V_N^\geo) = \varnothing\sspc.$

    We then conclude that the geodesic leaf $\lambda^*$ is forward-wandering, because
    $$ \bigcap_{N>0} L(f^{N}_\geo(\lambda^*)) = \bigcap_{N>0} L(f^{M_N}_\geo(\lambda^*)) \subset \bigcap_{N>0} L(V_N^\geo) = \varnothing\sspc,$$
    and this implies that the family $\{f^n_\geo(\lambda^*)\}_{n>0}$ is locally-finite.

    To finish the proof, we show that, since $\lambda^*$ is pushing-equivalent to $\lambda$, we have that $\lambda$ is also  forward-wandering. Indeed, according to the Pushing Lemma, there exists $K>0$ such that
    $$ L(f^{K}_\geo(\lambda)) \subset L(\lambda^*).$$
    This implies that
    $$ \bigcap_{n>K} L(f^{n}_\geo(\lambda))  = \bigcap_{n>0} L(f^{n}_\geo(\lambda^*)) = \varnothing\sspc.$$

    Therefore, the family $\{f^n_\geo(\lambda)\}_{n>0}$ is locally-finite, which concludes the proof.
\end{proof}

\section{Classifying (fine) Brouwer mapping classes}
\label{ch:applications_homotopy_brouwer_theory}

Let $\F$ be a transverse foliation of $f$, and let $\{\Gamma_1,...,\Gamma_r\}$ be a family of proper transverse trajectories associated to the orbits in $\OO = \{\O_{\sspc 1},..., \O_{\sspc r}\}$.
In this section, we show how the framework constructed in the previous chapters can be applied to recover Handel's classification of Brouwer mapping classes relative to at most two orbits, originally presented in \cite{handel99}.  

This is attained by studying the following two cases:

\vspace*{0.1cm}
\begin{itemize}[leftmargin = 1.4cm]
    \item[\textbf{(1)}] \textbf{Common-leaf case:} There exists a leaf $\phi \in \F$ crossed by every orbit in $\OO$.
    
    \vspace*{0.1cm}
    \item[\textbf{(2)}] \textbf{Totally separated case:} The orbits in $\OO$ are pairwise separated by some leaf of $\F$.
\end{itemize}

\vspace*{0.2cm}

\begin{figure}[h!]
    \center 
    \hspace*{0.2cm}\begin{overpic}[width=7cm, height=4cm, tics=10]{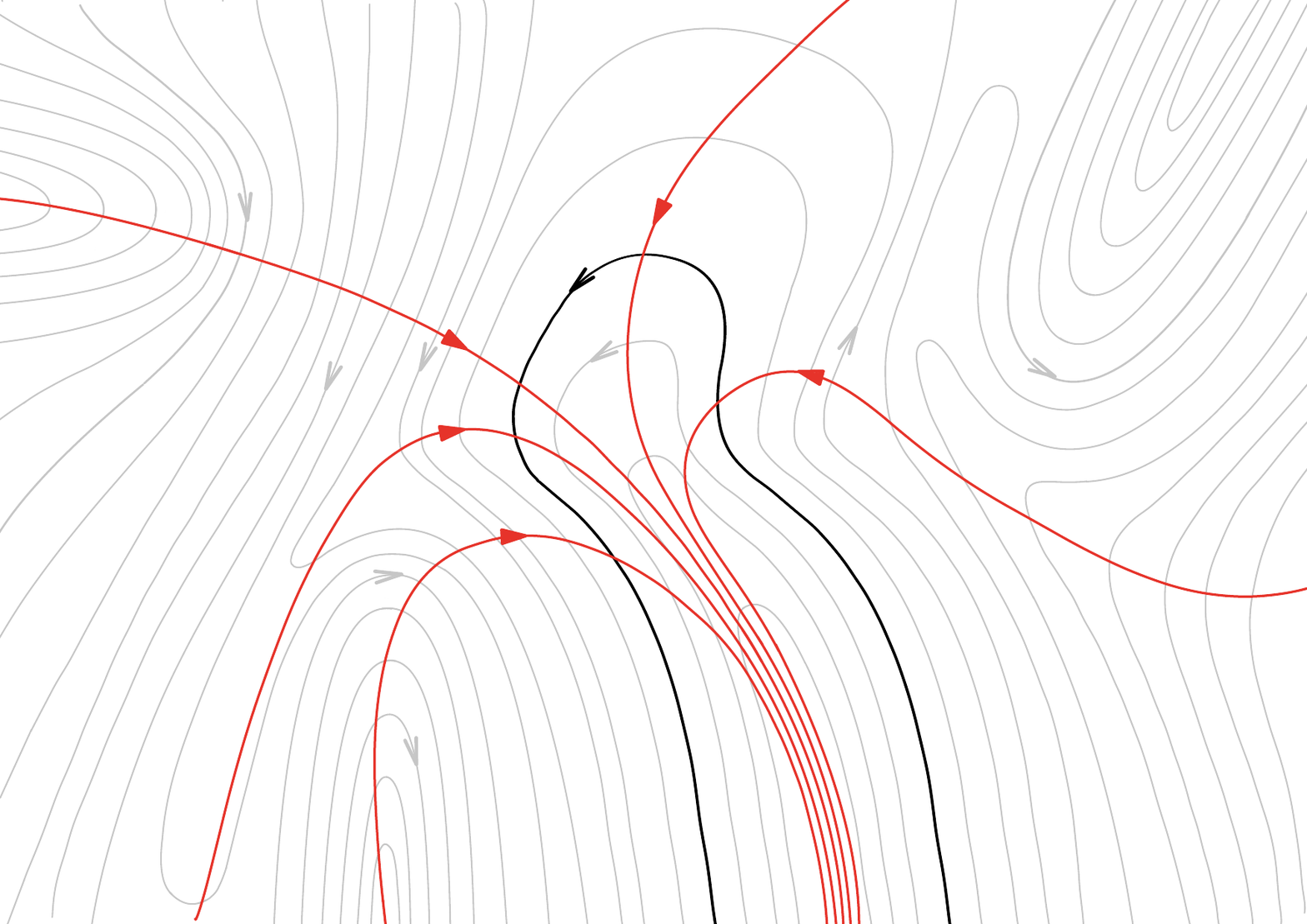}\put (-9,56) {\large\colorbox{white}{\color{black} \textbf{(1)}}}
\end{overpic}
\hspace*{1.2cm}\begin{overpic}[width=7cm, height=4cm, tics=10]{TotallySeparated.pdf}
    \put (-9,56) {\large\colorbox{white}{\color{black} \textbf{(2)}}}
\end{overpic}
\end{figure}

The following result describes (1) and implies Corollary \ref{cor:absolute_brouwer_classification-intro}, which improves Handel's classification in the case $r=1$ by allowing one to control fixed-point--free isotopies.

\begin{corollary}\label{cor:absolute_brouwer_classification_fbt}
    If there exists a leaf $\phi\in \F$ crossed by every orbit in $\OO$, then the fine Brouwer mapping class $\bclass{f,\OO}$ admits a representative $F\in\bclass{f,\OO}\sspc$ that is conjugate to $T$.
\end{corollary}

The following result describes (2), and together with \ref{cor:absolute_brouwer_classification_fbt}, completely recovers Handel's classification in the case $r=2$. This result aligns with an idea introduced by Handel, but here the concept of \textit{forward and backward proper homotopy streamlines} (see \cite{handel99}) is replaced with the more intuitive notion of proper transverse trajectories.

\begin{corollary}\label{corollary:separated_case}
     Assume that, for any $1 \leq i < j \leq r$ there exists a leaf in $\F$ that separates the orbits $\O_i$ and $\O_j$, then $\class{f,\OO}$ is a flow-class which admits a representative $\Phi_1 \in \class{f,\OO}$ which is the time one map of a flow $\{\Phi_t\}_{t \in \R}$ having $\Gamma_1\sspc,\sspc ... \sspc ,\sspc \Gamma_r$  as its integral trajectories.
\end{corollary}

\pagebreak

\begin{proof}[Proof Corollary \ref{corollary:separated_case}]
    For each $1\leq i \leq r$, we write
    $$\Gamma_i = \prod_{n \in \Z} \gamma_i^n\sspc,$$
    where $\gamma_i^n$ is a path positively transverse to $\F$ that connects the points $f^n(x_i)$ and $f^{n+1}(x_i)$, for some initial basepoint $x_i \in \O_i$. Since the transverse trajectories in the family $\{\Gamma_1,...,\Gamma_r\}$ are proper and pairwise disjoint, we have that $\{\gamma_i^n\}_{i, n}$ is a locally-finite family of pairwise disjoint lines on $\R^2\setminus \OO$ and, since $f$ is a homeomorphism, the same applies to $\{f(\gamma_i^n)\}_{i, n}$.

    According to Proposition \ref{prop:separated_geodesic_trajectories}, in this particular case, for any $1\leq i \leq r$ and any $n\in \Z$, the path $f(\gamma_{i}^n)$ is isotopic to $\gamma_{i}^{n+1}$ on the surface $\R^2\setminus \OO$. Thus, using the Straightening Principle (see Section \ref{sec:hyperbolic_geometry_R2Z}), there exists an isotopy $(h_t)_{t\in [0,1]}$  on the surface $\R^2\setminus \OO$ such that $h_0 = \text{id}_{\R^2\setminus \OO}$ and $h_1$ satisfies $h(f^n(\gamma_i^n)) = \gamma_i^{n+1}$ for all $i\in \{1,...,r\}$ and $n\in \Z$. Defining  $f_t:=h_t\circ f\vert_{\R^2\setminus \OO}$, we obtain an isotopy $(f_t)_{t\in [0,1]}$ on $\R^2\setminus \OO$ joining $f_0 = f\vert_{\R^2\setminus \OO}$ to a map $f_1=:f^\pp $ that satisfies
    $$ f^\pp(\gamma_i^n) = \gamma_i^{n+1}\sspc, \quad \forall i\in \{1,...,r\}, \ n\in \Z.$$

    By using Alexander's trick on each connected component of $\R^2\setminus \bigcup_{i=1}^r \Gamma_i$, we can obtain an isotopy $(F_t)_{t\in [0,1]}$ on $\R^2\setminus \OO$ relative to $\bigcup_{i=1}^r \Gamma_i$ joining $F_0 = f^\pp$ to  $F_1 = \Phi_1\vert_{\R^2\setminus\OO}$, which is the time one map of a flow $\{\Phi_t\}_{t \in \R}$ having $\Gamma_1\sspc,\sspc ... \sspc ,\sspc \Gamma_r$  as its integral trajectories.
\end{proof}

\begin{proof}[Proof of Corollary \ref{cor:absolute_brouwer_classification_fbt}]
    In this case, we can take an essential leaf $\phi\in \F$ that is crossed by every orbit in $\OO$. According to Theorem \ref{thmx:II-C}, the geodesic representative $\phi^\geo$ of $\phi$ is both forward-wandering and backward-wandering. This allows us to apply Proposition \ref{lemma:simplification_wandering_leaves} on the left and right sides of $\phi^\geo$ to obtain an isotopy $(f_t)_{t\in [0,1]}$ on $\R^2\setminus \OO$, where for every $t\in [0,1]$ the extension $\overline{\rule{0cm}{0.32cm}f_t\sspc}$ is a Brouwer homeomorphism, that joins $f_0 = f\vert_{\R^2\setminus \OO}$ to a homeomorphism $f_1 =: f^\pp$ admitting a transverse foliation $\F^\pp$ that is trivial and $f^\pp$-invariant. Therefore, the homeomorphism $\overline{\rule{0cm}{0.32cm}f^\pp}$ must be conjugate to the translation $T:(x,y) \longmapsto (x+1,y)$.
\end{proof}

\begin{remark}
    To observe that this result recovers Handel's original classification from \cite{handel99}, it suffices to note that in the case $r=1$ the conditions of Corollary \ref{cor:absolute_brouwer_classification_fbt} are trivially satisfied, implying that $\bclass{f,\OO}$ is conjugate to $\bclass{T,\spc\Z_1}$, in particular, $\class{f,\OO}$ is conjugate to $\class{T,\spc\Z_1}$.

    Meanwhile, in the case $r=2$, the cases (1) and (2) described in the beginning of this section are the only two possible configurations. In the case (1), the same argument from the case $r=1$ concludes the argument. In the case (2), the conditions of Corollary \ref{corollary:separated_case} are satisfied, implying that $\class{f,\OO}$ is conjugate to either $\class{T,\spc\Z_{1,2}}$, $\class{T^{-1},\spc \Z_{1,2}}$, $\class{R\sspc ,\spc \Z_{1,2}}$ or  $\class{R^{-1},\spc \Z_{1,2}}$. 
\end{remark}

\setstretch{1.6}	

    \bibliography{biblio1.bib}

\end{document}